\documentclass[10pt]{amsart}
\usepackage{amsmath}
\usepackage{amsthm}
\usepackage{amsbsy}
\usepackage{amsopn}
\usepackage{amsmath,amscd}
\usepackage{amssymb}
\usepackage{amsfonts}
\usepackage{multirow}
\usepackage[official ]{eurosym}

\newcommand{\ve}{\varepsilon}

\newcommand{\cA}{\mathcal{A}}
\newcommand{\cC}{\mathcal{C}}
\newcommand{\cF}{\mathcal{F}}
\newcommand{\cH}{\mathcal{H}}
\newcommand{\cI}{\mathcal{I}}

\newcommand{\cM}{\mathcal{M}}
\newcommand{\cP}{\mathcal{P}}

\newcommand{\cS}{\mathcal{S}}

\newcommand{\cZ}{\mathcal{Z}}

\newcommand{\bC}{\mathbb{C}}

\newcommand{\bQ}{\mathbb{Q}}\newcommand{\bR}{\mathbb{R}}

\newcommand{\bZ}{\mathbb{Z}}

\usepackage{graphicx}

  \newcommand{\mat}[1]{\ensuremath{
\left[\begin{matrix}#1
\end{matrix}\right]
}}

\usepackage[dvipsnames,svgnames,x11names,hyperref]{xcolor}
\usepackage{url,graphicx,verbatim,amssymb,enumerate,stmaryrd}
\usepackage[pagebackref,colorlinks,citecolor=Bittersweet,linkcolor=Bittersweet,urlcolor=Bittersweet,filecolor=Bittersweet]{hyperref}
\usepackage{tikz} 
\usetikzlibrary{arrows,decorations.pathmorphing,backgrounds,fit,positioning,shapes.symbols,chains,calc}
\usepackage[all]{xy}
\xyoption{matrix}
\xyoption{arrow}
\usepackage[hmargin=3cm,vmargin=3cm]{geometry}
\usepackage{setspace,kantlipsum}
\tikzset{help lines/.style={step=#1cm,very thin, color=gray},
help lines/.default=.5} 

\usepackage{booktabs}

\newtheorem{theorem}{Theorem}[section]
\newtheorem{lemma}[theorem]{Lemma}
\newtheorem{proposition}[theorem]{Proposition}
\newtheorem{corollary}[theorem]{Corollary}

\theoremstyle{definition}
\newtheorem{definition}[theorem]{Definition}

\newtheorem{inductive lemma}[theorem]{Inductive Lemma}

\theoremstyle{remark}
\newtheorem{problem}[theorem]{Problem}
\newtheorem{remark}[theorem]{Remark}
\newtheorem{example}[theorem]{Example}
\newtheorem{question}[theorem]{Question}
\newtheorem{warning}[theorem]{Warning}

\title{A Legendrian Turaev torsion via generating families}
\dedicatory{To our teachers Yasha Eliashberg and Allen Hatcher}

\author{Daniel  \'Alvarez-Gavela}
\address[D. \'Alvarez-Gavela]{Department of Mathematics, Princeton University, Princeton, 086540, NJ, USA}
\thanks{DAG was supported by NSF Grant No. DMS-1638352 and by the Simons Foundation.}
 \author{Kiyoshi Igusa}
 \address[K. Igusa]{Department of Mathematics, Brandeis University, PO Box 9110, Waltham, MA, 02454-9110, USA}
\thanks{KI is supported by the Simons Foundation.}

\begin{document}



\begin{abstract}
We introduce a Legendrian invariant built out of the Turaev torsion of generating families. This invariant is defined for a certain class of Legendrian submanifolds of 1-jet spaces, which we call of Euler type. We use our invariant to study mesh Legendrians: a family of 2-dimensional Euler type Legendrian links whose linking pattern is determined by a bicolored trivalent ribbon graph. The Turaev torsion of mesh Legendrians is related to a certain monodromy of handle slides, which we compute in terms of the combinatorics of the graph. As an application, we exhibit pairs of Legendrian links in the 1-jet space of any orientable closed surface which are formally equivalent, cannot be distinguished by any natural Legendrian invariant, yet are not Legendrian isotopic. These examples appeared in a different guise in the work of the second author with J. Klein on pictures for $K_3$ and the higher Reidemeister torsion of circle bundles.
 \end{abstract}
 
 \maketitle


 
 \onehalfspacing
 \tableofcontents
 
\section{Introduction}\label{Section: introduction}

\subsection{Main results}\label{Subsection: main results}
 
 In this article we use Turaev torsion, a refinement of Reidemeister torsion, to define an invariant of a certain.class of Legendrians in 1-jet spaces. We call this invariant {\it Legendrian Turaev torsion}. As an application, we exhibit peculiar pairs of Legendrian links in the 1-jet space of any closed orientable surface. Indeed, from our structural results on mesh Legendrians \ref{euler}, \ref{reidemeister}, \ref{calculation} and \ref{turaev} we deduce the following.
 
 \begin{corollary}\label{corollarypairs}
For any closed orientable surface $\Sigma$ there exist infinitely many distinct pairs of Legendrian links $\Lambda_\pm \subset J^1(\Sigma)$ such that
\begin{itemize}
\item[(a)] $\Lambda_+$ and $\Lambda_-$ are equivalent as formal Legendrian links.
\item[(b)] $\Lambda_+$ cannot be distinguished from $\Lambda_-$ by any natural Legendrian invariant.
\item[(c)] $\Lambda_+$ is not Legendrian isotopic to $\Lambda_-$.
\end{itemize}
\end{corollary}

\begin{remark} The pairs $\Lambda_\pm$ are all links of 2-dimensional Legendrian spheres. Each pair is obtained using the same underlying trivalent graph $G \subset \Sigma$, but one has all vertices colored positive and the other all negative. The precise construction will be given below in this introduction, after we introduce mesh Legendrians. For the time being we make some remarks about the stated properties.\begin{itemize} 
\item[(a)] In particular $\Lambda_+$ and $\Lambda_-$ are smoothly isotopic, but since the $\Lambda_{\pm}$ are links of 2-dimensional spheres  this is equivalent to the a priori stronger condition of being formally isotopic. In fact, for our pairs both $\Lambda_+$ and $\Lambda_-$ are formal unlinks and this is easy to verify explicitly. 
\item[(b)] By construction there is a diffeomorphism $\phi: \Sigma \to \Sigma$ (not isotopic to the identity) whose 1-jet lift $\Phi:J^1(\Sigma) \to J^1(\Sigma)$ sends $\Lambda_+$ to $\Lambda_-$. Hence $\Lambda_+$ cannot be distinguished from $\Lambda_-$ using any Legendrian invariant which is natural in the sense that it is preserved by strict contactomorphisms of the form $\Phi=j^1(\phi)$. See Section \ref{naturality} for further discussion of naturality.
\item [(c)] We distinguish $\Lambda_+$ from $\Lambda_-$ by showing that their Legendrian Turaev torsions are not equal. It follows that the Legendrian Turaev torsion is not a natural invariant.\end{itemize}
\end{remark}

Before discussing Legendrian Turaev torsion we recall the generating family construction. Consider the 1-jet space $J^1(B)=T^*B \times \bR$ of a closed manifold $B$ and the front projection $\pi: J^1(B)  \to J^0(B)$, where $J^0(B)=B \times \bR$. This is the product of the cotangent bundle projection $T^*B \to B$ and the identity on the $\bR$ factor. Let $F \to W \to B$ be a fibre bundle of closed, connected, orientable manifolds and denote by $F_b$ the fibre over $b \in B$. We view a function $f:W \to \bR$ as a family of functions $f_b:F_b \to \bR$. 

\begin{definition} The Cerf diagram of $f$ is the subset $\Sigma_f=\{ (b,z) :  \text{$z$ is a critical value of $f_b$} \} \subset B \times \bR$. 
\end{definition} 

Consider the graph of the differential $\Gamma(df) \subset T^*W$, which is a (graphical) Lagrangian submanifold. Suppose that the fibrewise derivative of $f$ satisfies the generic condition $\partial_F f \pitchfork 0 $. Then $ \{ \partial_F f = 0 \} \subset \Gamma(df)$ is an isotropic submanifold contained in the coisotropic subbundle $E \subset T^*W$ of covectors with zero fibrewise derivative. The symplectic reduction $E \to T^*B$ restricts to a Lagrangian immersion $\{ \partial_F f= 0 \} \to T^*B$, which in turn lifts to a Legendrian immersion $\{ \partial_Ff  = 0 \} \to J^1(B)$ via the function $f$ itself. Generically this Legendrian immersion is an embedding.

 
\begin{definition} We say that a Legendrian submanifold $\Lambda \subset J^1(B)$ is generated by $f:W \to \bR$ when $\partial_F f \pitchfork 0 $ and $\{ \partial_F f = 0 \} \to J^1 (B)$ is an embedding with image $\Lambda$. 
\end{definition}

\begin{remark} 
Note that in particular the Cerf diagram of $f$ is the front of $\Lambda$, i.e. $\Sigma_f=\pi(\Lambda)$. 
\end{remark}

Let $R$ be a commutative ring and $U(R)$ its group of units. Let $\rho: \pi_1W \to U(R)$ be a representation. Denote by $H^*(F;R^\rho)$ the cohomology of $F$ with the twisted $R$ coefficients given by the restriction of $\rho$ to $\pi_1F$ and assume that $H^*(F;R^\rho)$ is trivial. Then we can consider the Reidemeister torsion of $F$ with respect to $\rho$, which is an element of $U(R) / \pm \rho(\pi_1F)$. Following Turaev, by choosing an Euler structure on $F$ it is possible to lift the Reidemeister torsion to a finer invariant, the Turaev torsion, which is an element of $U(R) / \pm 1$. 

In this article we show that the global geometry of 
$\Lambda$ can sometimes be used to distinguish a preferred class of Euler structures on $F$. More precisely, we consider the following situation.

\begin{definition}
A Legendrian submanifold $\Lambda \subset J^1(B)$ is said to be of Euler type if each connected component $\Lambda_i$ of $\Lambda$ is simply connected and the projection to the base $\Lambda_i \to B$ has degree zero.
\end{definition}

\begin{definition}
A torsion pair consists of a fibre bundle $F \to W \to B$ of closed, connected, orientable manifolds and a representation $\rho:\pi_1W \to U(R)$ such that $H^*(F;R^\rho)=0$ and $\rho(\pi_1F)=\rho(\pi_1W)$.
\end{definition}

If $\Lambda \subset J^1(B)$ is an Euler Legendrian and $(W,\rho)$ is a torsion pair, then to each generating family $f$ for $\Lambda$ on an even stabilization of $W$ we will assign a certain Turaev torsion of $F$ by computing its $\rho$-twisted cohomology using the Morse theory of $f|_F$. We denote by $T(\Lambda, W, \rho)$ the resulting set of Turaev torsions for varying $f$ but fixed $\Lambda$, $W$ and $\rho$. 

\begin{definition} The subset $T(\Lambda, W, \rho) \subset U(R) / \pm1$ is called the Legendrian Turaev torsion. \end{definition}

Each element of $T(\Lambda,W,\rho)$ maps to the Reidemeister torsion of $F$ with respect to $\rho$ under the natural homomorphism $U(R) / \pm 1 \to U(R) / \pm \rho(\pi_1 F)$. So the Legendrian $\Lambda$ selects certain lifts of the Reidemeister torsion, which we assemble into a Legendrian invariant $T(\Lambda, W, \rho)$.  The invariance property is the following.
\begin{theorem} $T(\Lambda_1,W ,\rho)=T(\Lambda_2, W, \rho)$ whenever $\Lambda_1$ and $\Lambda_2$ are Legendrian isotopic.
\end{theorem}
We prove this theorem using the homotopy lifting property for generating families. More generally, Legendrian Turaev torsion exhibits functorial behavior with respect to a certain class of Legendrian cobordisms, see Section \ref{sec: invariant}. We use our Legendrian Turaev torsion to study a family of Euler type Legendrians called mesh Legendrians, which we introduce next.

\begin{definition} A ribbon graph $G$ is a finite connected graph where each vertex is equipped with a cyclic ordering of the half-edges incident to the vertex.
\end{definition}

Every ribbon graph $G$ can be fattened to an oriented surface with boundary $S_G$, replacing vertices by disks and replacing edges by thin rectangles, which are pasted to the disks according to the corresponding cyclic orderings. Let $\Sigma_G$ be the closed oriented surface obtained by attaching a 2-cell to each boundary component of $S_G$, see Figure \ref{FigureFatgraph}. By construction $G$ comes equipped with an embedding $G \subset \Sigma_G$ such that the orientation of $\Sigma_G$ determines the ribbon structure of $G$.

\begin{figure}[htbp]
\begin{center}
\begin{tikzpicture}[scale=.75]
\clip (-8,-1)rectangle(9,4);
\begin{scope}[xshift=-6cm]
\coordinate (A) at (0,1);
\coordinate (B) at (1,2);
\coordinate (C) at (2.5,1.5);
\coordinate (E) at (2,-2);
\coordinate (D) at (5,2);
\coordinate (F) at (-2,3);
\coordinate (G) at (1,4);
\coordinate (X) at (1.03,.05);
\coordinate (Y) at (-2.5,-2);
\coordinate (Z) at (2,0);
\coordinate (W) at (-1.5,3);
\draw (W) node{$G:$};
\draw[fill] (A) circle[radius=3pt];
\draw[fill] (B) circle[radius=3pt];
\draw[fill] (C) circle[radius=3pt];
\draw[very thick] (A)..controls (Y) and (Z)..(C);
\draw[fill,color=white] (X) circle[radius=3pt];
\draw[very thick] (A)--(B)--(C)..controls (D) and (E) ..(A)..controls (F) and (G)..(B);
\end{scope}
\begin{scope}
\coordinate (A) at (0,1);
\coordinate (B) at (1,2);
\coordinate (C) at (2.5,1.5);
\coordinate (A1) at (-.3,1);
\coordinate (A2) at (0,1.3);
\coordinate (A3) at (0.3,1);
\coordinate (A4) at (0,.7);
\coordinate (B1) at (.7,2);
\coordinate (B2) at (1,2.3);
\coordinate (B3) at (1.3,2.1);
\coordinate (B4) at (1,1.7);
\coordinate (C1) at (2.2,1.3);
\coordinate (C2) at (2.5,1.7);
\coordinate (C3) at (2.8,1.7);
\coordinate (C4) at (2.7,1.2);
\coordinate (E) at (2.5,-2.5);
\coordinate (D) at (6,2);
\coordinate (E1) at (2,-1.6);
\coordinate (D1) at (5.3,2);
\coordinate (Y) at (-2.5,-2);
\coordinate (Z) at (2,0);
\coordinate (Y1) at (-2,-1.5);
\coordinate (Z1) at (2,0.5);
\coordinate (X1) at (.8,.2);
\coordinate (X2) at (1.25,-.2);
\draw[thick] (A2)--(B1) (A3)--(B4) (B3)--(C2) (B4)--(C1);
\draw[thick] (A1)..controls (-3,3) and (.8,5)..(B2);
\draw[thick] (A2)..controls (-2,3) and (.5,4)..(B1);
\draw[thick] (A1)..controls (Y) and (Z)..(C4);
\draw[thick] (A4)..controls (Y1) and (Z1)..(C1);
\draw[fill,color=white] (X1) circle[radius=.35cm];
\draw[fill,color=white] (X2) circle[radius=.4cm];
\draw[thick] (A4)..controls (E) and (D)..(C3); 
\draw[thick] (A3)..controls (E1) and (D1)..(C4); 
\draw[fill,color=white] (A)circle[radius=.4cm];
\draw[fill,color=white] (B)circle[radius=.4cm];
\draw[fill,color=white] (C)circle[radius=.4cm];
\draw (A)circle[radius=.4cm];
\draw (B)circle[radius=.4cm];
\draw (C)circle[radius=.4cm];
\coordinate (W) at (-2,3);
\draw (W) node{$S_G:$};
\end{scope}
\begin{scope}[xshift=6cm,yshift=.5cm]
\coordinate (W) at (-1,2.5);
\draw (W) node{$\Sigma_G:$};
\coordinate (A) at (1,1);
\coordinate (B) at (1,0);
\draw[thick] (A) ellipse[x radius=1.5cm,y radius=.8cm];
\begin{scope}
\clip  (1,4) circle[radius=3cm];
\draw[thick] (1,.3)circle[radius=.9cm];
\end{scope}
\clip (B) circle[radius=1.4cm];
\draw[thick] (1,4) circle[radius=3cm];
\end{scope}
\end{tikzpicture}
\caption{A ribbon graph ($G$, left) can be fattened into an oriented surface with boundary ($S_G$, middle) in an essentially unique way. The closed surface ($\Sigma_G$, right) is given by attaching 2-cells to each boundary component of $S_G$. In this case $\Sigma_G$ has genus zero because $\chi(\Sigma_G)=\chi(G)+2=3-5+2=0$ since two $2$ cells were added.
}
\label{FigureFatgraph}
\end{center}
\end{figure}
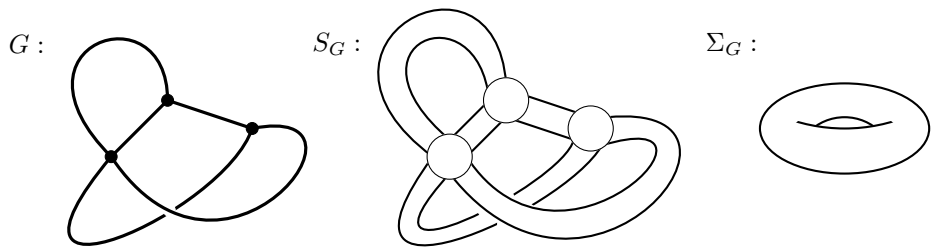

Suppose now that $G$ is a ribbon graph which is trivalent and bicolored, so each vertex only has three half-edges incident to it and is labeled with a decoration by the symbol $+$ or $-$. We emphasize that the labels are arbitrary; in particular we do not assume that the bicolored graph is bipartite. To each such $G$ we associate a Legendrian submanifold $\Lambda_G \subset J^1(\Sigma_G)$ in the following way. 

\textbf{(1) Faces.} For each face $F_j$ of $\Sigma_G \setminus G$ we have a connected component $\Lambda^j_G$ of $\Lambda_G$, which is a standard Legendrian unknot, i.e. a flying saucer in the front projection. Explicitly, the front $\pi(\Lambda^j_G) \subset J^0(\Sigma_G)=\Sigma_G \times \bR$ consists of two parallel copies of $F_j$ which meet at cusps along a copy of $\partial F_j$ which is slightly pushed out along the outwards pointing normal to $\partial F_j$, see Figure \ref{FigureLens}. 
Note that the projections of the $\Lambda^j_G$ to $\Sigma_G$ overlap, and indeed cover $\Sigma_G$ with their interiors.

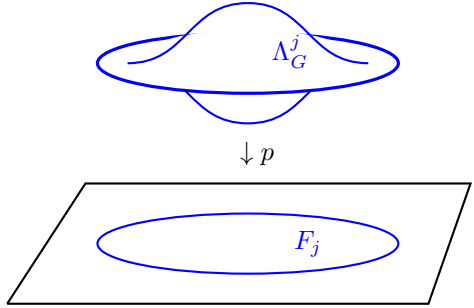
\begin{figure}[htbp]
\begin{center}
\begin{tikzpicture}[scale=.8]
\begin{scope}
\draw[thick,color=blue] (-2,0)..controls (-1,0) and (-1,-1)..(0,-1); 
\draw[thick,color=blue] (2,0)..controls (1,0) and (1,-1)..(0,-1);
\draw[fill,color=white] (0,0) ellipse [x radius=2.5cm,y radius=.5cm];
\end{scope}
\begin{scope}
\draw[very thick,color=blue] (0,0) ellipse [x radius=2.5cm,y radius=.5cm];
\draw[fill,color=white] (0,0.05) ellipse[x radius=1.8cm, y radius=.5cm];
\end{scope}
\draw[thick,color=blue] (-2,0)..controls (-1,0) and (-1,1)..(0,1); 
\draw[thick,color=blue] (2,0)..controls (1,0) and (1,1)..(0,1);
\draw[color=blue] (0.7,0.2) node{$\Lambda_G^j$};
\draw (0.15,-1.5) node{$\downarrow p$};
\begin{scope}[yshift=-4cm]
\draw[thick,color=blue] (0,1) ellipse [x radius=2.5cm,y radius=.5cm] (1,1) node{$F_j$};
 \draw[thick] (-4,0)--(3,0)--(3.7,2)--(-2.7,2)--(-4,0);
\end{scope}
\end{tikzpicture}
\caption{Front projection of a standard Legendrian unknot.}
\label{FigureLens}
\end{center}
\end{figure}

\textbf{(2) Edges.} Along each edge $E$ of $G$ the Legendrians $\Lambda^{i}_G$ and $\Lambda^{j}_G$ corresponding to the faces $F_i$ and $F_j$ whose boundary contains $E$ are linked; as in the 1-dimensional Legendrian clasp but multiplied by a trivial factor in the $E$ direction, see Figure \ref{FigureE}. Note that a face could meet itself along an edge (i.e. $F_i=F_j$), and this is allowed. 

\textbf{(3) Vertices.} Finally, at each vertex $V$ of $G$ the cusp loci of the front projections of the $\Lambda^F_G$ spiral around each other with a chirality that depends on the sign $+$ or $-$ of the decoration, see Figure \ref{FigureT}. 

\begin{remark} Hence $\Lambda_G$ is always a link of unknotted Legendrian spheres. The front of each sphere has cusps along its equator and has no other caustics.  \end{remark}

\begin{definition}
We call $\Lambda_G \subset J^1(\Sigma_G)$ the mesh Legendrian associated to $G$.
\end{definition}

\begin{example}
See Figure \ref{Fig: start} for an illustration of the mesh Legendrian $\Lambda_G$ corresponding to the bicolored trivalent ribbon graph $G$ shown in Figure \ref{Fig: 11}.
\end{example}

%
\begin{figure}[htbp]
\begin{center}
\begin{tikzpicture}[scale=.75]
\coordinate (E) at (0.3,1.1);
\coordinate (A1) at (-1.8,1);
\coordinate (A2) at (-1.6,5.2);
\coordinate (A3) at (-1.9,3);
\coordinate (B1) at (3.1,1);
\coordinate (B2) at (3,3.5);
\coordinate (C) at (0.2,2.5);
\draw (C) node{$\downarrow p$};
\draw[color=gray!50!black] (E) node{$E$};
\draw (A1) node{$F_i$};
\draw (A2) node{$\Lambda_G^i$};
\draw[color=blue] (B1) node{$F_j$};
\draw[color=blue] (B2) node{$\Lambda_G^j$};
\begin{scope}
	\draw[very thick] (1.8,0)--(3,2);
	\draw[thick] (-3,0)--(1.8,0)--(3,2)--(-1.4,2)--(-3,0);
	\draw[very thick, color=blue] (-1.8,-.05)--(-.3,1.95);
	\draw[thick,color=gray!70!black,dotted] (0,0)--(1.35,2);
	\draw[thick,color=blue] (-1.8,-.05)--(3,-.05)--(4.1,1.95)--(-.3,1.95);
\end{scope}
\begin{scope}[yshift=4.2cm,xshift=1.3cm] 
\clip (-3,1.42)rectangle (3,2);
	\draw[thick] (-2.7,1.9) ..controls (-1,2) and (0.75,1)..(1.7,1);
	\draw[thick,color=blue] (2.8,1.9) ..controls (1,2) and (-.75,1)..(-1.6,1);
\end{scope}
\begin{scope}[yshift=4.2cm,xshift=1.3cm] 
\clip (1.2,0)rectangle (2.8,.4);
	\draw[thick,color=blue] (2.8,0.1) ..controls (1,0) and (-.75,1)..(-1.6,1);
\end{scope}
\begin{scope}[yshift=4.2cm,xshift=1.3cm]
	\draw[thin, dashed] (-2.7,.1) ..controls (-1,0) and (0.75,1)..(1.7,1);
	\draw[thin, dashed] (-2.7,1.9) ..controls (-1,2) and (0.75,1)..(1.7,1);
	\draw[thin, dashed,color=blue] (2.8,0.1) ..controls (1,0) and (-.75,1)..(-1.6,1);
	\draw[thin, dashed,color=blue] (2.8,1.9) ..controls (1,2) and (-.75,1)..(-1.6,1);
\end{scope}
\begin{scope}[yshift=2.5cm] 
	\draw[thick] (-1.4,1.8)--(-3,0) ..controls (-1,0) and (0.75,1)..(1.75,1);
	\draw[thick] (-1.4,3.6)--(-3,2) ..controls (-1,2) and (0.75,1)..(1.75,1);
	\draw[very thick] (1.75,1)--(2.47,1.98);
	\draw[dashed] (1.75,1)--(3,2.7);
	\draw[thick,color=blue] (4.1,1.8)--(3,0) ..controls (1,0) and (-.75,1)..(-1.75,1);
	\draw[thick,color=blue] (4.1,3.6)--(3,2) ..controls (1,2) and (-.75,1)..(-1.75,1);
	\draw[very thick, color=blue] (-1.75,1)--(-1.13,1.74);
	\draw[dashed,color=blue] (-1.75,1)--(-.3,2.7);
	\draw[thick,color=gray!70!black,dotted] (0,1.41)--(1.35,3.11);
	\draw[thick,color=gray!70!black,dotted] (0.02,.64)--(0.45,1.2);
	\draw[color=gray!70!black,dotted] (0.5,1.26)--(1.35,2.31);
\end{scope}
\end{tikzpicture}
\caption{Two components $\Lambda_G^i$, $\Lambda_G^j$ are linked over the edge $E$.}
\label{FigureE}
\end{center}
\end{figure}
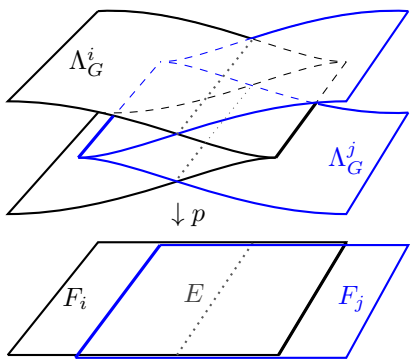
%

%
\begin{figure}[htbp]
\begin{center}
\begin{tikzpicture}[yscale=.75]
\coordinate (A) at (-1,6.8);
\coordinate (B) at (1,6.8);
\coordinate (C) at (-.6,4.8);
\coordinate (D) at (.45,3.5);
\coordinate (P) at (-1.4,3.3);
\coordinate (W1) at (-0.02,4.51);
\coordinate (W2) at (0.09,4.52);
\draw (A) node{$\Lambda_G^i$};
\draw[color=blue] (B) node{$\Lambda_G^j$};
\draw[color=green!70!black] (C) node{$\Lambda_G^k$};
\draw[color=green!70!black] (D) node{$\Lambda_G^k$};
\draw (P) node{$p\downarrow $};
\coordinate (K0) at (0,0);
\coordinate (I0) at (-2,2.4);
\coordinate (J0) at (2,2.4);
\coordinate (K1) at (.5,0);
\coordinate (K2) at (0,.5);
\coordinate (K3) at (-.5,0);
\coordinate (KI) at (-1.25,.875);
\coordinate (I1) at (-2,1.75);
\coordinate (I2) at (-1.5,2.15);
\coordinate (I3) at (-1.5,2.75);
\coordinate (IJ) at (0,2.75);
\coordinate (J1) at (1.5,2.75);
\coordinate (J2) at (1.5,2.15);
\coordinate (J3) at (2,1.75);
\coordinate (JK) at (1.25,.875);
\coordinate (C) at (0,1.75);
\draw[fill,color=gray] (C) circle[radius=1pt];
\draw[dotted,thick,color=gray] (KI)--(C)--(IJ) (C)--(JK);
\draw (I0) node{$F_i$};
\draw[color=blue] (J0) node{$F_j$};
\draw[color=green!70!black] (K0) node{$F_k$};
\coordinate (B1) at (-.75,.29);
\coordinate (Ba) at (.25,1.29);
\coordinate (Bb) at (1,2);
\coordinate (B2) at (1,2.75);
\begin{scope}
\clip (-1,0)rectangle (0.25,1.5);
\draw[very thick] (B1)..controls (Ba) and (Bb).. (B2);
\end{scope}
\coordinate (L1) at (.75,.29);
\coordinate (La) at (-.25,1.29);
\coordinate (Lb) at (-1,2);
\coordinate (L2) at (-1,2.75);
\draw[very thick,color=blue] (L1)..controls (La) and (Lb).. (L2);  
\coordinate (G1) at (-1.75,1.45);
\coordinate (Ga) at (-.5,2.5);
\coordinate (Gb) at (.5,2.5);
\coordinate (G2) at (1.75,1.45);
\draw[very thick,color=green!80!black] (G1)..controls (Ga) and (Gb).. (G2); 
\begin{scope}
\clip (0.2,1)rectangle (1.2,2.75);
\draw[very thick] (B1)..controls (Ba) and (Bb).. (B2);
\end{scope}
\begin{scope} 
\draw[ thick,color=green!80!black] (JK)--(K1)--(K2)--(K3)--(KI);
\draw[ thick] (KI)--(I1)--(I2)--(I3)--(IJ);
\draw[ thick, color=blue] (IJ)--(J1)--(J2)--(J3)--(JK);
\end{scope}
\coordinate (C1) at (0,5.7);
\coordinate (C2) at (0,5);
\coordinate (IK1) at (-1.22,4.84);
\coordinate (IK2) at (-1.17,4);
\coordinate (IJ1) at (0,6.93);
\coordinate (JK1) at (1.2,4.825);
\coordinate (JK2) at (1.2,4.05);
\draw[fill,color=gray] (C1) circle[radius=1pt];
\draw[thick,color=gray,dotted] (C1)--(IJ1) (IK1)--(C1)--(JK1);
\begin{scope}
\clip (JK2) circle[radius=.3cm];
\draw[thick,color=gray,dotted] (IK2)--(C2)--(JK2);
\end{scope}
\begin{scope}
\clip (IK2) circle[radius=.25cm];
\draw[thick,color=gray,dotted] (IK2)--(C2)--(JK2);
\end{scope}
%
\begin{scope}[yshift=3.7cm]
\draw[very thin,dashed,color=green!70!black] (-1.75,1.45)..controls (-.5,2.5) and (.5,2.5).. (1.75,1.45); 
\clip (-1.75,1.45) circle[radius=.37cm];
\draw[very thick,color=green!70!black] (-1.75,1.45)..controls (-.5,2.5) and (.5,2.5).. (1.75,1.45); 
\end{scope}
\begin{scope}[yshift=3.69cm]
\clip (1.75,1.45) circle[radius=.38cm];
\draw[very thick,color=green!70!black] (-1.75,1.45)..controls (-.5,2.5) and (.5,2.5).. (1.75,1.45); 
\end{scope}
\begin{scope}[yshift=3cm] 
\draw[thick] (-2,1.75).. controls (-1.25,.875) and (-1,1)..(-.75,.69);
\draw[thick] (-2,3.25).. controls (-1.25,2.375) and (-1,1)..(-.75,.69);
\begin{scope}
\draw[very thin,dashed] (-.76,.69)..controls (.25,1.69) and (1,2.5).. (1,3.25);
\clip (-.4,1) circle[radius=.77cm];
\draw[very thick] (-.76,.69)..controls (.25,1.69) and (1,2.5).. (1,3.25);
\end{scope}
\begin{scope}[rotate=45]
\draw[fill,color=white] (W1) rectangle (W2);
\end{scope}
\draw[thick] (-2,1.75)--(-1.65,2.02);
\draw[thick] (-2,3.25)--(-1.5,3.65)--(-1.5,4.25);
\draw[very thin,dashed] (-1.5,4.25)..controls (0,4.25)and(.25,3.7)..(1,3.25);
\clip (-1.6,3.93) rectangle (0,4.3);
\draw[thick] (-1.5,4.25)..controls (0,4.25)and(.25,3.7)..(1,3.25);
\end{scope}
\begin{scope}[yshift=4.5cm] 
\draw[thick,color=blue] (1.5,2.75)--(1.5,2.15)--(2,1.75);
\end{scope}
\begin{scope}[yshift=3cm] 
\draw[thick,color=blue] (1.67,2.02)--(2,1.75);
\end{scope}
\begin{scope}[yshift=3cm]
\draw[very thin,dashed,color=blue] (.76,.79)..controls (-.25,1.79) and (-1,2.5).. (-1,3.25);
\clip (.4,1.2) circle[radius=.65cm];
\draw[very thick,color=blue] (.76,.79)..controls (-.25,1.79) and (-1,2.5).. (-1,3.25);
\end{scope}
\begin{scope}[yshift=3cm] 
\draw[thick, color=green!70!black] (0,.5)--(-.5,0)..controls (-1.25,.875) and ( -1.5,1.85)..(-1.75,2.15);
\draw[thick, color=green!70!black] (0,.5)--(.5,0)..controls (1.25,.875) and ( 1.5,1.85)..(1.75,2.15);
\draw[thick, color=green!70!black] (0,1.7)--(-.5,1.2)..controls (-1.25,2.075) and ( -1.5,1.85)..(-1.75,2.15);
\draw[thick, color=green!70!black] (0,1.7)--(.5,1.2)..controls (1.25,2.075) and ( 1.5,1.85)..(1.75,2.15);
\end{scope}
\begin{scope}[yshift=3cm] 
\draw[thick,color=blue] (2,1.75).. controls (1.25,.875) and (1,1.1)..(.75,.79);
\draw[thick,color=blue] (2,3.25).. controls (1.25,2.375) and (1,1.1)..(.75,.79);
\draw[very thin,dashed,color=blue] (1.5,4.25)..controls (0,4.25)and(-.25,3.7)..(-1,3.25);
\clip (1.6,3.93) rectangle (0,4.3);
\draw[thick,color=blue] (1.5,4.25)..controls (0,4.25)and(-.25,3.7)..(-1,3.25);
\end{scope}
\end{tikzpicture}
\caption{Components $\Lambda_G^i$, $\Lambda_G^j$, $\Lambda_G^k$ spiral positively around a positive vertex in $G$. The $i$th cusp line (black) goes under the $j$th cusp line (blue).}
\label{FigureT}
\end{center}
\end{figure}
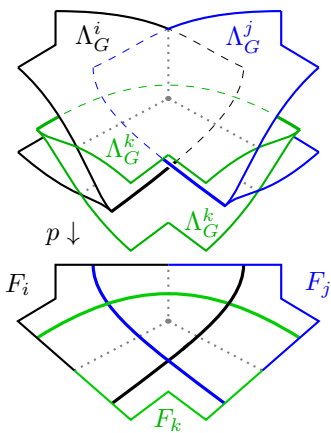

\begin{figure}[htbp] 
\begin{center}
\begin{tikzpicture}[scale=.8]

\coordinate (A) at (-1.3,0);
\coordinate (B) at (1.3,0);
\coordinate (T1) at (-.8,2);
\coordinate (T2) at (.8,2);
\coordinate (F1) at (-.8,-2);
\coordinate (F2) at (.8,-2);
\draw[fill] (A) circle[radius=1.3mm];
\draw[fill] (B) circle[radius=1.3mm];
\draw[thick] (B)--(A)..controls (T1) and (T2)..(B)..controls (F2)
and (F1)..(A);
\end{tikzpicture}
\caption{A bicolored trivalent ribbon graph with two vertices labeled with the same color.}
\label{Fig: 11}
\end{center}
\end{figure}
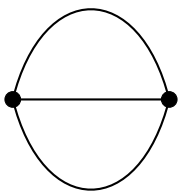
%

  %
\begin{figure}[htbp]
\begin{center}
\begin{tikzpicture}[scale=.7]
\begin{scope}[yshift=-2cm]

\begin{scope}[scale=.7, yshift=16.5cm, xshift=15cm]
\begin{scope}[scale=.95,xshift=-.35cm] 
	\draw[very thick, color=blue] (-4.5,-0.4)..controls (-4,-0.4) and (-3,-1.45)..(-1,-1.25);
	\draw[very thick, color=blue] (-4.5,-0.4)..controls (-4,-0.4) and (-3,1)..(-1,1.25);
	\begin{scope}[xshift=-2cm] 
		\draw[very thick, color=blue] (5.2,.4)..controls (4.5,.4) and (3,-1.25)..(1,-1.25);
		\draw[very thick, color=blue] (5.2,.4)..controls (4.5,.4) and (3,1.55)..(1,1.25);
	\end{scope}
\end{scope}
\begin{scope}[xshift=-4.5cm] 
	\draw[very thick] (-1.5,1.25)..controls (0,1.25) and (1,0)..(2,0) ;
	\draw[very thick] (-1.5,-1.25)..controls (0,-1.25) and (1,0)..(2,0) ;
	\begin{scope}[xshift=1cm]
		\draw[very thick] (4,0)..controls (5,0) and (6,1.25)..(7.5,1.25);
		\draw[very thick] (4,0)..controls (5,0) and (6,-1.25)..(7.5,-1.25);
	\end{scope}
\end{scope}
\begin{scope}
	\draw[very thick, color=red] (-5,0)..controls (-4,0) and (-3,1.25)..(-1,1.25);
	\draw[very thick, color=red] (-5,0)..controls (-4,0) and (-3,-1.25)..(-1,-1.25);
	\begin{scope}[xshift=-2cm] 
		\draw[very thick, color=red] (5,0)..controls (4,0) and (3,1.25)..(1,1.25);
		\draw[very thick, color=red] (5,0)..controls (4,0) and (3,-1.25)..(1,-1.25);
	\end{scope}
\end{scope}
\end{scope}

\begin{scope}[xshift=7.7cm,scale=.7,yshift=-16.2cm]%
\draw[very thick] (-2,1)--(8,1);
\draw[very thick] (-2,-1)--(8,-1) ;
\end{scope} 

\begin{scope}[scale=.7, yshift=-11.8cm,xshift=15cm]

\begin{scope}[xshift=-4cm] 
\draw[very thick] (-2,1.25)..controls (0,1.25) and (1,1)..(3,1)..controls (5,1) and (6,1.25)..(8,1.25);
\draw[very thick] (-2,-1.25)..controls (0,-1.25) and (1,-1)..(3,-1)..controls (5,-1) and (6,-1.25)..(8,-1.25) ;
\end{scope}
\begin{scope}[scale=.4,xshift=-2cm] 
\draw[very thick, color=red] (-5,0)..controls (-4,0) and (-3,1.25)..(-1,1.25);
\draw[very thick, color=red] (-5,0)..controls (-4,0) and (-3,-1.25)..(-1,-1.25);
\begin{scope}[xshift=-2cm] 
\draw[very thick, color=red] (5,0)..controls (4,0) and (3,1.25)..(1,1.25);
\draw[very thick, color=red] (5,0)..controls (4,0) and (3,-1.25)..(1,-1.25);
\end{scope}
\end{scope}
\end{scope}

\begin{scope}[scale=.7, yshift=-7.2cm,xshift=15cm]

\begin{scope}[xshift=-4cm] 
\draw[very thick] (-2,1.25)..controls (0,1.25) and (1,.4)..(3,.4)..controls (5,.4) and (6,1.25)..(8,1.25);
\draw[very thick] (-2,-1.25)..controls (0,-1.25) and (1,-.4)..(3,-.4)..controls (5,-.4) and (6,-1.25)..(8,-1.25) ;
\end{scope}

\begin{scope}[scale=.7,xshift=-.5cm] 
\draw[very thick, color=red] (-5,0)..controls (-4,0) and (-3,1.25)..(-1,1.25);
\draw[very thick, color=red] (-5,0)..controls (-4,0) and (-3,-1.25)..(-1,-1.25);
\end{scope}
\begin{scope}[scale=.7,xshift=-2.5cm] 
\draw[very thick, color=red] (5,0)..controls (4,0) and (3,1.25)..(1,1.25);
\draw[very thick, color=red] (5,0)..controls (4,0) and (3,-1.25)..(1,-1.25);
\end{scope}
\end{scope}

\begin{scope}[scale=.7, yshift=-2.7cm,xshift=15cm]
\begin{scope}[xshift=-4cm] 
\draw[very thick] (-2,1.25)..controls (0,1.25) and (1,0)..(2,0) (4,0)..controls (5,0) and (6,1.25)..(8,1.25);
\draw[very thick] (-2,-1.25)..controls (0,-1.25) and (1,0)..(2,0) (4,0)..controls (5,0) and (6,-1.25)..(8,-1.25);
\end{scope}
\begin{scope}
\draw[very thick, color=red] (-5,0)..controls (-4,0) and (-3,1.25)..(-1,1.25);
\draw[very thick, color=red] (-5,0)..controls (-4,0) and (-3,-1.25)..(-1,-1.25);
\end{scope}

\begin{scope}[xshift=-2cm] 
\draw[very thick, color=red] (5,0)..controls (4,0) and (3,1.25)..(1,1.25);
\draw[very thick, color=red] (5,0)..controls (4,0) and (3,-1.25)..(1,-1.25);
\end{scope}
\end{scope}

\begin{scope}[scale=.7, yshift=2cm,xshift=15cm]
\begin{scope}[scale=.25,xshift=-2.8cm] 
\draw[very thick, color=blue] (-5,0)..controls (-4,0) and (-3,1.25)..(-1,1.25);
\draw[very thick, color=blue] (-5,0)..controls (-4,0) and (-3,-1.25)..(-1,-1.25);
\begin{scope}[xshift=-2cm] 
\draw[very thick, color=blue] (5,0)..controls (4,0) and (3,1.25)..(1,1.25);
\draw[very thick, color=blue] (5,0)..controls (4,0) and (3,-1.25)..(1,-1.25);
\end{scope}
\end{scope}
\begin{scope}[xshift=-4.5cm] 
\draw[very thick] (-1.5,1.25)..controls (0,1.25) and (1,0)..(2,0) ;
\draw[very thick] (-1.5,-1.25)..controls (0,-1.25) and (1,0)..(2,0) ;
\begin{scope}[xshift=1cm]
\draw[very thick] (4,0)..controls (5,0) and (6,1.25)..(7.5,1.25);
\draw[very thick] (4,0)..controls (5,0) and (6,-1.25)..(7.5,-1.25);
\end{scope}
\end{scope}
\begin{scope}
\draw[very thick, color=red] (-5,0)..controls (-4,0) and (-3,1.25)..(-1,1.25);
\draw[very thick, color=red] (-5,0)..controls (-4,0) and (-3,-1.25)..(-1,-1.25);
\end{scope}

\begin{scope}[xshift=-2cm] 
\draw[very thick, color=red] (5,0)..controls (4,0) and (3,1.25)..(1,1.25);
\draw[very thick, color=red] (5,0)..controls (4,0) and (3,-1.25)..(1,-1.25);
\end{scope}
\end{scope}

\begin{scope}[scale=.7, yshift=6.8cm,xshift=15cm]
\begin{scope}[scale=.25,xshift=-4cm] 
\coordinate (A) at (-5.5,2.4);
\coordinate (B) at (-4.5,2.5);
\coordinate (Ar) at (5.5,-1.5);
\coordinate (Br) at (4.5,-1.6);
\coordinate (Ap) at (-5.5,1.5);
\coordinate (Bp) at (-4.5,1.6);
\coordinate (Arp) at (5.5,-2.4);
\coordinate (Brp) at (4.5,-2.5);
\coordinate (R) at (6.5,-1.5);
\coordinate (L) at (-6.5,1.5);

\draw[very thick, color=blue] (L)..controls (Ap) and (Bp)..(0,-0.7)..controls (Brp) and (Arp)..(R);
\draw[very thick, color=blue] (L)..controls (A) and (B)..(0,.7)..controls (Br) and (Ar)..(R);
\end{scope}
\begin{scope}[xshift=-4.5cm] 
\draw[very thick] (-1.5,1.25)..controls (0,1.25) and (1,0)..(2,0) ;
\draw[very thick] (-1.5,-1.25)..controls (0,-1.25) and (1,0)..(2,0) ;
\begin{scope}[xshift=1cm]
\draw[very thick] (4,0)..controls (5,0) and (6,1.25)..(7.5,1.25);
\draw[very thick] (4,0)..controls (5,0) and (6,-1.25)..(7.5,-1.25);
\end{scope}
\end{scope}
\begin{scope}
\draw[very thick, color=red] (-5,0)..controls (-4,0) and (-3,1.25)..(-1,1.25);
\draw[very thick, color=red] (-5,0)..controls (-4,0) and (-3,-1.25)..(-1,-1.25);
\end{scope}

\begin{scope}[xshift=-2cm] 
\draw[very thick, color=red] (5,0)..controls (4,0) and (3,1.25)..(1,1.25);
\draw[very thick, color=red] (5,0)..controls (4,0) and (3,-1.25)..(1,-1.25);
\end{scope}
\end{scope}


\begin{scope}[scale=.7, yshift=11.6cm,xshift=15cm]
\begin{scope}[scale=.8,xshift=-.4cm] 
	\draw[very thick, color=blue] (-5,0)..controls (-4,0) and (-3,1.25)..(-1,1.25);
	\draw[very thick, color=blue] (-5,0)..controls (-4,0) and (-3,-1.25)..(-1,-1.25);
	\begin{scope}[xshift=-2cm] 
		\draw[very thick, color=blue] (5,0)..controls (4,0) and (3,1.25)..(1,1.25);
		\draw[very thick, color=blue] (5,0)..controls (4,0) and (3,-1.25)..(1,-1.25);
	\end{scope}
\end{scope}
\begin{scope}[xshift=-4.5cm] 
\draw[very thick] (-1.5,1.25)..controls (0,1.25) and (1,0)..(2,0) ;
\draw[very thick] (-1.5,-1.25)..controls (0,-1.25) and (1,0)..(2,0) ;
\begin{scope}[xshift=1cm]
\draw[very thick] (4,0)..controls (5,0) and (6,1.25)..(7.5,1.25);
\draw[very thick] (4,0)..controls (5,0) and (6,-1.25)..(7.5,-1.25);
\end{scope}
\end{scope}
\begin{scope}
\draw[very thick, color=red] (-5,0)..controls (-4,0) and (-3,1.25)..(-1,1.25);
\draw[very thick, color=red] (-5,0)..controls (-4,0) and (-3,-1.25)..(-1,-1.25);
\end{scope}

\begin{scope}[xshift=-2cm] 
\draw[very thick, color=red] (5,0)..controls (4,0) and (3,1.25)..(1,1.25);
\draw[very thick, color=red] (5,0)..controls (4,0) and (3,-1.25)..(1,-1.25);
\end{scope}
\end{scope}
\end{scope}


\begin{scope}[scale=.7, yshift=-19cm]
\begin{scope}[scale=1]
	\draw[very thick, color=blue] (-5,0)..controls (-4,0) and (-3,1.25)..(-1,1.25);
	\draw[very thick, color=blue] (-5,0)..controls (-4,0) and (-3,-1.25)..(-1,-1.25);
	\begin{scope}[xshift=-2cm] 
		\draw[very thick, color=blue] (5,0)..controls (4,0) and (3,1.25)..(1,1.25);
		\draw[very thick, color=blue] (5,0)..controls (4,0) and (3,-1.25)..(1,-1.25);
	\end{scope}
\end{scope}
\begin{scope}[xshift=-4.5cm] 
\draw[very thick] (-1.5,-1.25)..controls (0,-1.25) and (1,0)..(2,0) ;
\draw[very thick] (-1.5,1.25)..controls (0,1.25) and (1,0)..(2,0) ;
\begin{scope}[xshift=1cm]
\draw[very thick] (4,0)..controls (5,0) and (6,1.25)..(7.5,1.25);
\draw[very thick] (4,0)..controls (5,0) and (6,-1.25)..(7.5,-1.25);
\end{scope}
\end{scope}
\begin{scope}[scale=.95,xshift=-.35cm]
	\draw[very thick, color=red] (-4.5,0.4)..controls (-4,0.4) and (-3,1.45)..(-1,1.25);
	\draw[very thick, color=red] (-4.5,0.4)..controls (-4,0.4) and (-3,-1)..(-1,-1.25);
	\begin{scope}[xshift=-2cm] 
		\draw[very thick, color=red] (5.2,-.4)..controls (4.5,-.4) and (3,1.25)..(1,1.25);
		\draw[very thick, color=red] (5.2,-.4)..controls (4.5,-.4) and (3,-1.55)..(1,-1.25);
	\end{scope}\end{scope}

\end{scope}


\begin{scope}[scale=.7,yshift=-14.5cm]
\begin{scope}[scale=.8,xshift=-.4cm] 
\draw[very thick, color=red] (-5,0)..controls (-4,0) and (-3,1.25)..(-1,1.25);
\draw[very thick, color=red] (-5,0)..controls (-4,0) and (-3,-1.25)..(-1,-1.25);
\begin{scope}[xshift=-2cm] 
\draw[very thick, color=red] (5,0)..controls (4,0) and (3,1.25)..(1,1.25);
\draw[very thick, color=red] (5,0)..controls (4,0) and (3,-1.25)..(1,-1.25);
\end{scope}
\end{scope}
\begin{scope}[xshift=-4.5cm] 
\draw[very thick] (-1.5,1.25)..controls (0,1.25) and (1,0)..(2,0) ;
\draw[very thick] (-1.5,-1.25)..controls (0,-1.25) and (1,0)..(2,0) ;
\begin{scope}[xshift=1cm]
\draw[very thick] (4,0)..controls (5,0) and (6,1.25)..(7.5,1.25);
\draw[very thick] (4,0)..controls (5,0) and (6,-1.25)..(7.5,-1.25);
\end{scope}
\end{scope}
\begin{scope}
\draw[very thick, color=blue] (-5,0)..controls (-4,0) and (-3,1.25)..(-1,1.25);
\draw[very thick, color=blue] (-5,0)..controls (-4,0) and (-3,-1.25)..(-1,-1.25);
\end{scope}

\begin{scope}[xshift=-2cm] 
\draw[very thick, color=blue] (5,0)..controls (4,0) and (3,1.25)..(1,1.25);
\draw[very thick, color=blue] (5,0)..controls (4,0) and (3,-1.25)..(1,-1.25);
\end{scope}
\end{scope}

\begin{scope}[scale=.7,yshift=-10cm]
\begin{scope}[scale=.25,xshift=-4cm] 
\coordinate (A) at (-5.5,-2.4);
\coordinate (B) at (-4.5,-2.5);
\coordinate (Ar) at (5.5,1.5);
\coordinate (Br) at (4.5,1.6);
\coordinate (Ap) at (-5.5,-1.5);
\coordinate (Bp) at (-4.5,-1.6);
\coordinate (Arp) at (5.5,2.4);
\coordinate (Brp) at (4.5,2.5);
\coordinate (R) at (6.5,1.5);
\coordinate (L) at (-6.5,-1.5);

\draw[very thick, color=red] (L)..controls (Ap) and (Bp)..(0,0.7)..controls (Brp) and (Arp)..(R);
\draw[very thick, color=red] (L)..controls (A) and (B)..(0,-.7)..controls (Br) and (Ar)..(R);
\end{scope}
\begin{scope}[xshift=-4.5cm] 
\draw[very thick] (-1.5,1.25)..controls (0,1.25) and (1,0)..(2,0) ;
\draw[very thick] (-1.5,-1.25)..controls (0,-1.25) and (1,0)..(2,0) ;
\begin{scope}[xshift=1cm]
\draw[very thick] (4,0)..controls (5,0) and (6,1.25)..(7.5,1.25);
\draw[very thick] (4,0)..controls (5,0) and (6,-1.25)..(7.5,-1.25);
\end{scope}
\end{scope}
\begin{scope}
\draw[very thick, color=blue] (-5,0)..controls (-4,0) and (-3,1.25)..(-1,1.25);
\draw[very thick, color=blue] (-5,0)..controls (-4,0) and (-3,-1.25)..(-1,-1.25);
\end{scope}

\begin{scope}[xshift=-2cm] 
\draw[very thick, color=blue] (5,0)..controls (4,0) and (3,1.25)..(1,1.25);
\draw[very thick, color=blue] (5,0)..controls (4,0) and (3,-1.25)..(1,-1.25);
\end{scope}
\end{scope}

\begin{scope}[scale=.7, yshift=-5.5cm]
\begin{scope}[scale=.25,xshift=-2.8cm] 
\draw[very thick, color=red] (-5,0)..controls (-4,0) and (-3,1.25)..(-1,1.25);
\draw[very thick, color=red] (-5,0)..controls (-4,0) and (-3,-1.25)..(-1,-1.25);
\begin{scope}[xshift=-2cm] 
\draw[very thick, color=red] (5,0)..controls (4,0) and (3,1.25)..(1,1.25);
\draw[very thick, color=red] (5,0)..controls (4,0) and (3,-1.25)..(1,-1.25);
\end{scope}
\end{scope}
\begin{scope}[xshift=-4.5cm] 
\draw[very thick] (-1.5,1.25)..controls (0,1.25) and (1,0)..(2,0) ;
\draw[very thick] (-1.5,-1.25)..controls (0,-1.25) and (1,0)..(2,0) ;
\begin{scope}[xshift=1cm]
\draw[very thick] (4,0)..controls (5,0) and (6,1.25)..(7.5,1.25);
\draw[very thick] (4,0)..controls (5,0) and (6,-1.25)..(7.5,-1.25);
\end{scope}
\end{scope}
\begin{scope}
\draw[very thick, color=blue] (-5,0)..controls (-4,0) and (-3,1.25)..(-1,1.25);
\draw[very thick, color=blue] (-5,0)..controls (-4,0) and (-3,-1.25)..(-1,-1.25);
\end{scope}

\begin{scope}[xshift=-2cm] 
\draw[very thick, color=blue] (5,0)..controls (4,0) and (3,1.25)..(1,1.25);
\draw[very thick, color=blue] (5,0)..controls (4,0) and (3,-1.25)..(1,-1.25);
\end{scope}
\end{scope}

\begin{scope}[yshift=-.5cm,scale=.7]
\begin{scope}[xshift=-4cm] 
\draw[very thick] (-2,1.25)..controls (0,1.25) and (1,0)..(2,0) (4,0)..controls (5,0) and (6,1.25)..(8,1.25);
\draw[very thick] (-2,-1.25)..controls (0,-1.25) and (1,0)..(2,0) (4,0)..controls (5,0) and (6,-1.25)..(8,-1.25);
\end{scope}
\begin{scope}
\draw[very thick, color=blue] (-5,0)..controls (-4,0) and (-3,1.25)..(-1,1.25);
\draw[very thick, color=blue] (-5,0)..controls (-4,0) and (-3,-1.25)..(-1,-1.25);
\end{scope}

\begin{scope}[xshift=-2cm] 
\draw[very thick, color=blue] (5,0)..controls (4,0) and (3,1.25)..(1,1.25);
\draw[very thick, color=blue] (5,0)..controls (4,0) and (3,-1.25)..(1,-1.25);
\end{scope}
\end{scope}

\begin{scope}[scale=.7, yshift=4cm]
\begin{scope}[xshift=-4cm] 
\draw[very thick] (-2,1.25)..controls (0,1.25) and (1,.4)..(3,.4)..controls (5,.4) and (6,1.25)..(8,1.25);
\draw[very thick] (-2,-1.25)..controls (0,-1.25) and (1,-.4)..(3,-.4)..controls (5,-.4) and (6,-1.25)..(8,-1.25) ;
\end{scope}

\begin{scope}[scale=.7,xshift=-.5cm] 
\draw[very thick, color=blue] (-5,0)..controls (-4,0) and (-3,1.25)..(-1,1.25);
\draw[very thick, color=blue] (-5,0)..controls (-4,0) and (-3,-1.25)..(-1,-1.25);
\end{scope}
\begin{scope}[scale=.7,xshift=-2.5cm] 
\draw[very thick, color=blue] (5,0)..controls (4,0) and (3,1.25)..(1,1.25);
\draw[very thick, color=blue] (5,0)..controls (4,0) and (3,-1.25)..(1,-1.25);
\end{scope}
\end{scope}

\begin{scope}[scale=.7,yshift=8.9cm] 
\begin{scope}[xshift=-4cm] 
\draw[very thick] (-2,1.25)..controls (0,1.25) and (1,1)..(3,1)..controls (5,1) and (6,1.25)..(8,1.25);
\draw[very thick] (-2,-1.25)..controls (0,-1.25) and (1,-1)..(3,-1)..controls (5,-1) and (6,-1.25)..(8,-1.25) ;
\end{scope}
\begin{scope}[scale=.4,xshift=-2cm] 
\draw[very thick, color=blue] (-5,0)..controls (-4,0) and (-3,1.25)..(-1,1.25);
\draw[very thick, color=blue] (-5,0)..controls (-4,0) and (-3,-1.25)..(-1,-1.25);
\begin{scope}[xshift=-2cm] 
\draw[very thick, color=blue] (5,0)..controls (4,0) and (3,1.25)..(1,1.25);
\draw[very thick, color=blue] (5,0)..controls (4,0) and (3,-1.25)..(1,-1.25);
\end{scope}
\end{scope}
\end{scope} 

\begin{scope}[xshift=-2.8cm,yshift=9.6cm,scale=.7]
\draw[very thick] (-2,1)--(8,1);
\draw[very thick] (-2,-1)--(8,-1) ;
\end{scope} 

\end{tikzpicture}
\caption{The mesh Legendrian $\Lambda_G$ corresponding to the $G$ of Figure \ref{Fig: 11} consists of three components, one black, one blue and one red. Its front $\pi(\Lambda_G)$ is a subset of $J^0(S^2)=S^2 \times \bR_z$. We think of $S^2$ as $\bR^2_{x,y}$ compactified at infinity and draw $\pi(\Lambda_G)$ as a movie of 1-dimensional fronts $\Sigma_y \subset J^0(\bR_x)=\bR_x \times \bR_z$, fixed at infinity, where $y$ is the time coordinate. We will see that $\Lambda_G$ is generated by a function $f:S^3 \to \bR$, where we view $S^3$ as the total space of the Hopf fibration $S^1 \to S^3 \to S^2$. The Legendrian $\Lambda_G$ is a Borromean link in the sense that any two of the three components of $\Lambda_G$ form a trivial Legendrian link, but the three together form a nontrivial Legendrian link. Moreover, $\Lambda_G$ is trivial as a formal Legendrian link.}
\label{Fig: start}
\end{center}
\end{figure}
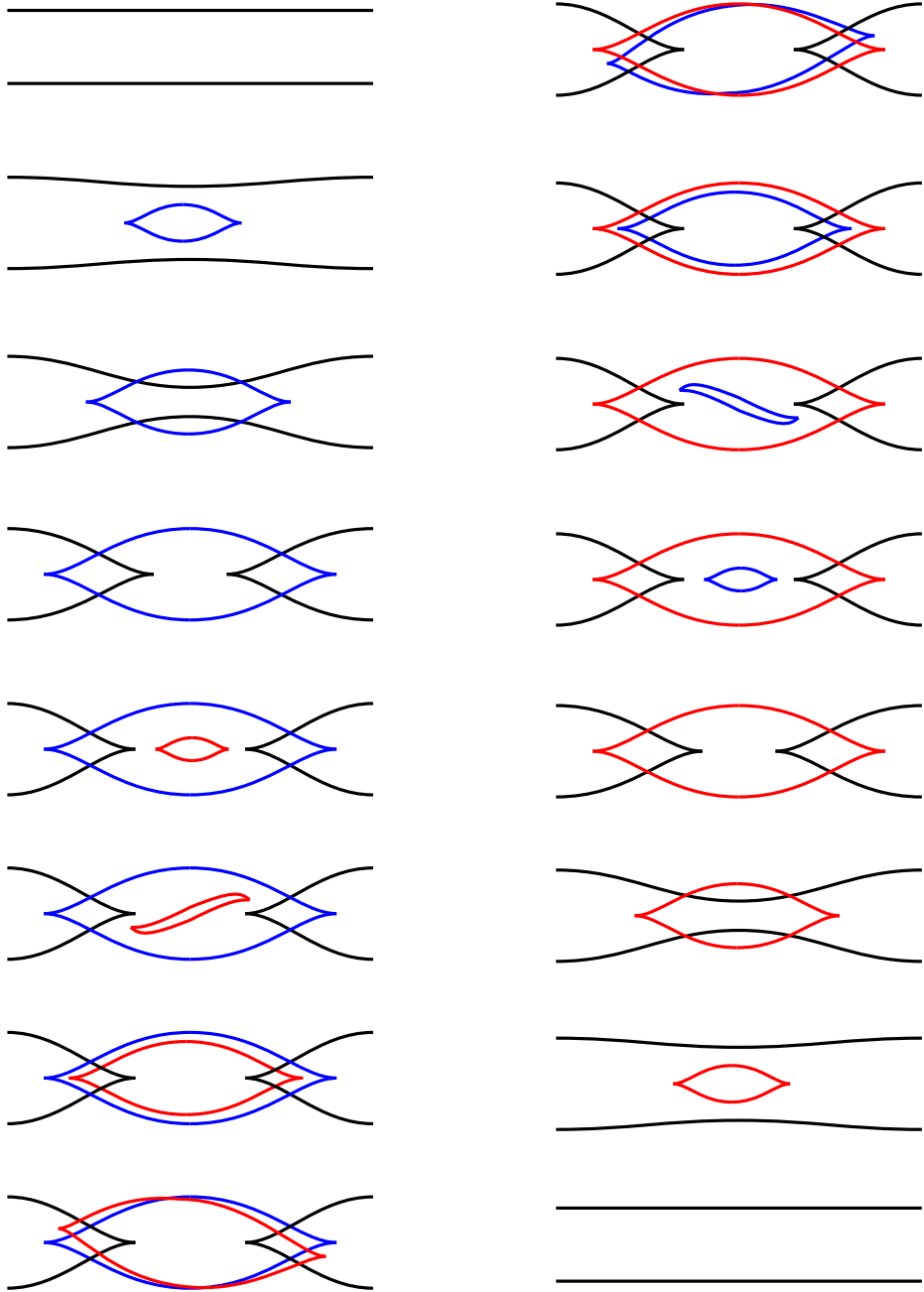

As we will see in Section \ref{systems of disks}, every mesh Legendrian $\Lambda_G$ admits a generating family on the circle bundle $S^1 \to E_G \to \Sigma_G $ of Euler number $\pm w(G)$, where $w(G) \in \bZ$ is defined as follows. 

 \begin{definition}
The winding number of a bicolored trivalent ribbon graph $G$ is $w(G) =  \frac{1}{2}(P-N)$, where $P$ is the number of positive vertices and $N$ is the number of negative vertices.
 \end{definition}
 
\begin{remark}
Since $G$ is trivalent, the number of vertices $P+N$ is even, hence so is $P - N$.
\end{remark}

By virtue of the homotopy lifting property for generating families, the fact that a mesh Legendrian $\Lambda_G$ admits a generating family on a circle bundle implies that $\Lambda_G$ is globally linked. This means that any $\Lambda$ which is Legendrian isotopic to $\Lambda_G$ must intersect every fibre of the projection $J^1(\Sigma_G) \to \Sigma_G$. 
The argument only uses the homology of the $S^1$ fibre, see Corollary \ref{corollary: mesh Legendrians are globally linked}. In order to distinguish different mesh Legendrians from each other we will use not only the properties of the fibre, but also the global properties of the circle bundle. In fact, in addition to circle bundles $E$ we will consider stabilized circle bundles $W=E \times \bR^{2k}$. 
We only allow generating families on $E \times \bR^{2k}$ which are fibrations at infinity and remain a bounded $C^1$ distance from the standard quadratic form $||x||^2-||y||^2$, for example the stabilization of a generating family on $E$.  We denote by $e(E) \in \bZ$ the Euler number of an oriented circle bundle $E$. Recall that changing the orientation of $E$ changes the sign of $e(E)$.

 \begin{theorem}\label{euler} A mesh Legendrian $\Lambda_G$ admits a generating family on the stabilized circle bundle $W=E \times \bR^{2k}$ if and only if $e(E)=\pm w(G)$. 
\end{theorem}
 
Hence $E_G$ is the only circle bundle which can be used to generate $\Lambda_G$, even stably. Combining Theorem \ref{euler} with the homotopy lifting property for generating families we deduce the following.
 
  \begin{corollary}\label{reidemeister}
If $\Lambda_{G_1}$ and $\Lambda_{G_2}$ are Legendrian isotopic, then $|w(G_1)|=|w(G_2)|$.
 \end{corollary}

To show that $\Lambda_G$ remembers the sign of the winding number $w(G)$ we use the Legendrian Turaev torsion. Fix a rank 1 unitary local system $\rho: \pi_1 E_G \to U(1)$ whose restriction to the fundamental group of the fibre $\pi_1S^1$ is nontrivial. Recall that the Legendrian Turaev torsion is a priori a set $T(\Lambda, W, \rho) \subset \bC^\times / \pm 1$, whose elements correspond to different generating families $f$ for $\Lambda$ on an even stabilization of $W$. It turns out that for a mesh Legendrian $\Lambda=\Lambda_G$ and for the circle bundle $W=E_G$, the Legendrian Turaev torsion is a nonzero complex number well defined up to sign $\tau(G, \rho) \in \bC^\times/ \pm 1$, which does not depend on $f$. In other words, $T(\Lambda_G,E_G,\rho)=\{\tau(G,\rho)\}$ is a one-element set.

Explicitly, set $n=|w(G)|$ for a mesh Legendrian $\Lambda_G$. We assume $n \neq 1$, since for $n=1$ there are no local systems $\rho: \pi_1 E_G \to U(1)$ whose pullback to the fundamental group of the fibre is nontrivial. The orientation which gives the circle bundle $S^1 \to E_G \to \Sigma_G$ its non-negative Euler number $n \geq 0$ determines an isomorphism $\pi_1S^1 \simeq \bZ$ for the $S^1$ fibre. Set $\ve \in \{ \pm 1 \}$ to be the sign of $w(G) \in \bZ$. 

\begin{theorem}\label{calculation} The Legendrian Turaev torsion of $\Lambda_G$ with respect to the rank 1 unitary local system $\rho:\pi_1E_G \to U(1)$ which sends $1 \in \bZ$ to the primitive $n$-th root of unity $\zeta^{-1}$ is $\tau(G, \rho)=\pm(1-\zeta^\ve)$.
\end{theorem}

 \begin{remark}\label{rem: notation of IK}
In the portions of this article involving Morse theory we use the notation of \cite{IK93} so that our calculations will be consistent with the related calculations carried out in that paper. This includes the convention that $\rho$ sends the generator of $\pi_1S^1$ to $\zeta^{-1}$, also denoted $u^{-1}$ in \cite{IK93}.
\end{remark}

Since $\pm(1-\zeta) \neq \pm(1-\zeta^{-1})$ for $n \geq 3$, we deduce the following consequence.

 \begin{corollary}\label{turaev}
 If $\Lambda_{G_1}$ and $\Lambda_{G_2}$ are Legendrian isotopic and $|w(G_i)| \neq 1,2$, then $w(G_1)=w(G_2)$.
 \end{corollary}

The techniques of the present article do not seem to work without the restriction $|w(G)|\neq 1,2$. The reason is that there do not exist interesting enough representations of $\pi_1E$ when $|e(E)| =  1,2$.

\begin{problem}
Extend Corollary \ref{turaev} to the case $|w(G)|=1,2$.
\end{problem}

We now prove Corollary \ref{corollarypairs} assuming the above results on the Turaev torsion of mesh Legendrians.

\begin{proof}[Proof of Corollary \ref{corollarypairs}] Given a trivalent ribbon graph $G$, let $\Sigma=\Sigma_G$ be the associated closed oriented surface. An orientation reversing diffeomorphism $\phi: \Sigma \to \Sigma$ produces an identification $\Sigma \simeq \Sigma_{H}$, where $H$ is the trivalent ribbon graph obtained from $G$ by reversing the cyclic orientation at each vertex. Note that $H$ is embedded in $\Sigma$ as $\phi(G)$. 

Let $\Lambda_+ \subset J^1(\Sigma)$ be the mesh Legendrian corresponding to the bicolored trivalent ribbon graph $G^+$ obtained from $G$ by coloring every vertex with positive labels. Let $\Lambda_- \subset J^1(\Sigma)$ be the mesh Legendrian corresponding to the bicolored trivalent ribbon graph $H^-$ obtained from $H$ by coloring every vertex with negative labels. Note that $\Lambda_-=\Phi(\Lambda_+)$ for $\Phi=j^1(\phi)$ the 1-jet lift of $\phi$. Note also that $w(G^+)=n$ and $w(H^-)=-n$, where $G$ has $2n$ vertices. Hence $w(G^+)$ and $w( H^-)$ have the same absolute value but opposite signs. 

When $n\geq 3$ we can apply Corollary \ref{turaev} and conclude that $\Lambda_+$ is not Legendrian isotopic to $\Lambda_-$. This proves property (c) of Corollary \ref{corollarypairs}. Property (b) is true by construction and property (a) is easy to check by hand because both $\Lambda_+$ and $\Lambda_-$ are formal Legendrian unlinks, see Section \ref{formal triviality}.  \end{proof}
The simplest examples of mesh Legendrians have appeared in disguise in work of the second author with J. Klein \cite{IK93}, which studied pictures for $K_3$ and higher Reidemeister torsion of circle bundles. The relevant mesh Legendrian $\Lambda_G \subset J^1(S^2)$ is generated by a positive generalized Morse function on a lens space (viewed as a circle bundle over the sphere) and the corresponding $G$ has all vertices colored with positive labels, just like our $\Lambda_+$ above. Given a nontrivial rank 1 unitary local system on the lens space, the picture of handle slides corresponding to this function produces an element of $K_3(\bC)$ from which the higher Reidemeister torsion of the circle bundle is computed. The Legendrian Turaev torsion of the mesh Legendrian can also be computed in terms of this handle slide picture. We discuss the connection further in Section \ref{pictures for K3} of the Appendix.

\subsection{Structure of the article}
The remainder of this introductory Section \ref{Section: introduction} gives some context for our work but is in no way logically necessary for the rest of the article. In particular we will not mention pseudo-holomorphic curve theory or microlocal sheaf theory after this introduction. 

In Section \ref{Section: generating families} we review the elements of the theory of generating families for Legendrian submanifolds of 1-jet spaces, setting the stage for Legendrian Turaev torsion. We put an emphasis on the extra mileage that can be obtained from considering generating families on a fibre bundle with fibre more general than a vector space. The key ingredient we will need is the homotopy lifting property \ref{Persistence Theorem}, the statement of which can be taken as a black box for the remainder of the article. 

In Section \ref{Section: turaev torsion} we develop Legendrian Turaev torsion as an invariant of Euler Legendrians. The key point is to understand how the geometry of a generating family for an Euler Legendrian singles out a class of compatible (weak) Euler structures on the fibre of the generating family, which have a common Turaev torsion. Hence to each generating family and suitable local system is assigned a Turaev torsion, independent of the compatible (weak) Euler structure. 

In Section \ref{Section: mesh Legendrians} we compute the Legendrian Turaev torsion of mesh Legendrians and prove our structural results \ref{euler}, \ref{reidemeister}, \ref{calculation}, \ref{turaev}. The calculation consists of a 2-parametric Morse theory analysis of the picture of handle slides corresponding to a generating family for a mesh Legendrian. We show that the front projection of the mesh Legendrian essentially determines this picture, and extract from it the Legendrian Turaev torsion. 

Finally, in Appendix \ref{ap:torsion} we give a brief exposition of Whitehead, Reidemeister and Turaev torsion, of their higher analogues, and of their interaction with symplectic and contact topology.

\subsection{Naturality of Legendrian invariants}\label{naturality}
 
From Corollary \ref{corollarypairs} it follows that Legendrian Turaev torsion is not a natural invariant. We now explain what this means precisely. Let $\iota[\Lambda]$ be an invariant of Legendrian submanifolds $\Lambda \subset J^1(B)$. That $\iota[\Lambda]$ is an invariant is the property that $\iota[\Lambda]=\iota[\Lambda']$ whenever $\Lambda'$ is Legendrian isotopic to $\Lambda$. Hence an invariant is a map $\iota: \cI \to \cS$ from the set of Legendrian isotopy classes $\cI$ to some set $\cS$. Often the invariant is an algebraic gadget defined up to a certain class of isomorphisms and $\iota[\Lambda]$ is the isomorphism class. Sometimes $\iota[\Lambda]$ is only defined for a certain type of Legendrian submanifolds, for example connected, simply connected, oriented, spin...

\begin{remark}
One could categorify the discussion and instead define a Legendrian invariant to be a functor between appropriate categories, but we will restrict to the above definition for simplicity.
\end{remark}

Given a diffeomorphism $\phi:B \to B$, recall its 1-jet lift $\Phi=j^1(\phi):J^1(B) \to J^1(B)$. This is the map $(\phi^{-1})^* \times \text{id}_{\bR} : T^*B \times \bR \to   T^*B \times \bR$, where $J^1(B)=T^*B \times \bR$. Note that $(\phi^{-1})^*$ is an exact symplectomorphism and $\Phi$ is a strict contactomorphism.

\begin{definition}\label{def: naturality}
An invariant $\iota[\Lambda]$ is natural if $\iota[\Lambda]=\iota[\Phi(\Lambda)]$ for all $\Phi=j^1(\phi)$, $\phi \in \text{Diff}(B)$.
\end{definition}

As a first example, take the Chekanov-Eliashberg dga $(\cA_{\Lambda}, \partial_{\Lambda})$ with $\bZ/2$ coefficients \cite{C02},  \cite{E98}, \cite{EES07} (one make take either no basepoints or one basepoint per component). The invariant $\iota[\Lambda]$ is the stable tame isomorphism type of the dga over $\bZ/2$. Let $L_{\Lambda}=p(\Lambda) \subset T^*B$ be the image of $\Lambda$ under the Lagrangian projection $p:J^1(B) \to T^*B$. For a chord generic $\Lambda$, the generators of $\cA_{\Lambda}$ are in bijective correspondence with the self intersection points of $L_{\Lambda}$. Hence there is a bijection between the generators of $\cA_{\Lambda}$ and those of $\cA_{\Phi(\Lambda)}$ induced by $(\phi^{-1})^*|_{L_\Lambda}: L_\Lambda \to L_{\Phi(\Lambda)}$. 

To compute the differential one chooses extra data on $T^*B$, in particular an almost complex structure $J$ on $T^*B$ which is compatible with the symplectic form $\omega=dp \wedge dq$ and for which the moduli spaces which define $\partial_\Lambda$ are transversely cut out. One can then push forward this data by $\phi$ and use it to compute $\partial_{\Phi(\Lambda)}$. Indeed, $\phi_* J=d(\phi^{-1})^* \circ J \circ d \phi^*$ is also compatible with $\omega$. Moreover, the moduli spaces for $\partial_{\Phi(\Lambda)}$ are transversely cut out by $\phi_*J$ for tautological reasons and $u \mapsto (\phi^{-1})^* \circ u$ gives a bijection between rigid pseudo-holomorphic disks contributing to $\partial_\Lambda$ and to $\partial_{\Phi(\Lambda)}$. Hence for that specific $J$ and $\phi_*J$ we find that the algebras $(\cA_\Lambda, \partial_\Lambda)$ and $(\cA_{\Phi(\Lambda)}, \partial_{\Phi(\Lambda)})$ are isomorphic on the nose. We conclude that the invariant $\iota[\Lambda]$ is natural in the sense of Definition \ref{def: naturality}.

From the viewpoint of microlocal sheaf theory, the most basic invariant is also natural in the sense of Definition \ref{def: naturality}. Recall that for a closed Legendrian submanifold $\Lambda \subset J^1(B)$ we can view $J^1(B)$ as an open subset of the cosphere bundle $T^\infty J^0(B)$. Following \cite{STZ17}, consider the category $\cC(\Lambda)$ of constructible sheaves on $J^0(B)$ with microlocal support on $\Lambda \subset T^\infty J^0(B)$. By the sheaf quantization theorem of \cite{GKS12}, the equivalence class of this category is a Legendrian invariant $\iota[\Lambda]$. That this invariant is natural under $\Phi=j^1(\phi)$ follows from the push-forward of sheaves $j^0(\phi)_*: \cC(\Lambda) \to \cC\big(\Phi(\Lambda)\big)$. 

The above examples of Legendrian invariants can be greatly generalized, for instance the Chekanov-Eliashberg dga has now been defined over coefficients much more refined than $\bZ/2$. Some of these invariants are natural and some are not. Generally speaking, to obtain a Legendrian invariant which is not natural one must introduce an ambient object into the picture. For example one can use relative $H_2$ coefficients in the Chekanov-Eliashberg dga. In the case of Legendrian Turaev torsion the ambient object is a fibre bundle over the base. It would be interesting to know whether one can express Legendrian Turaev torsion in terms of pseudo-holomorphic curves or microlocal sheaves. A weaker question is whether our mesh Legendrians can be distinguished using pseudo-holomorphic curves or microlocal sheaves. We believe it is most likely possible, but we do not know how to do it.
\begin{remark}
Of course the equivalence class of a Legendrian submanifold as a formal Legendrian is a Legendrian invariant which is not in general natural. For example, the smooth isotopy class of $\Lambda$ as a smooth link need not be natural. Hence the significance of property (a) in Corollary \ref{corollarypairs}. 
\end{remark}


\subsection{Cluster algebras}
 Bicolored trivalent ribbon graphs appear in the theory of cluster algebras, which has been a source of motivation for the present article. Consider Figure \ref{FigureCA}, which is completed in Figure \ref{FigureCA2} illustrating the proof of Theorem \ref{thm: punching hole} (the Jensen-King-Su equation $b^k=a^{n-k}$ is proved in Remark \ref{proof of JKS}.) 
Figure \ref{FigureCA} illustrates the relationship between plabic diagrams (planar bicolored graphs) as shown on the left side of the figure and the dual quiver (directed graph) shown on the right side of the figure. This dual graph is oriented clockwise around each white vertex and counterclockwise around each black vertex. The Jacobian algebra of this dual graph gives a categorification of the Grassmannian $G(2,10)$ in the example, see \cite{JKS16}. In the present article the plabic diagram is completed to a ribbon graph, see Figure \ref{FigureCA2}. The dual graph corresponds to the handle slide pattern of any family of Morse functions which generates the corresponding mesh Legendrian. Study of this type of diagram has also led the second author, together with E. Hanson, to a counterexample of the $\phi$-dimension conjecture in representation theory \cite{HI19}. Finally, we mention that cluster algebras have already produced interesting examples of Legendrian submanifolds in the work of Shende, Treumann, Williams and Zaslow \cite{STWZ15}, who also use plabic diagrams but in a different way. Trivalent graphs were also used to produce interesting Legendrians in the work of Casals, Murphy and Sachel \cite{CM18}.

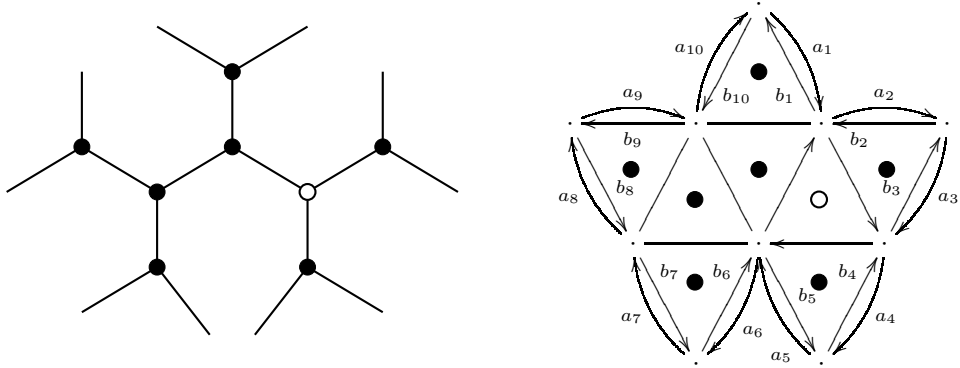
\begin{figure}[htbp]
\begin{center}
\begin{tikzpicture}
\coordinate (A1) at (-7,1.6);
\coordinate (A2) at (-5,1.6);
\coordinate (K1) at (-8,1);
\coordinate (K2) at (-6,1);
\coordinate (K3) at (-4,1);
\coordinate (L1) at (-8,0);
\coordinate (L2) at (-6,0);
\coordinate (L3) at (-4,0);
\coordinate (M0) at (-9,-.6);
\coordinate (M1) at (-7,-.6);
\coordinate (M2) at (-5,-.6);
\coordinate (M3) at (-3,-.6);
\coordinate (N1) at (-7,-1.6);
\coordinate (N2) at (-5,-1.6);
\coordinate (P1) at (-8,-2.2);
\coordinate (P2) at (-6.3,-2.5);
\coordinate (P3) at (-5.7,-2.5);
\coordinate (P4) at (-4,-2.2);
\foreach\x in {K2,L1,L2,L3,M1,N1,N2}\draw[fill](\x) circle[radius=3pt];
\begin{scope}
\draw[thick] (K1)--(L1)--(M1)--(L2)--(K2) (L2)--(M2)--(L3)--(K3)(A1)--(K2)--(A2) (M0)--(L1) (M3)--(L3);
\draw[thick] (P1)--(N1)--(P2) (P3)--(N2)--(P4) (M1)--(N1) (M2)--(N2);
\end{scope}
\draw[fill,color=white] (M2) circle[radius=3pt];
\draw[thick] (M2) circle[radius=3pt];
\begin{scope}
\draw[thick] (1.8,-.7) circle[radius=3pt];
\draw[fill] (1,1) circle[radius=3pt];
\draw[fill] (.15,-.7) circle[radius=3pt];
\draw[fill] (1,-.3) circle[radius=3pt];
\draw[fill] (2.7,-.3) circle[radius=3pt];
\draw[fill] (-.7,-.3) circle[radius=3pt];
\draw[fill] (.15,-1.8) circle[radius=3pt];
\draw[fill] (1.8,-1.8) circle[radius=3pt];
\end{scope}
%
\coordinate (R) at (1,-.5);
\draw(R) node{
$\xymatrixrowsep{35pt}\xymatrixcolsep{15pt}
\xymatrix{
 & & 
	&
	\cdot\ar[ld]^(.7){b_{10}}\ar@/^.5pc/[rd]^{a_1}\\ 
\cdot\ar[dr]^(.6){b_8} \ar@/^.5pc/[rr]^{a_9}&& \cdot\ar@/^.5pc/[ru]^{a_{10}}\ar[ll]^{b_9}\ar@{-}[rr] & & \cdot\ar@/^.5pc/[rr]^{a_2}\ar[lu]^(.3){b_1}\ar[rd] &&
	\cdot \ar[ll]^(.7){b_2}\ar@/^.5pc/[dl]^{a_3}\\ 
& \cdot \ar@/^.5pc/[lu]^{a_8}\ar[dr]^(.3){b_7} \ar@{-}[rr]\ar@{-}[ru]& &
	\cdot\ar@/^.5pc/[ld]^(.6){a_6}\ar[dr]^{b_5}\ar[ru]\ar@{-}[lu]&&\cdot\ar[ll]\ar[ru]^(.4){b_3}\ar@/^.5pc/[dl]^{a_4}\\
	&&\cdot\ar@/^.5pc/[lu]^{a_7}\ar[ru]^(.7){b_6} &&\cdot\ar[ru]^(.7){b_4}\ar@/^.5pc/[lu]^(.2){a_5}
	}$};
\end{tikzpicture}
\caption{A spanning tree in a bicolored trivalent ribbon graph (left) produces, on the perimeter, a cyclic quiver with $n$ clockwise arrow $a_i$ and $n$ counterclockwise arrows $b_i$ satisfying the preprojective relations $a_ib_i=a_jb_j$ for all $i,j$ and the Jensen-King-Su relations $b^k=a^{n-k}$ when there are $k-1$ positive vertices. In this example, $a^2=b^8$.}
\label{FigureCA}
\end{center}
\end{figure}

\subsection{Acknowledgements}
The first author is very grateful to Yasha Eliashberg, who taught him about the theory of generating families and its connection with pseudo-isotopy theory. The first author is also grateful for many useful discussions with other mathematicians which helped frame the results of the present article in the wider context. In particular he would like to thank Sylvain Courte, Josh Sabloff and Lisa Traynor for discussions on generating families, Georgios Dimitroglou-Rizell, Tobias Ekholm, John Etnyre, Yank\i$ $ Lekili and Michael Sullivan for discussions on Legendrian Contact Homology, St\'ephane Guillermou for discussions on microlocal sheaf theory, David Treumann for discussions on cluster algebras, and Eduardo Fern\'andez, Javier Mart\'inez-Aguinaga and Fran Presas for discussions on formal Legendrians. Finally, the first author is grateful to the Centre de Recherches Math\'ematiques of Montreal for the opportunity to visit as a Simons scholar in the summer of 2019, during which time part of this article was written.

Many years ago, the second author had the distinct honor and great fortune to become a student of Allen Hatcher, a great mathematician who spent many hours explaining to him multiparameter Morse theory and its relation to algebraic K-theory. The second author will forever be grateful. 
The second author would also like to thank Gordana Todorov for explaining to him the results of \cite{JKS16} also summarized in \cite{5W}. He also thanks the many researchers who explained to him (mostly in lectures) plabic diagrams and their cluster categories, in particular, Yusuke Nakajima and Matthew Pressland. Finally, he would like to thank Langte Ma for many useful discussions about symplectic topology.

Both authors are very grateful to the referees, whose useful corrections and suggestions have greatly improved the text.

 
\section{Generating families}\label{Section: generating families}
 
\subsection{Legendrians as Cerf diagrams}
 
 We begin by revisiting the generating family construction. Let $B$ be a closed manifold, set $n=\dim B$ and consider the 1-jet space $J^1(B)=T^*B \times \bR$. On $J^1(B)$ there is the contact form $dz-\lambda$, where $z$ is the coordinate on $\bR$ and $\lambda$ is the Liouville 1-form on $T^*B$, which is intrinsically defined by the property that for any 1-form $\alpha$ on $B$, viewed as a map $\alpha: B \to T^*B$, we have $\alpha^*\lambda=\alpha$. It is customary to denote $\lambda=pdq$, indeed if $q_1,\ldots, q_n$ are local coordinates on $B$ and $p_1,\ldots,p_n$ are the corresponding dual coordinates, then $\lambda=\sum_i p_i dq_i$.
 
\begin{definition} An $n$-dimensional submanifold $\Lambda \subset J^1(B)$ is said to be Legendrian if $\lambda |_\Lambda=0$. 
\end{definition}

Consider first the graphical case, so $\Lambda=\Gamma(s)$ is the graph of a section $s:B \to J^1(B)$. Write $s=(\alpha,f)$ with respect to the splitting $J^1(B)=T^*B \times \bR$, so that $\alpha$ is a 1-form on $B$ and $f$ is a function on $B$. Then $\Lambda$ is Legendrian if and only if $\alpha=df$, in which case $s=j^1(f)$ is the 1-jet lift of $f$. We can think of non-graphical Legendrians $\Lambda \subset J^1(B)$ as given by the 1-jet lifts of multi-valued functions on $B$, but there is a different viewpoint which is better suited for our purposes.
 
Recall the front projection $\pi:J^1(B) \to J^0(B)$, which is the map that forgets order 1 information. Explicitly, we have $J^1(B)=T^*B \times \bR$ and $J^0(B)=B \times \bR$. Then $\pi : T^*B \times \bR \to B \times \bR$ is the product of the cotangent bundle projection and the identity on the $\bR$ factor. 
 
 \begin{remark} Note that if $\Lambda \subset J^1(B)$ is a Legendrian submanifold, then $\pi(\Lambda) \subset J^0(B)$ generically determines $\Lambda$. Indeed, the Legendrian condition is that the form $dz-pdq$ vanishes on $\Lambda$, hence the missing coordinate $p$ can generically be recovered by the formula $p=dz/dq$. 
 \end{remark}
 
Consider a fibre bundle of manifolds without boundary $F \to W \to B$, where we also assume that $B$ is compact for simplicity. Let $f:W \to \bR$ be  a function on the total space. We view $f$ as a family of functions $f_b:F_b \to \bR$ on the fibres of $W \to B$, parametrized by $b \in B$. We repeat the basic definitions given in the introduction for convenience:

\begin{definition} The Cerf diagram of $f$ is the subset $\Sigma_f=\{ (b,z) : \, \text{$z$ is a critical value of $f_b$} \} \subset B \times \bR$. \end{definition}

We will always assume the generic condition that the fibrewise derivative of $f$ satisfies $\partial_F f \pitchfork 0 $. Then $\Gamma(df) \subset T^*W$ is a graphical Lagrangian submanifold and $ \{ \partial_F f = 0 \} \subset \Gamma(df)$ is an isotropic submanifold contained in the coisotropic subbundle $E \subset T^*W$ of covectors with zero fibrewise derivative. The symplectic reduction $E \to T^*B$ restricts to a Lagrangian immersion $\{ \partial_F f= 0 \} \to T^*B$ which lifts to a Legendrian immersion $\{ \partial_Ff  = 0 \} \to J^1(B)$ via the function $f$ itself. 

 
\begin{definition} We say that a Legendrian submanifold $\Lambda \subset J^1(B)$ is generated by $f:W \to \bR$ when $\{ \partial_F f = 0 \} \to J^1 (B)$ is an embedding with image $\Lambda$. 
\end{definition}

It is easy to see that any generic $f$ such that $\partial_F f = 0$ generates an embedded Legendrian $\Lambda \subset J^1(B)$, even though its Lagrangian projection in $T^*B$ may only be immersed. See Figure \ref{Fig: simple Legendrian} for a simple example of how a 1-dimensional Legendrian can arise as a Cerf diagram. See Figure \ref{FigLegendrian} for an example of a Legendrian which does not arise as a Cerf diagram.
\begin{figure}[htbp]
\begin{center}
\begin{tikzpicture}
%
\draw[very thick,color=blue] (-3.5,0)..controls (-2.7,0) and (-1,-.55)..(0,-.62);
\draw[very thick,color=blue] (3.5,0)..controls (2.7,0) and (.7,-.6)..(0,-.62);

\draw[thick,xshift=4.5cm] (-.5,-1.5)..controls (-.3,-.1) and (-.1,-.1)..(0,0)..controls (.1,.1) and (.3,.1)..(.5,1.5);
\draw[thick,xshift=3cm] (-.5,-1.5)..controls (-.3,.1) and (-.1,.1)..(0,0)..controls (.1,-.1) and (.3,-.1)..(.5,1.5);
\draw[thick,xshift=1.5cm] (-.5,-1.5)..controls (-.3,.6) and (-.1,.6)..(0,0)..controls (.1,-.6) and (.3,-.6)..(.5,1.5);
\draw[thick] (-.5,-1.5)..controls (-.3,1) and (-.1,1)..(0,0)..controls (.1,-1) and (.3,-1)..(.5,1.5);
\draw[thick,xshift=-1.5cm] (-.5,-1.5)..controls (-.3,.6) and (-.1,.6)..(0,0)..controls (.1,-.65) and (.3,-.65)..(.5,1.5);
\draw[thick,xshift=-3cm] (-.5,-1.5)..controls (-.3,.1) and (-.1,.1)..(0,0)..controls (.1,-.1) and (.3,-.1)..(.5,1.5);
\draw[thick,xshift=-4.5cm] (-.5,-1.5)..controls (-.3,-.1) and (-.1,-.1)..(0,0)..controls (.1,.1) and (.3,.1)..(.5,1.5);
\draw[very thick,color=blue] (-3.5,0)..controls (-3,0) and (-1,.62)..(0,.62);
\draw[very thick,color=blue] (3.5,0)..controls (2.7,0) and (.7,.6)..(0,.62);
\end{tikzpicture}
\caption{A 1-parameter family of functions generating an ``eye'' shaped Cerf diagram.}
\label{Fig: simple Legendrian}
\end{center}
\end{figure}
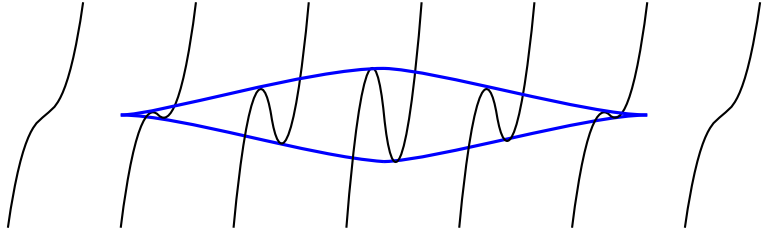
%


%
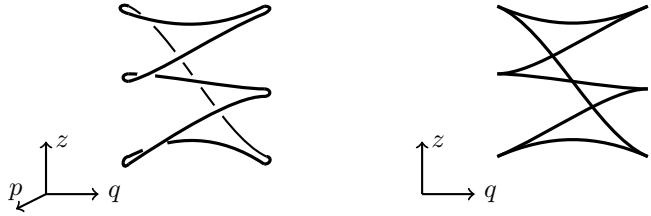
\begin{figure}[htbp]
\begin{center}
\begin{tikzpicture}
%
\begin{scope}[xshift= -3cm]
\coordinate (C) at (-2,-.5);
\coordinate (Q) at (-1.3,-.5);
\coordinate (Z) at (-2,.2);
\coordinate (P) at (-2.4,-.7);
\draw[thick,->] (C)--(Q);
\draw[thick,->] (C)--(Z);
\draw[thick,->] (C)--(P);
\draw (Q) node[right]{$q$};
\draw (Z) node[right]{$z$};
\draw (P) node[above]{$p$};
\coordinate (A0) at (-.9,-.1);
\coordinate (A0a) at (-.9,0);
\coordinate (A0b) at (-1,-.2);
\coordinate (A0c) at (-1,0);
\coordinate (A1) at (-.9,1.1);
\coordinate (A1a) at (-.9,1);
\coordinate (A1b) at (-1,1);
\coordinate (A1c) at (-1,1.1);
\coordinate (A11) at (-.5,1.1);
\coordinate (A2a) at (-.9,1.9);
\coordinate (A2) at (-1,2.05);
\coordinate (A2b) at (-.9,2);
\coordinate (A2c) at (-1.1,1.95);
\coordinate (B0) at (.93,0);
\coordinate (B0a) at (.9,-.1);
\coordinate (B0b) at (1,-.1);
\coordinate (B0c) at (1,0);
\coordinate (B1) at (.9,.8);
\coordinate (B1a) at (1,.8);
\coordinate (B1b) at (1,.9);
\coordinate (B1c) at (.9,.9);
\coordinate (B11u) at (.5,.9);
\coordinate (B11d) at (.5,.8);
\coordinate (B2) at (.9,1.9);
\coordinate (B2a) at (1,1.95);
\coordinate (B2b) at (1,2);
\coordinate (B2c) at (.9,2);
\draw[ thick] (A2b)..controls (-.3,1.7) and (.3,.3)..(B0);
\draw[fill,color=white,very thick] (-.65,1.82) circle[radius=.1cm];
\draw[fill,color=white,very thick] (-.22,1.33) circle[radius=.07cm];
\draw[fill,color=white,very thick] (0,1) circle[radius=.08cm];
\draw[fill,color=white,very thick] (.28,.6) circle[radius=.07cm];
\draw[very thick] (A2a)..controls (A2c) and (A2)..(A2b);
\draw[very thick] (A2a)..controls (-.3,1.7) and (.3,1.7)..(B2c);
\draw[very thick] (A0)..controls (A0b) and (A0c)..(A0a)..controls (-.3,.3) and (.3,.3)..(B0a) ;
\draw[very thick] (B0a)..controls (B0b) and (B0c)..(B0);
\draw[very thick] (A1)..controls (A11) and (B11u)..(B1c);
\draw[fill,color=white,very thick] (-.67,1.07) circle[radius=.1cm];
\draw[fill,color=white,very thick] (-.55,.13) circle[radius=.15cm];
\draw[very thick] (A0)..controls (-.3,.3) and (B11d)..(B1)..controls (B1a) and (B1b)..(B1c);
\draw[very thick] (A1a)..controls (A11) and (.3,1.66)..(B2)..controls (B2a) and (B2b)..(B2c);
\draw[very thick] (A1a)..controls (A1b) and (A1c)..(A1);
%
\end{scope}
\begin{scope}[xshift= 2cm]
\coordinate (C) at (-2,-.5);
\coordinate (Q) at (-1.3,-.5);
\coordinate (Z) at (-2,.2);
\draw[thick,->] (C)--(Q);
\draw[thick,->] (C)--(Z);
\draw (Q) node[right]{$q$};
\draw (Z) node[right]{$z$};
\coordinate (A0) at (-1,0);
\coordinate (A1) at (-1,1.1);
\coordinate (A11) at (-.5,1.1);
\coordinate (A2) at (-1,2);
\coordinate (B0) at (1,0);
\coordinate (B1) at (1,.9);
\coordinate (B11) at (.5,.9);
\coordinate (B2) at (1,2);
\draw[very thick] (A2)..controls (-.3,1.7) and (.3,1.7)..(B2);
\draw[very thick] (A0)..controls (-.3,.3) and (.3,.3)..(B0) ;
\draw[very thick] (A1)..controls (A11) and (B11)..(B1);
\draw[very thick] (A0)..controls (-.3,.3) and (B11)..(B1);
\draw[very thick] (A1)..controls (A11) and (.3,1.7)..(B2);
\draw[very thick] (A2)..controls (-.3,1.7) and (.3,.3)..(B0);
\end{scope}
\end{tikzpicture}
\caption{Legendrian submanifold of $J^1\mathbb R$ on left with front projection on right. This Legendrian cannot be generated by any family of functions.}
\label{FigLegendrian}
\end{center}
\end{figure}
%



We shall extract information about a Legendrian from the Morse theory of its generating families. Since we will need to consider the case of non-compact fibres $F$, it is imperative that we impose some sort of control at infinity. We use the terminology of Eliashberg and Gromov \cite{EG98}.

\begin{definition}\label{Definition: fibration at infinity}
A function $f:W \to \bR$ is a fibration at infinity if there exists a closed interval $I \subset \bR$ such that the following properties hold.
\begin{itemize}
\item $f^{-1}(\bR \setminus I) \to \bR \setminus I$ is a fibration. 
\item $f^{-1}(I)\to I$ is a fibration outside a compact subset $K \subset f^{-1}(I)$. 
\end{itemize}
\end{definition}

\begin{remark}\label{Remark: infinity family}
In a homotopy $f_t: W \to \bR$ of fibrations at infinity we require these conditions to hold uniformly, in the sense that the conditions above should hold with the same interval $I$ and the same compact set $K$ for every $f_t$.
\end{remark}

In what follows, whenever we refer to a function $f:W \to \bR$ as a generating family, it will be implicitly assumed that $f$ is a fibration at infinity. It will moreover be assumed that generating families always generate their Legendrians transversely, i.e. that $\partial_F f \pitchfork 0$.

\subsection{Existence of generating families}

Locally, Legendrians in 1-jet spaces are always given by the generating family construction. To be more precise we have the following result, probably known in some form to Hamilton and Jacobi and later rediscovered by Arnold and H\"ormander \cite{AGV85}.

\begin{proposition}
Let $\Lambda \subset J^1(B)$ be a Legendrian submanifold and $x \in \Lambda$ a point. Then there exists an integer $k \geq 0$ and a function $f:  B \times \bR^k \to \bR$ which generates $\Lambda$ locally near $x$.
\end{proposition}

Hence we can think of Legendrians as Cerf diagrams, at least locally. However, not all Legendrians $\Lambda \subset J^1(B)$ are globally given by the generating family construction. Indeed, existence of a global generating family leads to strong rigidity phenomena. In particular, flexibility is an obstruction to the global existence of a generating family. This is made precise in the following folklore result.

\begin{proposition}\label{Proposition: global non-existence} A loose Legendrian $\Lambda \subset J^1(B)$ does not admit any global generating family.
\end{proposition}

\begin{proof}Consider first the case in which there exists a ball $U \subset J^0(B)$ such that the front $\pi(\Lambda) \cap U $ has a zig-zag which is disjoint from the rest of the front, see Figure \ref{FigZigZag}. Suppose that $\Lambda$ admits a generating family $f$ on a fibre bundle of manifolds $F \to W \to B$. We recall that $f$ is assumed to be a fibration at infinity. Therefore the homology of the sublevel sets $\{ f_b \leq z \}$ of $f_b : F_b \to \bR$ can only change when $z$ crosses the front $\pi(\Lambda)$. When it does so at a Morse point the change in the homology is induced by the addition of a $k$-dimensional handle to the sublevel set, where $k$ is the index of the Morse point. 

As you go from top to bottom the Morse indices of the three sections of the zig-zag are $k-1,k$ and $k+1$, in that order, for some $k \in \bZ$. The monotonicity of the Morse indices is imposed by the normal form of a Morse birth/death point. One then derives a contradiction by computing the change in the homology of the sublevel set $\{ f_b \leq z \}$ as $z$ crosses the zig-zag near each of its two endpoints.

We now prove the general case. Since $\Lambda$ is loose, by the work of Murphy \cite{E12} there exists a Legendrian isotopy $\Lambda_t$ between $\Lambda_0=\Lambda$ and a $\Lambda_1$ for which there exists such a ball $U$. We can then reduce to the previous case using the homotopy lifting property for generating families \ref{Persistence Theorem}. \end{proof}


%
\begin{figure}[htbp]
\begin{center}
\begin{tikzpicture}
\coordinate (A) at (-2.5,.7);
\coordinate (AB) at (-.8,0.5);
\coordinate (CD) at (.8,-0.5);
\draw (AB) node[above]{$k+1$};
\draw (CD) node[below]{$k-1$};
\draw (0.1,0) node[right]{$k$};
\coordinate (ABa) at (-1.2,0.5);
\coordinate (ABb) at (0.2,0.3);
\coordinate (CDa) at (-0.2,-0.3);
\coordinate (CDb) at (1.2,-0.5);
\coordinate (B) at (1,.8);
\coordinate (C) at (-1,-.8);
\coordinate (D) at (2.5,-.7);
\draw[thick] (A)..controls (ABa) and (ABb)..(B);
\draw[thick] (B)--(C);
\draw[thick] (C)..controls (CDa) and (CDb)..(D);
\draw[thick,->] (-3.5,-.5)--(-3.5,.5);
\draw (-3.5,0) node[right]{$f_b$};
\draw (-3.3,1.1)--(-3.7,1.1) node[left]{$z_1$};
\draw (-3.3,-1.2)--(-3.7,-1.2) node[left]{$z_0$};
\end{tikzpicture}
\caption{If the zig-zag is disjoint from the rest of the front, $H_\ast(f_b^{-1}[z_0,z_1],f_b^{-1}(z_0))$ would be the same at both endpoints. But this relative homology in degree $k-1$ is zero on the left end and nonzero on the right end.}
\label{FigZigZag}
\end{center}
\end{figure}
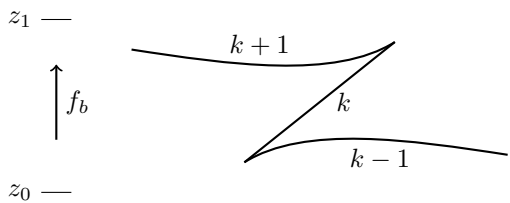

\begin{remark}
Loose Legendrians have dimension at least 2 by definition, but we note that the proof of Proposition \ref{Proposition: global non-existence} also works for stabilized 1-dimensional Legendrians.
\end{remark}

\begin{remark}
The argument in the proof of Proposition \ref{Proposition: global non-existence} is homological, hence can be recast in the language of microlocal sheaf theory to prove that there is no complex of sheaves whose microsupport is a loose Legendrian. This is spelled out for stabilized 1-dimensional Legendrians in \cite{STZ17}.
\end{remark}

From Proposition \ref{Proposition: global non-existence} we deduce the following.

\begin{corollary}
Mesh Legendrians are not loose.
\end{corollary}
\begin{proof} Every mesh Legendrian admits a generating family on a circle bundle, see Section \ref{systems of disks}.\end{proof}

\subsection{Homotopy lifting property}

The slogan is that whenever generating families exist, they persist. The precise statement can be formulated in terms of a homotopy lifting property for Legendrian isotopies. As a consequence, much of the Morse-theoretic information one can extract from a generating family $f$ for a Legendrian submanifold $\Lambda \subset J^1(B)$ is invariant under Legendrian isotopies of $\Lambda$. 

The homotopy lifting property for generating families has a long history, which in some sense starts with the work of Chaperon on Arnold's conjecture for Hamiltonian isotopies of the zero section in the cotangent bundle of the $n$-torus. Building on Chaperon's ideas \cite{C84}, Laudenbach and Sikorav gave a proof of the strong form of the Arnold conjecture in arbitrary cotangent bundles \cite{LS85} by constructing a generating family out of broken geodesics. The homotopy lifting property for generating families was then established by Sikorav \cite{S86} and Chekanov \cite{C96}. There have since appeared several generalizations in the literature. The version of the homotopy lifting property we will use is that from the article of Eliashberg and Gromov \cite{EG98}.

It is important to remark that the homotopy lifting property for generating families only holds up to stabilization. This is secretly the same stabilization that appears in pseudo-isotopy theory.
\begin{definition}
The stabilization of a fibre bundle $F \to W \to B$ is the fibre bundle $F \times \bR^2 \to W \times \bR^2 \to B$. We denote the total space by $\text{stab}(W)=W \times \bR^2$. The stabilization of a function $f:W \to \bR$ is the function  $\text{stab}(f): \text{stab}(W) \to \bR$, where
\[ \text{stab}(f)(w,x,y)=f(w)+x^2-y^2, \quad \quad (w,x,y) \in W \times \bR^2. \]
\end{definition}

We denote by $\text{stab}^k$ the $k$-fold stabilization, so that $\text{stab}^{k}(W)=W \times \bR^{2k}$. Note that if $f:W \to \bR$ is a fibration at infinity generating a Legendrian $\Lambda \subset J^1(B)$, then the function $\text{stab}^{k}(f): \text{stab}^{k}(W) \to \bR$ is also a fibration at infinity generating $\Lambda$. Of course the same would be true for any stabilization by a non-degenerate quadratic form, but we choose to always stabilize by the fixed quadratic form of balanced signature $x^2-y^2$ for reasons that will become apparent.

In fact we will always stabilize an even number of times in order to preserve the parity of the Morse indices. As we will see, it follows from the definition that stabilizing an odd number of times would change the sign of the higher Reidemeister torsion and invert the Turaev torsion. Hence from now on we insist on even stabilizations $\text{stab}^{2k}(W)=W \times \bR^{4k}$. The following is Theorem 4.1.1 from \cite{EG98}.

\begin{theorem}[Homotopy lifting property]\label{Persistence Theorem}
Let $f:W \to \bR$ be a fibration at infinity generating a Legendrian $\Lambda \subset J^1(B)$ and let $ \Lambda_t $ be a Legendrian isotopy of $\Lambda_0=\Lambda$, where $t \in [0,1]$. Then there exists $k\geq0$ and a homotopy of functions $f_t: \text{stab}^{2k}(W) \to \bR$ such that the following properties hold.
\begin{itemize}
\item $f_0=\text{stab}^{2k}(f)$.
\item $f_t$ generates $\Lambda_t$.
\item $f_t=f_0+ \varepsilon_t$ for a homotopy of compactly supported functions $\varepsilon_t$.
\end{itemize}
\end{theorem}

\begin{remark} Note in particular that every $f_t$ is a fibration at infinity.  \end{remark}



The homotopy lifting property belongs to the realm of `hard' symplectic/contact topology. It can be interpreted as a Morse-theoretic manifestation of the compactness theorems for pseudo-holomorphic curves or of the quantization theorem for microlocal sheaves. 
For example, an immediate corollary of Theorem \ref{Persistence Theorem} is the Arnold Conjecture, which gives the lower bound $\#  \varphi_1(B)  \cap B\geq \text{rank }H^*(B)$ for the number of intersection points between the zero section $B$ and any transverse Hamiltonian isotopic image $\varphi_1(B)$ in $T^*B$. In fact we get the stronger bound $\#  \varphi_1(B)  \cap B\geq \text{SM}(B)$ for $\text{SM}(B)$ the stable Morse number of $B$. Moreover, the same proof also yields the general form of the Arnold Conjecture, where we don't assume $\varphi_1(B) \pitchfork B$ and instead get the lower bound $\#  \varphi_1(B)  \cap B \geq \text{LS}(B)$, where $\text{LS}(B)$ is the Lusternik-Schnirelman category of $B$. 
 
\subsection{Generating family invariants}
 
By virtue of the homotopy lifting property there exist a number of Legendrian invariants which one can build out of generating families. Although some of the results below have been upgraded to fancier coefficients, we will restrict our exposition to the case of $\bZ/2$ coefficients for simplicity. As a first example, let $V$ be a vector space and let $\Lambda \subset J^1(B)$ be generated by $f:B \times V \to \bR$, a function which is quadratic at infinity. This means that there exists a family of nondegenerate quadratic forms $Q_b$ on $V$, parametrized by $b \in B$, such that $f(b,v) = Q_b(v)$ outside of a compact subset. The difference function $\phi: B \times V \times V \to \bR$ is defined by the formula $\phi(b,v_1,v_2)=f(b,v_1)-f(b,v_2)$, where $ b \in B$ and $v_i \in V$. Observe that the critical points of the difference function $\phi$ which do not lie on the diagonal $\{ v_1 = v_2 \} \subset B \times V \times V$ are in bijective correspondence with the Reeb chords of $\Lambda$, i.e. the self-intersections of its Lagrangian projection. 
  
 \begin{definition}
The generating family homology of $\Lambda$ with respect to $f$ is the finite-dimensional graded vector space $GH_*(\Lambda, f) = H_*(\phi<+\infty, \phi \leq \delta)$, where $\delta>0$ is small enough so that $(0, \delta)$ consists entirely of regular values of $\phi$.
 \end{definition}
 
 The following result is a corollary of the homotopy lifting property \ref{Persistence Theorem}.  
 
 \begin{theorem}\label{theorem: GF}
 The set $GH(\Lambda)=\{ GH_*(\Lambda, f) \}_f $ of generating family homologies of $\Lambda \subset J^1(B)$ for all quadratic at infinity generating families $f$ is invariant under Legendrian isotopies of $\Lambda$.
 \end{theorem}
 
 \begin{remark}
 There is also a version of Theorem \ref{theorem: GF} for compact $\Lambda$ and non compact $B$, for example $B=\bR$. In order to generate compact Legendrians one should take $V \times \bR$ instead of $V$ and demand that $f$ is linear-quadratic at infinity, i.e. $f(b,v,t)=t+Q_b(v)$ outside of a compact subset of $B \times V \times \bR$. 
 \end{remark}
 
\begin{remark}
Both the Legendrian Turaev torsion $T(\Lambda, W, \rho)$ and the $GH(\Lambda)$ of Theorem \ref{theorem: GF} consist of a set of functional invariants, whose elements are parametrized by certain collections of generating families for $\Lambda$. However, instead of a set of homological functional invariants, Legendrian Turaev torsion is a set of torsion functional invariants.
\end{remark}


Generating family homology has enjoyed numerous applications, particularly when $\dim B =1$. See the work of Traynor \cite{T01} and that of Jordan and Traynor \cite{T06}, which marks the beginning of the story. Moreover, in this 1-dimensional case it was shown by Fuchs and Rutherford \cite{FR11} that for each generating family $f$ as above there exists an augmentation $\ve$ of the Chekanov-Eliashberg dga such that the generating family homology $GH_*(\Lambda , f)$ is isomorphic to the linearized Legendrian contact homology $LCH^\ve_*(\Lambda)$. Under this isomorphism, Alexander duality for generating family homology corresponds to Sabloff duality for linearized contact homology \cite{S06}. Generating family homology can also be used to construct invariants of families of Legendrians, as in the work of Sabloff and Sullivan \cite{SS16}. Finally, we mention the connection with rulings \cite{FI04}, \cite{S05}.

To elucidate the relation between generating families and augmentations when $\dim B =1$, Henry developed the notion of a Morse complex sequence \cite{H04}, building on unpublished work of Pushkar. To each generating family $f:W \to \bR$ there is an associated Morse complex sequence, which is given by the 1-parameter family of fibrewise Thom-Smale complexes $C_*(f_t)$ (what we call the Thom-Smale complex is often called the Morse complex) \footnote{See F. Laudenbach's justification in the introduction to \cite{La11} for calling it the Thom-Smale complex.}. This family of chain complexes experiences elementary row and column operations at the discretely many times $t \in B$ when $f_t$ has a handle slide. A Morse complex sequence is an algebraic gadget which behaves just like this family of Thom-Smale complexes, but is purely combinatorial.

Being purely combinatorial, Morse complex sequences are easier to work with than generating families. For example, Henry and Rutherford proved that Morse complex sequences are in bijective correspondence with augmentations of the Chekanov-Eliashberg dga \cite{HR15}. The precise relation between Morse complex sequences and actual generating families is more subtle and not yet fully understood, due to the presence of homotopical obstructions for the existence and uniqueness of generating families.
 
 The analogous notion to a Morse complex sequence when $\dim B =2$ is that of a Morse complex 2-family (MC2F), which was introduced by Rutherford and Sullivan \cite{RS18}. A MC2F is an algebraic gadget which mimics the 2-parametric family of Thom-Smale complexes associated to a generating family. Among other applications they proved that the existence of a MC2F is equivalent to the existence of an augmentation for the Chekanov-Eliashberg dga. The extension of Morse complex sequences and 2-families when $\dim B >2$ is yet to be worked out.
 
 \begin{remark} In some sense MC2Fs appear implicitly in our calculation, when we express Turaev torsion in terms of a monodromy of handle slides. However, we don't use this formalism. It would be interesting to know if the techniques of the present article can be recast in the language of MC2Fs.
 \end{remark} 
 
 There are other types of Legendrian invariants which one can extract from generating families, such as the spectral numbers introduced by Viterbo \cite{V92}. These are metric measurements rather than Legendrian isotopy invariants, and have applications to quantitative aspects of symplectic and contact geometry, such as dynamics. We do not discuss this further as it is not relevant to our present work. 
 

\subsection{Beyond trivial bundles}
 
The majority of the literature on applications of the generating family construction has focused on the case when $W$ is a trivial bundle and moreover when the fibre $F$ is Euclidean space. However, as it was shown in \cite{EG98} there is subtler information available if we look at more general fibre bundles. For example, consider the following notion.

\begin{definition}
A Legendrian $\Lambda \subset J^1(B)$ is globally linked if any $\Lambda' \subset J^1(B)$ which is Legendrian isotopic to $\Lambda$ intersects every fibre of the projection $J^1(B) \to B$.
\end{definition}

When $F$ is a closed manifold we have the following result, which is implicit in \cite{EG98}.

\begin{proposition}\label{globally linked}
Suppose that a Legendrian $\Lambda \subset J^1(B)$ is generated by a function $f:W \to \bR$ on a fibre bundle of closed manifolds $F \to W \to B$. Then $\Lambda$ is globally linked.
\end{proposition}

\begin{proof}
Let $\Lambda'$ be Legendrian isotopic to $\Lambda$. By Theorem \ref{Persistence Theorem} there is a generating family $g$ for $\Lambda'$ on a stabilization of $W$ which differs from the stabilization of $f$ only by a compactly supported function. Hence for every $b \in B$, it follows that up to a shift in grading $H_*(g_b<+ \infty, g_b<-\infty) \simeq H_*(f_b<+\infty , f_b<-\infty) \simeq H_*(F)$. Since $\text{crit}(g_b) = \Lambda' \cap J^1_b(B)$, we see that $\Lambda' \cap J^1_b(B) = \varnothing$ implies $H_*(g_b<+ \infty, g_b<-\infty)=0$, which is a contradiction with $H_*(F) \neq 0$. \end{proof}

\begin{remark}
A slight variation of this argument also proves that the front projection of any $\Lambda'$ which is Legendrian isotopic to $\Lambda$ must disconnect $+\infty$ from $- \infty$ in $B \times \bR$.
\end{remark}

\begin{remark} Since the proof is homological, the same reasoning can be used to deduce an analogous result in microlocal sheaf theory (which is a priori stronger since every generating family produces a complex of sheaves). Indeed, using the quantization theorem \cite{GKS12} we can deduce global linking of $\Lambda$ from the existence of a complex of sheaves $\cS$ on $B \times \bR$ whose microsupport is $\Lambda$ and whose restriction $\cS_+$ to $+\infty$ has a stalk which is homologically different to that of the restriction $\cS_-$ to $-\infty$. 
\end{remark}

\begin{remark}
The proof of Proposition \ref{globally linked} also implies the stronger conclusion that the cardinality of the intersection of $\Lambda'$ with any fibre $J^1_b(B)$ is bounded below by the minimal number of critical points of a function on $F \times \bR^{2n}$ which is a fibration outside of a compact set and is a finite distance from $||x||^2-||y||^2$ in the $C^0$ norm. For example when $F=S^1$ this number is $2$. 
\end{remark}

Since mesh Legendrians admit generating families on circle bundles, we deduce the following.

\begin{corollary}\label{corollary: mesh Legendrians are globally linked}
Mesh Legendrians are globally linked.
\end{corollary}
 
This result is already nontrivial even for the simplest mesh Legendrian $\Lambda_G \subset J^1(S^2)$, where $G$ is the graph consisting of one circular edge with no vertices, see Figure \ref{FigSimpleLink}. In this case $\Lambda_G= \Lambda_1 \cup \Lambda_2$ is a link of two spheres, which is generated by a function on the trivial circle bundle $S^2 \times S^1$, see Figure \ref{FigFuncOnCircle}. It is easy to check that $\Lambda_G$ is a smooth unlink, and in fact it is trivial as a formal Legendrian link. So even though $\Lambda_G$ is globally linked, it is not formally globally linked. Moreover, it seems to us that in order to prove that $\Lambda_G$ is globally linked using pseudo-holomorphic curves it is necessary to use relative $H_2$ coefficients in the Chekanov-Eliashberg dga, where the fundamental class of the base $Q=[S^2]$ plays a crucial role. 

\begin{figure}[htbp]
\begin{center}
\begin{tikzpicture}[scale=.75]
%
\begin{scope}
\coordinate (T1) at (0,2.5);
\coordinate (T11) at (1.5,2.5);
\coordinate (T12) at (2,1.5);
\coordinate (T13) at (2.3,1);
\coordinate (T14) at (2.6,.3);
\coordinate (T15) at (2,-.2);
\coordinate (T16) at (1.8,-1);
\draw[very thick,color=blue] (T1)..controls (T11) and (T12)..(T13)..controls (T14) and (T15)..(T16);
\coordinate (T3) at (0,1.5);
\coordinate (T31) at (1,1.5);
\coordinate (T32) at (1.2,.8);
\coordinate (T33) at (1.5,.4);
\coordinate (T34) at (1.7,0);
\coordinate (T35) at (2,-.4);
\coordinate (T36) at (1.8,-1);
\draw[very thick,color=blue] (T3)..controls (T31) and (T32)..(T33)..controls (T34) and (T35)..(T36);
\coordinate (T2) at (0,-2.5);
\coordinate (T21) at (1.5,-2.5);
\coordinate (T22) at (2,-1.5);
\coordinate (T23) at (2.3,-1);
\coordinate (T24) at (2.6,-.3);
\coordinate (T25) at (2,.2);
\coordinate (T26) at (1.8,1);
\draw[very thick,color=black] (T2)..controls (T21) and (T22)..(T23)..controls (T24) and (T25)..(T26);
\coordinate (T4) at (0,-1.5);
\coordinate (T41) at (1,-1.5);
\coordinate (T42) at (1.2,-.8);
\coordinate (T43) at (1.5,-.4);
\coordinate (T44) at (1.7,0);
\coordinate (T45) at (2,.4);
\coordinate (T46) at (1.8,1);
\draw[very thick,color=black] (T4)..controls (T41) and (T42)..(T43)..controls (T44) and (T45)..(T46);
\end{scope}
\begin{scope}[rotate=180]
\coordinate (T1) at (0,2.5);
\coordinate (T11) at (1.5,2.5);
\coordinate (T12) at (2,1.5);
\coordinate (T13) at (2.3,1);
\coordinate (T14) at (2.6,.3);
\coordinate (T15) at (2,-.2);
\coordinate (T16) at (1.8,-1);
\draw[very thick,color=black] (T1)..controls (T11) and (T12)..(T13)..controls (T14) and (T15)..(T16);
\coordinate (T3) at (0,1.5);
\coordinate (T31) at (1,1.5);
\coordinate (T32) at (1.2,.8);
\coordinate (T33) at (1.5,.4);
\coordinate (T34) at (1.7,0);
\coordinate (T35) at (2,-.4);
\coordinate (T36) at (1.8,-1);
\draw[very thick,color=black] (T3)..controls (T31) and (T32)..(T33)..controls (T34) and (T35)..(T36);
\coordinate (T2) at (0,-2.5);
\coordinate (T21) at (1.5,-2.5);
\coordinate (T22) at (2,-1.5);
\coordinate (T23) at (2.3,-1);
\coordinate (T24) at (2.6,-.3);
\coordinate (T25) at (2,.2);
\coordinate (T26) at (1.8,1);
\draw[very thick,color=blue] (T2)..controls (T21) and (T22)..(T23)..controls (T24) and (T25)..(T26);
\coordinate (T4) at (0,-1.5);
\coordinate (T41) at (1,-1.5);
\coordinate (T42) at (1.2,-.8);
\coordinate (T43) at (1.5,-.4);
\coordinate (T44) at (1.7,0);
\coordinate (T45) at (2,.4);
\coordinate (T46) at (1.8,1);
\draw[very thick,color=blue] (T4)..controls (T41) and (T42)..(T43)..controls (T44) and (T45)..(T46);
\end{scope}
\coordinate (X) at (0,-3);
\coordinate (Y) at (0,3);
\draw[thick,dotted] (X)--(Y);
%
\end{tikzpicture}
\caption{Spin this figure about the dotted line to get (the front projection of) the mesh Legendrian corresponding to the graph with one circular edge and no vertices.}
\label{FigSimpleLink}
\end{center}
\end{figure}
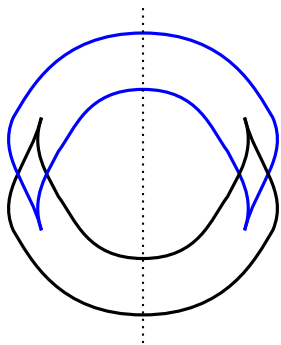
\begin{figure}[htbp]
\begin{center}
\begin{tikzpicture}
\coordinate (A1) at (-1.45,.3);
\coordinate (A11) at (-1.7,.3);
\coordinate (A12) at (-1.3,.3);
\coordinate (A2) at (-1,1.15);
\coordinate (A22) at (-1.3,1.15);
%
\coordinate (C1) at (-1.2,-.4);
\coordinate (C11) at (-1.4,-.4);
\coordinate (C12) at (-1,-.4);
\coordinate (C2) at (-1,-1.15);
\coordinate (C22) at (-1.5,-1.15);

\begin{scope}
\clip (-1,-1.3) rectangle (-.4,1.3);
\draw[very thick] (-1,0) ellipse[x radius=.3cm,y radius =1.15cm]; 
\end{scope}
\draw[very thick] (C2)..controls (C22) and (A11)..(A1)..controls (A12) and (C11)..(C1)..controls (C12) and (A22)..(A2); 

\coordinate (B1) at (1.45,.35);
\coordinate (B11) at (1.7,.35);
\coordinate (B12) at (1.3,.35);
\coordinate (B2) at (1,1.15);
\coordinate (B22) at (1.3,1.15);
%
\coordinate (D1) at (1.2,-.4);
\coordinate (D11) at (1.4,-.4);
\coordinate (D12) at (1,-.4);
\coordinate (D2) at (1,-1.15);
\coordinate (D22) at (1.5,-1.15);

\begin{scope}
\clip (1,-1.3) rectangle (.4,1.3);
\draw[very thick] (1,0) ellipse[x radius=.3cm,y radius =1.15cm]; 
\end{scope}
\draw[very thick] (D2)..controls (D22) and (B11)..(B1)..controls (B12) and (D11)..(D1)..controls (D12) and (B22)..(B2); 

%
\coordinate (2K) at (-4.1,-1.2);
\draw(2K) node[left]{$y_i$};
\coordinate (2Kp) at (-4.1,1.2);
\draw(2Kp) node[left]{$x_i$};
\coordinate (2Kr) at (4.1,-1.2);
\draw[color=blue] (2Kr) node[right]{$y_j$};
\coordinate (2Kpr) at (4.1,1.2);
\draw[color=blue] (2Kpr) node[right]{$x_j$};
\clip (-4.4,-2) rectangle (4.4,1.5);
\begin{scope}[xshift=-3cm] 
\draw[ thick] (-1,1.25)--(1,1.25)..controls (2.8,1.25) and (4.8,0)..(5.8,0);
\draw[ thick] (-1,-1.25)--(1,-1.25)..controls (3,-1.25) and (4.5,0)..(5.8,0);
\draw[very thick] (-1,0) ellipse[x radius=.3cm,y radius =1.25cm]; 
\begin{scope}[xshift=-.5cm]
\clip (1,-1.3) rectangle (1.4,1.3);
\draw[very thick] (1,0) ellipse[x radius=.3cm,y radius =1.25cm]; 
\end{scope}
\draw[very thick,xshift=-.5cm] (1,-1.25)..controls (.8,-1.25) and (.4,0)..(.7,0)..controls (1,0) and (.8,1.25)..(1,1.25); 
\end{scope}
\begin{scope}[xshift=1cm] 
\draw[ thick, color=blue] (-3.8,0)..controls (-2.3,0) and (-1.3,1.25)..(1,1.25)--(3,1.25);
\draw[ thick, color=blue] (-3.8,0)..controls (-2.2,0) and (-1.3,-1.25)..(1,-1.25)--(3,-1.25);
\draw[very thick] (3,0) ellipse[x radius=.3cm,y radius =1.25cm]; 
\end{scope}
\begin{scope}[xshift=1.5cm]
\clip (1,-1.3) rectangle (.6,1.3);
\draw[very thick] (1,0) ellipse[x radius=.3cm,y radius =1.25cm]; 
\end{scope}
\draw[very thick,xshift=3.5cm] (-1,-1.25)..controls (-.8,-1.25) and (-.4,0)..(-.7,0)..controls (-1,0) and (-.8,1.25)..(-1,1.25); 

\end{tikzpicture}
\caption{A one parameter family of functions on the circle generating the Cerf diagram of Figure \ref{Fig: link}.}
\label{FigFuncOnCircle}
\end{center}
\end{figure}
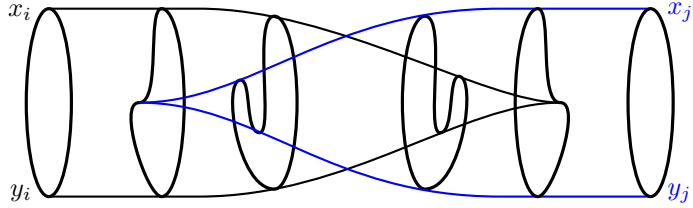
%


Consider next a fibre bundle where not only the fibre is nontrivial, but the bundle itself is also nontrivial. For circle bundles over surfaces $S^1 \to E \to \Sigma$ the theory is already quite rich, as we will see below. A natural question one can ask is the following.
\begin{question}\label{Question: remember} To what extent does a Legendrian $\Lambda \subset J^1(\Sigma)$ which is generated by a family on a circle bundle remember the circle bundle? 
\end{question}
For example, in \cite{DR11} Dimitroglou-Rizell considered the Legendrian sphere $\Lambda_h \subset J^1(S^2)$ generated by the height function $h:S^3 \to \bR$, where we think of $S^3$ as the Hopf bundle over $S^2$. \begin{remark} The front $\pi(\Lambda_h) \subset J^0(S^2)$ consists of two parallel copies of the zero section $S^2$, which are joined together by a conical singularity $z^2=x^2+y^2$ (see Figure \ref{Fig: cone}) over a single point $(x,y) \in S^2$. This caustic is not generic: after perturbation it decomposes into four swallowtails with arcs of cusps between them (see Figure \ref{FigDR}). 
\end{remark}
Dimitroglou-Rizell made the following observation.

\begin{figure}[htbp] 
\begin{center}
\begin{tikzpicture}[scale=.7]

\draw[very thick] (0,2) ellipse[x radius=15mm, y radius=5mm];

\begin{scope}
\draw[very thick] (0,-2) ellipse[x radius=15mm, y radius=5mm];
\draw[fill, color=white] (-2,-1.9) rectangle (2,0);
\draw[very thick,dotted] (0,-2) ellipse[x radius=15mm, y radius=5mm];
\end{scope}

\draw[very thick] (1.47,1.9)--(-1.47,-1.9);
\draw[very thick] (-1.47,1.9)--(1.47,-1.9);

\end{tikzpicture}
\caption{A cone singularity}
\label{Fig: cone}
\end{center}
\end{figure}
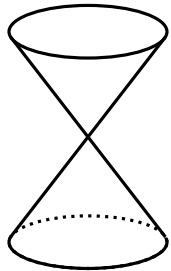
\begin{figure}[htbp]
\begin{center}
\begin{tikzpicture}[scale=.75]
\coordinate (A) at (-1.3,0);
\coordinate (B) at (1.3,0);
\coordinate (B3) at (1.5,.3);
\coordinate (B4) at (1.8,.5);
\coordinate (C) at (0,.35);
\coordinate (C2) at (0,-.3);
\coordinate (D) at (-1.2,1.5);
\coordinate (D2) at (2,.7);
\coordinate (D3) at (0,.8);
\coordinate (D4) at (1.2,.4);
	\draw[dashed] (A)..controls (-1,.5) and (-1,1)..(D);
	\draw[dashed] (D)..controls (D3) and (D4)..(D2);
\begin{scope}
	\clip (1,.5) rectangle (2.1,1);
	\draw[very thick] (D)..controls (D3) and (D4)..(D2);
\end{scope}
	\draw[very thick] (B)..controls (B3) and (B4)..(D2);
\begin{scope}
	\clip (-2.1,-.1) rectangle (-1,.7);
	\draw[very thick] (A)..controls (-1,.5) and (-1,1)..(D);
\end{scope}
\draw[very thick] (-2,.9)--(-3,-1)..controls (-1,-1) and (0,0)..(B);
\draw[very thick] (3,1)--(4,2.5)--(-2,2.5)--(-3,1)..controls (-1,1) and (0,0)..(B);
\draw[very thick] (3.1,1.2)--(4,1.2)--(3,-1)..controls (1,-1) and (0,0)..(A);
\draw[very thick] (3,1)..controls (1,1) and (0,0)..(A);
\draw[thick] (C)--(D);
	\draw[dashed] (C2)--(D2);
\begin{scope}
	\clip (0,-.5)rectangle (.8,.3);
	\draw[thick] (C2)--(D2);
\end{scope}
\draw[<-] (0,-1.5)--(0,-.9);
\draw (0,-1.2) node[right]{$p$};
\begin{scope}[yshift=-4cm]
\coordinate (A) at (-1.3,0);
\coordinate (A1) at (-2.5,-1);
\coordinate (A2) at (-2,-.7);
\coordinate (A3) at (-1.6,-.5);
\coordinate (B) at (1.3,0);
\coordinate (B0) at (1,-1.3);
\coordinate (B1) at (1,-1);
\coordinate (B2) at (1.1,-.2);
\coordinate (B3) at (1.5,.3);
\coordinate (B4) at (1.8,.5);
\coordinate (C) at (0,.35);
\coordinate (C2) at (0,-.3);
\coordinate (D) at (-1.2,1.5);
\coordinate (D2) at (2,.7);
\coordinate (D3) at (0,.8);
\coordinate (D4) at (1.2,.4);
\coordinate (X1) at (-4,-2);
\coordinate (X2) at (-2,2);
\coordinate (X3) at (4,2);
\coordinate (X4) at (2,-2);
\draw[very thick] (X1)--(X2)--(X3)--(X4)--(X1);
\draw[thick,dotted] (-3,0)--(3,0);
	\draw[thick] (A1)..controls (A2) and (A3)..(A)..controls (-1,.5) and (-1,1)..(D);
	\draw (A1)--(D2) (D)--(B0);
	\draw[thick] (D)..controls (D3) and (D4)..(D2);
	\draw[thick] (A1)..controls (-1,-.7) and (0,-.7)..(B0)..controls (B1) and (B2)..(B)..controls (B3) and (B4)..(D2);
\end{scope}
\end{tikzpicture}
\caption{The Legendrian of Dimitroglou-Rizell cut over dotted line to show details.}
\label{FigDR}
\end{center}
\end{figure}
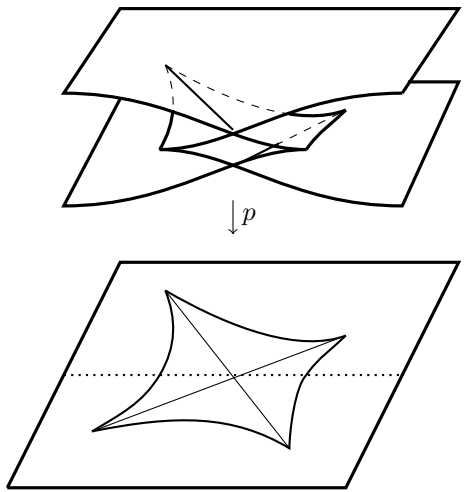

\begin{proposition}
The Legendrian $\Lambda_h$ cannot be generated by a family on the trivial bundle $S^2 \times S^1$.
\end{proposition}
 
The proof uses the Morse inequalities for $S^2 \times S^1$, from which the conclusion is immediate. The same proof works stably, with $S^2 \times S^1 \times \bR^{2k}$ instead of $S^2 \times S^1$.  Although $\Lambda_h$ is not a mesh Legendrian, our work on mesh Legendrians is in a similar spirit. Indeed, Theorem \ref{euler} answers Question \ref{Question: remember} for the class of mesh Legendrians, and even stably. However, in order to distinguish different mesh Legendrians which are generated by families on the same circle bundle we must look at somewhat finer invariants, namely those coming from parametrized Morse theory instead of just Morse theory. 

Finally, we mention related work of Sullivan and Rutherford \cite{RS18} who used MC2Fs together with their cellular technology for the Chekanov-Eliashberg dga to produce examples of Legendrians in a 1-jet space of a surface $B$ which do not admit generating families on trivial fibre bundles. Their obstruction comes from expressing the action of $\pi_1(B)$ on the homology of the fibre $H_*(F)$ in terms of the Chekanov-Eliashberg dga and its augmentations. 
 
\section{Turaev torsion}\label{Section: turaev torsion}
 
\subsection{Torsion from the Morse-theoretic viewpoint}\label{reidemeistertorsion}
 
 Fix a commutative ring $R$. Let $M$ be a closed, connected, orientable manifold and let $\rho: \pi_1M \to U(R)$ be a representation, where $U(R)$ is the group of units of $R$. Observe that as far as $\rho$ is concerned we do not need to specify the basepoint for $\pi_1M$. Indeed, the ambiguity in the basepoint corresponds to conjugation in $\pi_1 M$ and $\rho$ takes values in the abelian group $U(R)$.
 
We recall how to compute the cohomology $H^*(M; R^\rho)$ of $M$ with coefficients in $R$ twisted by the local system $\rho$ in terms of a Morse function $f:M \to \bR$. Fix an orientation of $M$ and an orientation for the negative eigenspace of the Hessian $d^2f$ at each critical point of $f$. Fix also a vector field $Z$ on $M$ which is gradient-like for $f$. We assume that $Z$ is Morse-Smale for $f$, which means that the $Z$ stable and unstable manifolds of any two critical points of $f$ intersect transversely on any regular level set of $f$. Fix a basepoint $m_0 \in M$, which we use to identify $\pi_1M=\pi_1(M,m_0)$, and choose homotopy classes of paths $p_x$ in $M$ from the basepoint $m_0$ to each critical point $x$ of $f$. 
 
The Thom-Smale complex $C_*(f ; \bZ[\pi_1M]$) is freely generated in degree $k$ over $\bZ[\pi_1M]$ by the critical points of $f$ of index $k$, where $0 \leq k \leq \dim M$. For $x$ a critical point of index $k$, the differential $\partial x = \sum_y  a_{xy} y$ is a weighted sum of critical points $y$ of index $k-1$, where the coefficient $a_{xy} \in \bZ[\pi_1M]$ is given as follows. Because of the Morse-Smale condition, there are only finitely many $Z$ trajectories between $x$ and $y$. We can cap off each such trajectory $\gamma$ using the paths $p_x$ and $p_y$ to obtain an element $\widetilde \gamma \in \pi_1M$. Moreover, the intersection of $\gamma$ with an intermediate regular level set for $f$ corresponds to an intersection point between the $Z$ stable manifold for $x$ and the $Z$ unstable manifold for $y$ in that level. Since the former is oriented and the latter is co-oriented, we get a sign $\ve_\gamma \in \{\pm 1\}$ for each trajectory $\gamma$. The coefficient is then given by the formula $a_{xy} = \sum_\gamma \ve_\gamma \widetilde \gamma  \in \bZ[\pi_1M]$.

\begin{remark}\label{second remark about IK conventions}
Following the conventions of \cite{IK93}, we will write the matrix of the boundary map $\partial :C_k(f ; \bZ[\pi_1M]) \to C_{k-1}(f ; \bZ[\pi_1 M])$ as $d_{yx}=a_{xy}$, i.e. we transpose the matrix $(a_{xy})$ so that these matrices are composed right to left. Hence the entries of the boundary map should really be taken in the opposite ring $\bZ[\pi_1M]^{op}$. Since we will later pass to the commutative ring $R$, this will not matter.
\end{remark}

That $\partial^2=0$, or, equivalently, that $\sum_y a_{xy} a_{yz}=0$, follows as usual from considering the boundary of the 1-dimensional moduli spaces of $Z$ trajectories between critical points of index difference $2$. As a $\bZ$-module, the homology of $(C_*(f ; \bZ[\pi_1M]), \partial)$ is isomorphic to the singular cohomology of the universal covering space $\widetilde M$ over $\bZ$ . However, we are interested not in the Thom-Smale complex $(C_*(f ; \bZ[\pi_1M]) , \partial)$ but in the twisted Thom-Smale complex $(C_*(f ; R^\rho), \partial^\rho)$. This complex is freely generated in degree $k$ over $R$ by the critical points of index $k$ and its differential $\partial^\rho$ is obtained from $\partial$ by applying the map on coefficients $\bZ[\pi_1M] \to R$ induced from $\rho$. The homology of $(C_*(f; R^\rho),\partial^\rho)$ is precisely $H^*(M;R^\rho)$. 

Assume now that $H^*(M ; R^\rho)=0$, so that the twisted Thom-Smale complex $C_*(f ; R^\rho)$ is acyclic. In this case there exists a chain contraction $\delta^\rho$ for $\partial^\rho$, i.e. a chain homotopy between the identity and zero. It is readily seen that $\partial^\rho+\delta^\rho : C_{\text{odd}} (f ;R^\rho) \to C_{\text{even}}(f ; R^\rho)$ is an isomorphism of finite rank free $R$-modules. Fixing an ordering of the critical points of even and odd index we can therefore represent $\partial^\rho+\delta^\rho$ by an invertible matrix over $R$. Consider the determinant $\det(\partial^\rho+\delta^\rho) \in U(R)$ of this matrix.

First of all, this determinant is only defined up to sign, since a permutation of the chosen orderings for the critical points will change the determinant by $\pm1 $. Next, changing the choices of orientations for the negative eigenspaces of $d^2f$ at the critical point of $f$ will also change the determinant by $\pm 1$, and similarly if we change the chosen orientation for $M$. Something somewhat more dramatic happens if we change one of the paths from the basepoint $m_0\in M$ to a critical point of $f$ by an element $\alpha \in \pi_1M$. The effect on the determinant $\det(\partial^\rho + \delta^\rho)$ is that it gets multiplied by $\rho(\alpha) \in U(R)$ or $\rho(\alpha)^{-1} \in U(R)$, depending on the parity of the index of the critical point. 

However, changing the gradient-like Morse-Smale vector field $Z$ will not affect the determinant. Indeed, any two such $Z_0$ and $Z_1$ can be joined by a family $Z_t$ which is Morse-Smale at all but finitely many times $t \in [0,1]$, at which the Thom-Smale complex experiences handle slides. These bifurcations correspond to connecting trajectories between critical points of the same index. They have the algebraic effect of an elementary row/column operation on the matrix for $\partial^\rho+\delta^\rho$, which therefore does not affect the determinant. Finally, a homological algebra argument shows that the choice of chain contraction $\delta^\rho$ also does not affect the determinant. We conclude that $\text{det}(\partial^\rho+\delta^\rho)$ is well defined up to multiplication by $\pm$ an element in $\rho(\pi_1M) \subset U(R)$. 

\begin{definition}\label{reidemeisterdef}
The Reidemeister torsion of $M$ with respect to the function $f:M \to \bR$ and the local system $\rho: \pi_1M \to U(R)$ is $r(f, \rho) = \text{det}(\partial^\rho+\delta^\rho) \in U(R) / \pm \rho( \pi_1M )$. 
\end{definition}

\begin{remark}
The above definition is the multiplicative form of the Reidemeister torsion. When $R=\bC$, other common definitions in the literature include the absolute value $|\text{det}(\partial^\rho+\delta^\rho)| \in \bR_{>0}$ and its additive form $\log |\text{det}(\partial^\rho+\delta^\rho)| \in \bR$. We will only use the multiplicative form \ref{reidemeisterdef}.
\end{remark}

For now we are taking $M$ to be a closed manifold, in which case the Reidemeister torsion $r(\rho)=r(f, \rho)$ is also independent of the function $f:M \to \bR$. The reason is that any two Morse functions $f_0$ and $f_1$ can be joined by a family $f_t$ which is Morse at all but finitely many times $t \in [0,1]$, at which $f_t$ will experience a Morse birth/death bifurcation. This occurs when two Morse (quadratic) critical points die or are born together at a (cubic) critical point. For times close to the instant of birth/death we can assume that there is a unique trajectory connecting the two critical points and no other trajectories connecting them to other critical points. Therefore the algebraic effect of the birth/death is to add or remove a row and column to $\partial^\rho+\delta^\rho$ with zeros everywhere except at the diagonal entry, which is an element of $\pm \rho( \pi_1M)$. Hence the determinant of $\partial^\rho+\delta^\rho$ gets multiplied or divided by that same element. Additionally, the Thom-Smale complex might experience further handle slides throughout the homotopy $f_t$, but as above this does not affect the determinant. We conclude that the Reidemeister torsion does not depend on $f$. 

To be precise, the `independence of birth/deaths' needed so that the algebraic effect of a birth/death point on the Thom-Smale complex is exactly as described above can only be ensured when the dimension of $M$ is sufficiently large. Otherwise there may not be enough room to guarantee the existence of a gradient like $Z$ for which there are no connecting trajectories between the birth/death point and the other critical points. However, this can be solved as follows. 

As we will see below, torsion invariants are stable, as long as we always stabilize an even number of times. This means that they do not change when we replace $M$ with $M \times \bR^4$ and $f$ with $f +x_1^2+x_2^2 - y_1^2- y_2^2$. Under an odd stabilization, the determinant of $\partial^\rho+\delta^\rho$ will be inverted, but under an even stabilization it remains unchanged. Hence by first stabilizing a sufficiently large even number of times, we can make the necessary room to achieve the desired `independence of birth/deaths'. See Section \ref{Section: stability of torsion} for further discussion.



\subsection{Euler structures and Turaev torsion}
Following Turaev, we can lift the Reidemeister torsion to a finer invariant by choosing an Euler structure.

\begin{definition} An Euler structure for a Morse function $f:M \to \bR$ consists of a bijection between the sets of critical points of odd and even indices such that corresponding critical points have index difference 1, together with a homotopy class of paths connecting every such pair of critical points.
\end{definition}

\begin{remark}
Call two Euler structures $e_1$ and $e_2$ for $f$ equivalent if the 1-cycle of formal path differences $e_1-e_2$ is zero in $H_1(M)$. For a closed 3-manifold $M$ the set $\text{Eul}(M)$ of equivalence classes is in natural bijection with the set of $\text{spin}^c$ structures on $M$.
\end{remark}

We say that a collection of paths $p_x$ from the basepoint $m_0\in M$ to the critical points $x$ of $f$ is compatible with an Euler structure if for any $x$ and $y$ critical points paired by the Euler structure the composition $p_x \# p_{xy} \# \overline{p}_y $ is a null homotopic loop, where $\overline{p}_y$ denotes $p_y$ but with reversed orientation and $p_{xy}$ is the path from $x$ to $y$ determined by the Euler structure. See Figure \ref{Fig: Euler}. Consider as before the determinant $\det(\partial^\rho + \delta^\rho)$, where we now use a choice of paths compatible with the Euler structure to construct the twisted Thom-Smale complex $C_*(f; R^\rho)$. 

\begin{figure}[htbp] 
\begin{center}
\begin{tikzpicture}[scale=.7]

\coordinate (X) at (1.8,2);
\coordinate (Xm) at (1.75,1.85);
\coordinate (Y) at (1,0);
\coordinate (Ym) at (1,0.15);
\coordinate (m0) at (-2,.8);
\coordinate (mp) at (-1.9,.8);
\coordinate (mpp) at (-1.85,.75);
\coordinate (Zp) at (1.8,1);
\coordinate (Z) at (1.7,.8);
\coordinate (Zm) at (1.6,.6);

\draw (.1,1.7) node{$p_x$} (-.5,-.2) node{$\overline{p}_y$};

\draw[semithick, ->] (X)..controls (2,1.5) and (Zp)..(Z);
\draw[semithick] (Z)..controls (Zm) and (1.3,.4)..(Y) (Z) node[right]{$p_{xy}$};
\draw (m0) node{$\ast$};
\draw (m0) node[left]{$m_0$};

\draw[semithick,->] (mp) ..controls (-1.4,2) and (-1,1.5)..(-.2,1.4)..controls (.6,1.3) and (1.3,1.4)..(Xm);
\draw[semithick,<-] (mpp) ..controls (-1.4,0) and (-1,0)..(0,0.3)..controls (.7,.6) and (1.1,1)..(Ym);
\begin{scope}[xshift=1.5cm]
\draw[thick] (0.3,2)..controls (.9,2) and (1.1,2.5)..(1.2,3);
\draw[thick] (0.3,2)..controls (-.4,2) and (-.6,2.5)..(-.7,3);
\end{scope}
\draw[thick] (1,0)..controls (1.6,0) and (1.8,-.5)..(1.9,-1);
\draw[thick] (1,0)..controls (.4,0) and (.2,-.5)..(.1,-1);

\draw (X) node{$\bullet$} node[above]{$x$};
\draw (Y) node{$\bullet$} node[below]{$y$};
\end{tikzpicture}
\caption{The paths $p_x$ from $m_0$ to each critical point $x$ are compatible with the Euler structure if each composition $p_x\# p_{xy}\# \overline p_y$ is null homotopic.}
\label{Fig: Euler}
\end{center}
\end{figure}
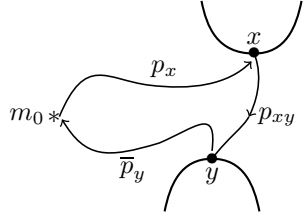

\begin{lemma}\label{indet}
The determinant $\tau(f,e,\rho)=\det(\partial^\rho + \delta^\rho)$ is a unit in $R$ well defined up to sign. It only depends on the function $f$, the Euler structure $e$ and the local system $\rho$. In particular it is independent of the choice of compatible paths and of any other choices.
\end{lemma}

\begin{proof}
Let $A$ be the matrix representing $\partial^\rho + \delta^\rho$, where for simplicity we order the rows and columns so that the critical point corresponding to the $k$th column is paired by the Euler structure with the critical point corresponding to the $k$th row. Then a different choice of compatible paths will conjugate the matrix $A$ by a diagonal matrix, hence will not change the determinant. This removes the $\rho(\pi_1M)$ ambiguity in Definition \ref{reidemeisterdef}. 
\end{proof}

\begin{definition}
The Turaev torsion of the function $f$ with respect to the Euler structure $e$ and the local system $\rho$ is $\tau( f,e, \rho)=\det(\partial^\rho + \delta^\rho) \in U(R) / \pm 1$.
\end{definition}

Any Turaev torsion $\tau(f, e, \rho)$ maps to the Reidemeister torsion $r(f, \rho)$ under the map $U(R) / \pm 1 \to U(R) / \pm \rho( \pi_1M)$. In this sense Turaev torsion is a refinement of Reidemeister torsion.


Let $f_t$ be a 1-parameter family of functions which is Morse at all but finitely many times $t \in [0,1]$, at which $f_t$ has a Morse birth/death bifurcation. 

\begin{definition}\label{Def: euler compatible} We say that a homotopy of functions $f_t$ is compatible with $e_t$ a family of Euler structures for $f_t$ if $e_t$ is continuous in the obvious sense when $f_t$ is Morse and such that two critical points of $f_t$ can only die or be born together if the following conditions hold.
\begin{itemize}
\item[(a)] The critical points are paired by $e_t$.
\item[(b)] At the instant of birth/death the $e_t$ path between them is a contractible loop.
 \end{itemize}
 \end{definition}

 \begin{lemma}\label{inv}
 Under conditions (a), (b) we have $\tau(f_0,e_0, \rho)=\tau(f_1,e_1, \rho)$ in $U(R) / \pm1$.
 \end{lemma}
 
 \begin{proof}
 Outside of the birth/death points the matrix for $\partial^\rho + \delta^\rho$ only changes by the row/column operations corresponding to handle slides, which do not affect the determinant. At a birth/death point the matrix $A$ for $\partial^\rho + \delta^\rho$ as above changes by adding or removing a row and column with the only nonzero entry being a $\pm 1$ in the diagonal.  \end{proof}
 
If we additionally fix an orientation for $M$ and an orientation for the negative eigenspaces of $d^2f$ at each critical point of $f$, then we can remove the sign ambiguity in the definition of Turaev torsion to obtain a well defined element of $ U(R)$. In the language of Turaev this data corresponds to a homology orientation and the resulting torsion is called sign-determined. In our intended applications we will not be able to remove the sign ambiguity, so our Turaev torsions will only be defined up to sign. On the flip side, allowing for the sign ambiguity has the advantage that we can work with a weaker version of an Euler structure, which will turn out to be very useful.

\begin{definition}\label{Definition: weak} A weak Euler structure for $f$ is a partition of the set of its critical points such that each part has the same number of critical points of even and odd index, together with a homotopy class of paths connecting any two points in the same part satisfying that any composable triangle of paths gives a null homotopic loop.
\end{definition}

\begin{remark} When each part has exactly two points we recover the definition of an Euler structure. 
\end{remark}

We say that a collection of paths $p_x$ from the basepoint $m_0\in M$ to the critical points $x$ of $f$ is compatible with the weak Euler structure if for $x$ and $y$ critical points in the same part the composition $p_x \# p_{xy} \# \overline{p}_y $ is a null homotopic loop, where $\overline{p}_y$ denotes $p_y$ but with reversed orientation and $p_{xy}$ is the path from $x$ to $y$ determined by the weak Euler structure. Consider the determinant $\det(\partial^\rho + \delta^\rho)$ as before, where we use a choice of paths compatible with the weak Euler structure to construct the twisted Thom-Smale complex $C_*(f; R^\rho)$. We have the following analogue of Lemma \ref{indet}, which is proved in exactly the same way.

\begin{lemma}\label{weakindet}
The determinant $\tau(f,e,\rho)=\det(\partial^\rho + \delta^\rho)$ is a unit in $R$ well defined up to sign. It only depends on the weak Euler structure for $f$ and the local system $\rho$. In particular it is independent of the choice of compatible paths and of any other choices.
\end{lemma}

Therefore we can define the Turaev torsion of a weak Euler structure, just like for an Euler structure.

\begin{definition}
The Turaev torsion of the function $f$ with respect to the weak Euler structure $e$ and the local system $\rho$ is $\tau(f, e, \rho)=\det(\partial^\rho + \delta^\rho) \in U(R) / \pm 1$.
\end{definition}


Let $f_t$ be a 1-parameter family of functions which is Morse at all but finitely many times $t \in [0,1]$, at which $f_t$ has Morse birth/death bifurcations. The weak analogue of Definition \ref{Def: euler compatible} is the following.

\begin{definition}\label{Def: weak euler compatible} We say that a homotopy $f_t$ is compatible with $e_t$ a family of a weak Euler structures for $f_t$ if $e_t$ is continuous in the obvious sense when $f_t$ is Morse and such that two critical points of $f_t$ can only die or be born together if the following conditions hold.

\begin{itemize}
\item[(a)] The critical points are in the same part determined by $e_t$.
\item[(b)] At the instant of birth/death the $e_t$ path between them is a contractible loop.
 \end{itemize}
 \end{definition}
 
The following analogue of Lemma \ref{inv} for weak Euler structures is proved in exactly the same way.
\begin{lemma}\label{weakinv}
 Under conditions (a)$\,$ and (b)$\,$ we have $\tau(f_0,e_0, \rho)=\tau(f_1,e_1, \rho)$ in $U(R) / \pm1$.
\end{lemma}

\subsection{Fibre Turaev torsion}
 
 We now introduce a version of Turaev torsion for fibre bundles $F \to W \to B$. The output will be a Turaev torsion of the fibre $F$. We start with a definition, where we abuse notation to use the same symbol $\rho$ to denote both a local system $\pi_1 W \to U(R)$ and its pullback to $\pi_1F$ by the inclusion $F \subset W$.
 
 \begin{definition}\label{torsionpair}
A torsion pair $(\pi, \rho)$ consists of a fibre bundle of manifolds $\pi: W \to B$ with fibre $F$ and a representation $\rho: \pi_1W \to U(R)$ such that the following properties hold.
\begin{itemize} 
\item $B$, $F$ and $W$ are closed, connected and orientable.
\item $H^*(F; R^\rho)=0$.
\item $\rho(\pi_1F)=\rho(\pi_1W)$.
\end{itemize}
\end{definition}
 
Let $(\pi,\rho)$ be a torsion pair. The Reidemeister torsion of the fibre $F$ with respect to $\rho$ is well defined and is independent of the fibre. Given a function $f:W \to \bR$, there is sometimes a preferred class of weak Euler structures for the restriction of $f$ to the generic fibre. When this is the case it is possible to refine the Reidemeister torsion of the fibre to a Turaev torsion. 

Consider $C=\{ \partial_F f = 0 \} \subset W$ the fibrewise critical set of $f:W \to \bR$. For generic $f$ we have $\partial_F f \pitchfork 0$ and hence the fibrewise critical set is a smooth orientable submanifold of $W$ of the same dimension as $B$. Generically, the fibre $F_b$ of $W \to B$ over $b \in B$ intersects $C$ transversely, in which case the restriction $f_b: F_b \to \bR$ of $f$ is Morse.  

\begin{definition}
A weak Euler structure for $f_b:F_b \to \bR$ is said to be $\rho$-compatible with $f$ if the following conditions hold. (See Figure \ref{Fig: rho-compatible weak Euler structure}.)
\begin{itemize}
\item[(1)] The partition of $\text{crit}(f_b)$ determined by the Euler structure is given by the connected components $C_i$ of $C$. In other words, the partition is $\cP= \{ C_i \cap F_b \}_i$.
\item[(2)] Let $\gamma= \alpha \# \beta$ be the composition of the path $\alpha$ in $F_b$ between two critical points $x,y$ of $f_b$ determined by the weak Euler structure and a path $\beta$ in $W$ from $y$ to $x$ which is contained in $C_i$. Then the loop $\gamma \subset W$ lies in the kernel of $\rho$.
\end{itemize}
\end{definition}

In general, $\rho$-compatible weak Euler structures for a given $f$ may not exist. However, as we will see in Lemma \ref{familyturaev} below, the following condition on $f$ ensures their existence.

\begin{definition}\label{Def: function euler}
We say that a function $f:W \to \bR$ is of Euler type if $\partial_F f \pitchfork 0$ and each of the connected components $C_i$ of the fibrewise critical set $C$ is simply connected and projects down the base as a degree zero map $C_i \to B$. 
\end{definition}

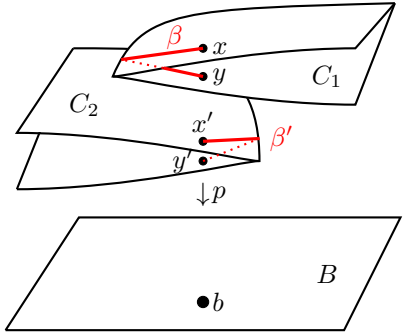
\begin{figure}[htbp]
\begin{center}
\begin{tikzpicture}[scale=.75]
\coordinate (G) at (-2,.5);
\begin{scope}
\clip (-3,-1) rectangle (-2,-2.1);
	\draw[ thick,yshift=-1.5cm]  (G)--(-3,-.5);
\end{scope}
\draw[ thick,yshift=-1.5cm] (G)--(-3,.5)..controls (-1,.5) and (1,0)..(1.3,0);
\draw[ thick,yshift=-1.5cm] (-3,-.5)..controls (-1,-.5) and (1,0)..(1.3,0)..controls (1.3,2) and (0,2)..(G);
%
\coordinate (R) at (3.7,1.3);
\draw[fill,color=white] (R)--(3,-.5)..controls (1,-.5) and (-1,0)..(-1.3,0)..controls (-.8,1.2) and (0,1.2)..(R);
\draw[ thick] (R)--(3,-.5)..controls (1,-.5) and (-1,0)..(-1.3,0)..controls (-.8,1.2) and (0,1.2)..(R);
\draw[ thick] (R)--(3,.5)..controls (1,.5) and (-1,0)..(-1.3,0);

\begin{scope}[yshift=-2.5cm, xshift=.8cm]
\coordinate (X1) at (-4,-2);
\coordinate (X2) at (-2.7,0);
\coordinate (X3) at (3,0);
\coordinate (X4) at (2,-2);
\coordinate (b) at (-.5,-1.5);
\coordinate (x1) at (-.5,3);
\coordinate (x2) at (-.5,2.5);
\coordinate (x3) at (-1.95,2.8);
\coordinate (x23) at (-1.2,2.65);
\coordinate (x4) at (-1,3.2);
\coordinate (y1) at (-.5,1.35);
\coordinate (y2) at (-.5,1);
\coordinate (y3) at (-.5,.4);
\coordinate (y4) at (.5,1.4);
\draw (y3) node{$\downarrow$} node[right]{$p$};
\foreach\x in {x1,x2,y1,y2}
	\draw[fill] (\x) circle[radius=.7mm];
	\draw[very thick, color=red] (y1)--(y4) node[right]{$\beta'$};
	\draw[thick, color=red, dotted] (y2)--(y4);
\draw (x1) node[right]{$x$};
\draw (x2) node[right]{$y$};
\draw (y1) node[above]{$x'$};
\draw (y2) node[left]{$y'$};
\draw[very thick, color=red] (x23)--(x2) (x1)--(x3) (x4) node{$\beta$};
\draw[dotted,thick, color=red] (x23)--(x3);
\coordinate (B) at (1.7,-1);
\coordinate (C1) at (1.7,2.5);
\coordinate (C2) at (-2.6,2);
\draw (C1) node{$C_1$};
\draw (C2) node{$C_2$};
\draw[fill] (b) circle[radius=1mm];
\draw (b) node[right]{$b$} (B) node{$B$};
\draw[ thick] (X1)--(X2)--(X3)--(X4)--(X1);
\end{scope}
\end{tikzpicture}
\caption{Given that $C_1,C_2$ are simply connected, a $\rho$-compatible weak Euler structure on $F_b$ is given by choosing a path $\alpha$ from $x$ to $y$ so that $\gamma=\alpha\#\beta$ is in the kernel of $\rho$,  where $\beta$ is any path from $x$ to $y$ in $C_1$ and similarly for $x',y'$.}
\label{Fig: rho-compatible weak Euler structure}
\end{center}
\end{figure}

\begin{lemma}\label{familyturaev}
Suppose that $(\pi, \rho)$ is a torsion pair and that $f:W \to \bR$ is of Euler type. Then there exists a $\rho$-compatible weak Euler structure $e_b$ on the generic fibre $F_b$. Moreover, the Turaev torsion $\tau(f_b,e_b, \rho)$ is independent of both the fibre $F_b$ and of the $\rho$-compatible weak Euler structure $e_b$.
\end{lemma}

\begin{proof} We can compute the degree of $C_i \to B$ by taking preimages over $b \in B$. On the generic fibre we have $F_b \pitchfork C_i$, hence each preimage contributes $\pm 1$ to the degree. Moreover, each preimage is a Morse critical point of $f_b$ and the sign $\pm 1$ depends only on the parity of the Morse index. Therefore the assumption that $C_i \to B$ has degree zero implies that $f_b$ has the same number of critical points of even and odd index in the component $C_i$. 

Fix $x$ a critical point of $f_b$ in the component $C_i$. For every other critical point $y$ of $f_b$ in the component $C_i$ take any path $\alpha$ in $F_b$ between $x$ and $y$. Since $C_i$ is simply connected, there is a unique path $\beta$ in $C_i$ from $y$ to $x$, up to homotopy. Concatenating these two paths we get a loop $\gamma = \alpha \# \beta$ in $W$.  Let $\sigma$ be a loop in $F_b$ based at $y$ such that $\rho(\sigma)=\rho(\gamma)$, which exists by the assumption that $\rho(\pi_1F_b)=\rho(\pi_1W)$. We can replace $\alpha$ by the concatenation of $\alpha$ with the inverse of $\sigma$ to obtain a path from $x$ to $y$ as in condition (2) of a $\rho$-compatible weak Euler structure. 

For every other pair $y,z$ of critical points of $f_b$ in the component $C_i$ the path from $y$ to $z$ is determined by the path from $x$ to $y$ and from $x$ to $z$, because the resulting triangle must be null homotopic in $F_b$. Hence we have the desired choice of paths corresponding to one component $C_i$. Repeating the above procedure for each component produces the required $\rho$-compatible weak Euler structure.

Next, suppose that $e_1$ and $e_2$ are two different $\rho$-compatible weak Euler structures for $f_b$. Then the paths $\alpha_1, \alpha_2$ in $F_b$ between two critical points $x,y$ of $f_b$ in a same component $C_i$ differ by a loop $\sigma$ such that $\rho(\sigma)=1$. Hence  for a fixed choice of basepoint and reference paths, the incidence matrices over $\bC$ relative to $e_1$ and $e_2$ are equal. In particular it follows that $\tau(f_b,e_1,\rho)=\tau(f_b, e_2, \rho)$. Hence Turaev torsion is independent of the $\rho$-compatible weak Euler structure for $f_b$. 

Finally, we address independence of the fibre $F_b$. Observe that the condition $\partial_F f \pitchfork 0$ ensures that  $f_b$ only has Morse singularities in the complement of a codimension 1 subset of $B$. Moreover, $f_b$ only has Morse or Morse birth/death singularities in the complement of a codimension 2 subset of $B$. Since $B$ is connected, it follows that any two points $b_0, b_1 \in B$ can be joined by a path $b_t \in B$ such that $f_t=f_{b_t}$ is Morse for all but finitely many times $t\in [0,1]$, at which $f_t$ has Morse birth/death singularities. Start with a $\rho$-compatible Euler structure for $f_0$ and propagate it to a family $e_t$ of $\rho$-compatible Euler structures for $f_t$. The propagation is uniquely determined away from the birth/death points and the Turaev torsion does not change by continuity. Near a birth/death point it might be that condition (b) of Definition \ref{Def: weak euler compatible} does not hold. When this is the case we replace the path between the two critical points which are about to die by the unique connecting trajectory between them. According to Definition \ref{Definition: weak}, in doing so we must also change the paths for all other critical points in that same connected component of the critical locus. However, this still gives a $\rho$-compatible weak Euler structure, hence the Turaev torsion is unchanged. We can then conclude by Lemma \ref{weakinv} that the Turaev torsion does not change as we cross the birth/death point. \end{proof}

\begin{definition}
The common Turaev torsion produced by Lemma \ref{familyturaev} is called the fibre Turaev torsion and is denoted by $\tau(f,\pi, \rho) \in U(R) / \pm 1$. It is a unit of $R$, well defined up to sign, which only depends on the torsion pair $(\pi, \rho)$ and the Euler function $f:W \to \bR$.
\end{definition}


\subsection{Torsion of open manifolds}

Up to now we have been working exclusively with closed manifolds, but for the intended applications it will be necessary to consider a slight generalization. We begin by discussing Reidemeister torsion. Let $M$ be an orientable manifold without boundary and recall the notion of a fibration at infinity \ref{Definition: fibration at infinity}. For $f:M \to \bR$ a fibration at infinity, we can build the Thom-Smale complex as before, because all the critical points as well as the trajectories between them stay in a compact subset of $M$. The homology of the Thom-Smale complex, twisted or untwisted, will in general depend on the fibration at infinity $f$. For example, over $\bZ$ the Thom-Smale complex computes the integral cohomology of the pair $(M,M_{-})$, where $M_{-}$ is the sublevel set $\{ f \leq z \}$ for $z\ll0$, which of course depends on $f$. However, observe the following.

\begin{lemma}\label{Lemma: bars}
Suppose that $f_i:M \to \bR$, $i=1,2$ are two fibrations at infinity on a boundaryless manifold $M$ which are a finite distance from each other in the $C^0$ norm. Then their Thom-Smale complexes have isomorphic homologies.
\end{lemma}
\begin{proof}
Suppose that $||f_1-f_2||_{C^0} \leq C$, where $C>0$ is a constant. Let $M^i_z=\{f_i \leq z \} \subset M $ for $z \in \bR$. Then $H_*(M,M^1_z) \to H_*(M,M^2_{z+C})$ is an isomorphism for $z \ll 0$, because $H_*(M,M^2_{z+C}) \to H_*(M,M^1_{z+2C})$ is a left inverse and $H_*(M,M^2_{z-C}) \to H_*(M,M^1_z)$ is a right inverse. \end{proof}

\begin{remark} One can take integral coefficients or twisted coefficients in the above proof. \end{remark}

If we demand a $C^1$ bound then we have an even greater control.

\begin{lemma}\label{lemma: we need C1}
Any two fibrations at infinity which are a bounded distance from each other in the $C^1$ norm and which also have the property that the norm of their derivative is a proper function are homotopic through fibrations at infinity.
\end{lemma}

\begin{proof}
We can simply take a convex interpolation $tf+(1-t)g$ between the two functions, $0 \leq t \leq 1$. The hypothesis implies that each $tf+(1-t)g$ has proper derivative, hence the interpolation is a homotopy of fibrations at infinity, see Criterion 0.2.1 in \cite{EG98}.
\end{proof}

\begin{example}\label{ex: standard quadratic}
Examples of functions with this property are given by functions on $W\times \bR^{2n}$ with $W$ compact which are a bounded $C^1$ distance from the standard quadratic function $||x||^2-||y||^2$ on $\bR^{2n}$. We will apply Lemma \ref{lemma: we need C1} to such functions below.
\end{example}

\begin{warning}
Two fibrations at infinity which are a finite distance from each other in the $C^1$ norm need not be homotopic through fibrations at infinity.  The following instructive example was pointed out to us by E. Giroux. Let $X_0$ and $X_1$ be non-diffeomorphic closed manifolds of dimension $\geq 5$ which are h-cobordant. Let $M_{01}$ be an h-cobordism from $X_0$ to $X_1$ and let $M_{10}$ be the inverse h-cobordism. Let $M$ be the noncompact boundaryless manifold obtained by concatenating $M_{01}$ and $M_{10}$ infinitely many times. Choose Morse functions $g_0:M_{01} \to [0,1]$ and $g_1:M_{10} \to [0,1]$, standard near the boundary, and assemble them into a Morse function $f:M \to \bR$. Since the concatenation of $M_{01}$ and $M_{10}$ is diffeomorphic to $X_0 \times [0,1]$, the function $f$ can be modified by a $C^1$ bounded (in fact, periodic) perturbation to a fibration $f_0 : M \to \bR$ with fibre $X_0$. Switching the roles of $X_0$ and $X_1$, it follows that $f$ can also be modified by a $C^1$ bounded perturbation to a fibration $f_1: M \to \bR$ with fibre $X_1$. The fibrations at infinity (in fact, fibrations) $f_0$ and $f_1$ are a finite distance from each other in the $C^1$ norm, yet they are not homotopic through fibrations at infinity since their fibres are not even diffeomorphic.
\end{warning}

Let $f:M \to \bR$ be a fibration at infinity and consider a local system $\rho: \pi_1 M \to U(R)$ such that the twisted Thom-Smale complex $(C_*(f; R^\rho), \partial^\rho)$ is acyclic. We can define the Reidemeister torsion $r(f,\rho) \in U(R) / \pm \rho(\pi_1M)$ in the same way as before, by taking the determinant of $\partial^\rho + \delta^\rho$. Although the Reidemeister torsion may now depend on the fibration at infinity $f$, it is easy to see that $r(f_0,\rho)=r(f_1,\rho)$ for any other fibration at infinity $f_1: M \to \bR$ which is homotopic to $f$ through fibrations at infinity $f_t$. The proof of this invariance property is the same as in the compact case. The key point is to ensure that the Thom-Smale complex does not experience dramatic bifurcations coming from critical points or trajectories escaping to infinity, see Remark \ref{Remark: infinity family}. 

Similarly we can define the Turaev torsion $\tau(f,e,\rho) \in U(R) / \pm 1$ of a fibration at infinity $f:M \to \bR$ with respect to an Euler structure or a weak Euler structure for $f$. For the invariance property as in Lemmas \ref{inv} and \ref{weakinv} we must again take note of Remark \ref{Remark: infinity family}. For a noncompact version of the fibre Turaev torsion we must generalize Definition \ref{torsionpair} of a torsion pair $(\pi,\rho)$ to allow $F$ (and hence also $W$) to be noncompact. The additional data we need to keep track of is a reference fibration at infinity, which must satisfy the following property.

\begin{definition}
Let $\pi: W \to B$ be a fibre bundle of manifolds, $\rho: \pi_1W \to U(R)$ a local system and $g:W \to \bR$ a fibration at infinity. We say that $g$ is $\rho$-fibre acyclic if the $\rho$-twisted Thom-Smale complex $(C_*(g_b; R^\rho), \partial^\rho)$ of the restriction $g_b$ of $g$ to the generic fibre of $\pi$ is acyclic with respect to the pullback of $\rho$ to $\pi_1F$ by the inclusion $F \subset W$.
\end{definition}

We are now ready to generalize the notion of a torsion pair \ref{torsionpair} to the noncompact setting.

 \begin{definition}\label{torsiontriple}
A torsion triple $(\pi, g, \rho)$ consists of a fibre bundle of manifolds $\pi: W \to B$ with fibre $F$, a local system $\rho: \pi_1W \to U(R)$ and a fibration at infinity $g:W \to \bR$ such that the following properties hold.
\begin{itemize} 
\item $B$, $F$ and $W$ are boundaryless, connected and orientable.
\item $g$ is $\rho$-fibre acyclic.
\item $\rho(\pi_1F)=\rho(\pi_1W)$.
\end{itemize}
\end{definition}
Suppose that $(\pi, g, \rho)$ is a torsion triple and $f:W \to \bR$ is a fibration at infinity which is of Euler type and which is a bounded distance from $g$ in the $C^0$ norm. Then by Lemma \ref{Lemma: bars} the $\rho$-twisted Thom-Smale complex $(C_*(f_b; R^\rho), \partial^\rho)$ is also acyclic. Hence we can define its fibre Turaev torsion $\tau( f,\pi,\rho) \in U(R) / \pm1 $, just as before. The proof of Lemma \ref{familyturaev} remains the same. Moreover, if $f$ is homotopic to $g$ through fibrations at infinity, then the Reidemeister torsion of the restriction of $f$ to the fibre is always the same as that of $g$. However, the fibre Turaev torsion $\tau(f,\pi, \rho)$, which is a lift of the Reidemeister torsion of the restriction of $g$ to the fibre, in general does depend on $f$. 


\subsection{Stability of torsion}\label{Section: stability of torsion}
We now consider the effect of stabilization on torsion. Recall that given a manifold $M$ we set $\text{stab}(M)=M \times \bR^2$, and given a fibration at infinity $f: M \to \bR$ we define $\text{stab}(f): \text{stab}(M) \to \bR$ by the formula $\text{stab}(f)(m,x,y)=f(m) + x^2 - y^2$, where $m\in M$ and $(x,y) \in \bR^2$. The first observation is that the critical points of $f$ are in canonical bijection with the critical points of $\text{stab}(f)$. Under this bijection the Morse indices increase by 1. Next, suppose that $\rho: \pi_1M \to U(R)$ is such that the $\rho$-twisted Thom-Smale complex of $f$ is acyclic. Let us compute $r(f,\rho)$ with respect to a set of choices: an orientation for $M$, orientations for the negative eigenspaces of $d^2f$ at the critical points of $f$, a gradient-like vector field $Z$, a basepoint $m_0\in M$ and paths $\gamma$ from $m_0$ to the critical points of $f$. 

Denote by $\text{stab}(\rho) : \pi_1 \text{stab}(M) \to U(R)$ the pullback of $\rho$ by the projection $M \times \bR^2 \to M$. Orient $\text{stab}(M)=M \times \bR^2$ using the product of the chosen orientation for $M$ and the canonical orientation for $\bR^2$. Orient the negative eigenspaces of $d^2 \text{stab}(f)$ using the product of the chosen orientations for the negative eigenspaces of $d^2f$ with the direction $\partial_y$. Let $\text{stab}(Z) = Z + \partial_x - \partial_y$, which is gradient-like for $\text{stab}(f)$. Take as a basepoint for $\text{stab}(M)$ the product of the basepoint of $M$ and the origin of $\bR^2$. Finally, for any path $\gamma$ in $M$ let $\text{stab}(\gamma)$ be the path in $\text{stab}(M)$ obtained by taking the product of $\gamma$ with the constant path at the origin in $\bR^2$. With respect to these choices the twisted Thom-Smale complex $C_*(\text{stab}(f); R^{\text{stab}(\rho)} )$ is the suspension of the Thom-Smale complex $C_*(f; R^\rho)$, i.e. it is the same chain complex but with a degree shift by 1. Hence 
\[ r\big( \text{stab}(f), \text{stab}(\rho)\big) = r(f, \rho)^{-1} \qquad \text{and} \qquad r\big( \text{stab}^2(f), \text{stab}^2(\rho)\big) = r(f, \rho).\]
In particular we deduce that Reidemeister torsion is invariant under stabilization as long as we do it an even number of times. This is called even stabilization. 

Next we consider the effect of stabilization on Turaev torsion. Let $e$ be an Euler structure for $f:M \to \bR$. Then there is an obvious Euler structure $\text{stab}(e)$ for $\text{stab}(f) : \text{stab}(M) \to \bR$ induced by the bijection of critical points and by the stabilization of paths as above. Moreover, for each choice of paths in $M$ compatible with $e$ there is a canonical choice of paths in $\text{stab}(M)$ compatible with $\text{stab}(e)$, namely the one obtained by stabilizing the paths. The same holds if we replace the Euler structure $e$ by a weak Euler structure. In both cases we deduce that 
\[  \tau\big(\text{stab}(f), \text{stab}(e), \text{stab}(\rho) \big) = \tau(f,e,\rho)^{-1} \qquad \text{and} \qquad \tau\big(\text{stab}^2(f), \text{stab}^2(e), \text{stab}^2(\rho) \big) = \tau(f,e, \rho). \]
We conclude that Turaev torsion is invariant under even stabilization. Hence the same is true for the fibre Turaev torsion. We record this fact explicitly for future reference.

\begin{proposition}\label{Proposition: turaev stability} Let $(\pi,g,\rho)$ be a torsion triple and let $f:W \to \bR$ be a fibration at infinity which is of Euler type and a bounded distance from $g$ in the $C^0$ norm. Then the fibre Turaev torsion satisfies $\tau(f, \pi, \rho) = \tau\big( \text{stab}^2(f) , \text{stab}^2(\pi), \text{stab}^2(\rho) \big)$.
\end{proposition}

\begin{remark} Recall that $\text{stab}^2(f) : \text{stab}^2(W) \to \bR$ is the (fibrewise) double stabilization of $f$ on $\text{stab}^2(\pi):\text{stab}^2(W) \to B$, which is the fibre bundle $F \times \bR^4 \to W \times \bR^4 \to B$. \end{remark}



 
\subsection{A Legendrian invariant}\label{sec: invariant}
 
 We are now ready to use the fibre Turaev torsion to produce a Legendrian invariant, which is the main protagonist of this article. This invariant will only be defined for a certain class of Legendrians. As before, let $B$ be a closed, connected, orientable manifold.
 
 \begin{definition}
A closed, orientable Legendrian submanifold $\Lambda \subset J^1(B)$ is said to be of Euler type if each connected component $\Lambda_i$ of $\Lambda$ is simply connected and the projection to the base $\Lambda_i \to B$ has degree zero. \end{definition}

\begin{remark} If $\Lambda$ is an Euler Legendrian then any generating family $f$ for $\Lambda$ is automatically of Euler type, see Definition \ref{Def: function euler}. We recall that a generating family is by definition always a fibration at infinity and moreover always generates $\Lambda$ transversely, i.e. $\partial_F f \pitchfork 0$. 
\end{remark}

The invariant in question is the following. 

\begin{definition}
The Legendrian Turaev torsion of an Euler Legendrian $\Lambda \subset J^1(B)$ with respect to a torsion pair $(\pi, \rho)$ is a set $T(\Lambda, \pi, \rho) = \{ \tau(f, \pi, \rho) \} \subset U(R) / \pm 1$. Its elements consist of the fibre Turaev torsions with respect to $\rho$ of all generating families $f$ for $\Lambda$ on even stabilizations $\text{stab}^{2k}(W)$ of $W$, where we only consider those $f$ which are a bounded distance from $||x||^2-||y||^2$ in the $C^1$ norm.
\end{definition}

\begin{remark}
We will write $T(\Lambda, W, \rho)$ instead of $T(\Lambda, \pi, \rho)$ when the fibration $\pi : W \to B$ is obvious from the context, as is often the case. 
We also implicitly identify local systems on $W$ and local systems on $\text{stab}^{2k}(W)$ via the projection $W \times \bR^{4k} \to W$.
\end{remark}

\begin{remark}
We could also define the Legendrian Turaev torsion of an Euler Legendrian $\Lambda$ with respect to a torsion triple $(\pi, g, \rho)$. 
Since we won't need this generalization for our applications, we will restrict our discussion to torsion pairs $(\pi, \rho)$ for simplicity, though see Remark \ref{remark: need triples}.
\end{remark}


Legendrian Turaev torsion exhibits a certain functoriality, which we now discuss. Let $X \subset B$ be a closed submanifold and let $\Lambda \subset J^1(B)$ be a Legendrian submanifold. Under a generic transversality hypothesis, the intersection of $\Lambda$ with $J^1(B)|_X$ reduces to a Legendrian submanifold $\Lambda_X \subset J^1(X)$. The Legendrian $\Lambda_X$ is characterized by the property that its front coincides with the restriction of the front of $\Lambda$ to $X \times \bR \subset B \times \bR$. We call $\Lambda_X$ the reduction of $\Lambda$ to $X$.

\begin{proposition}\label{functoriality}
Let $X \subset B$ be a closed, orientable, connected submanifold and let $\Lambda \subset J^1(B)$ be an Euler Legendrian such that the reduction $\Lambda_X \subset J^1(X)$ is also an Euler Legendrian. Then for every torsion pair $(\pi, \rho)$ we have $T(\Lambda , W, \rho)  \subset T(\Lambda_X, W_X, \rho_X)$, where $W_X$ is the restriction of $W$ to $X$ and $\rho_X$ is the precomposition of $\rho: \pi_1W \to U(R)$ with $\pi_1W_X \to \pi_1W$. 
\end{proposition}

\begin{proof} 
Every generating family for $\Lambda$ on $W$ restricts to a generating family for $\Lambda_X$ on $W_X$.
\end{proof}

We deduce that Legendrian Turaev torsion is a Legendrian invariant.

\begin{theorem}\label{legendrianinvariant}
Suppose that $\Lambda_0$ and $\Lambda_1$ are Legendrian isotopic Euler Legendrians in $J^1(B)$. Then for every torsion pair $(\pi, \rho)$ we have $T(\Lambda_0, W, \rho) = T(\Lambda_1, W, \rho)$ as subsets of $U(R) / \pm 1$.
\end{theorem}

\begin{proof}
Let $\Lambda_t\subset J^1(B)$ be a Legendrian isotopy between $\Lambda_0$ and $\Lambda_1$, and denote by $\Lambda$ its trace. This is the Legendrian submanifold of $J^1([0,1] \times B)$ whose front projection is $\bigcup_t t \times \Sigma_t \subset [0,1] \times B \times \bR$, where $\Sigma_t \subset B \times \bR$ is the front projection of $\Lambda_t$. The homotopy lifting property \ref{Persistence Theorem} implies the inclusion $T(\Lambda_0 , W , \rho) \subset T(\Lambda , [0,1] \times W  , \rho)$, where we also use $\rho$ to denote the local system $\pi_1 ([0,1] \times W ) \to U(R)$ which is the pullback of $\rho$ by the projection $[0,1] \times W \to W$. Combining this inclusion with the functoriality property \ref{functoriality} we obtain $T(\Lambda_0, W, \rho) \subset T(\Lambda_1, W, \rho)$. The reverse inclusion follows by switching the roles of $\Lambda_0$ and $\Lambda_1$. 
\end{proof}

\begin{remark}\label{remark: need triples}
Note that we implicitly used Proposition \ref{Proposition: turaev stability} in the proof, since the homotopy lifting property requires stabilization. Therefore, even if we are only interested in torsion pairs, in order to produce a Legendrian invariant we are forced to consider torsion triples (at least those which are stabilizations of torsion pairs).
\end{remark}
 
The Legendrian invariant $T(\Lambda, W, \rho)$ will be used below to study mesh Legendrians up to Legendrian isotopy, providing examples and explicit computations. We conclude this section by observing that the functoriality property can also be used to obstruct the existence of generating families for Legendrian cobordisms more general than the trace of Legendrian isotopies. 
 
 \begin{corollary}\label{cobordism}
Let $\Lambda_1, \Lambda_2 \subset J^1(B)$ be Euler Legendrians such that $T(\Lambda_1,W,\rho) \cap T(\Lambda_2, W, \rho) = \varnothing$ for some torsion pair $(\pi, \rho)$. Then any Euler Legendrian $\Lambda \subset J^1(B \times [0,1])$ whose reduction over $B \times 0$ and $B \times 1$ is $\Lambda_0$ and $\Lambda_1$ respectively cannot be generated by a family $f:W \times [0,1]\to \bR$. 
 \end{corollary}
 
 \begin{proof} This is an immediate consequence of Proposition \ref{functoriality}. \end{proof}
 
 \begin{example}
The examples of Corollary \ref{corollarypairs} are such that each pair $\Lambda_\pm$ is generated by a family $f_\pm : E \to \bR$, where $S^1 \to E \to \Sigma$ is the same circle bundle for both $\Lambda_+$ and $\Lambda_-$. We will see below that for a certain $\rho: \pi_1E \to U(R)$ such that $(E,\rho)$ is a torsion pair, the Turaev torsions $T(\Lambda_\pm, E , \rho)$ are distinct one-element subsets of $U(R) / \pm 1$. Hence from Theorem \ref{legendrianinvariant} we deduce that $\Lambda_+$ and $\Lambda_-$ are not Legendrian isotopic. Now, any generic homotopy $f_t:E \to \bR$ between $f_+$ and $f_-$ produces a Legendrian cobordism $\Lambda \subset J^1(\Sigma \times [0,1])$ whose reduction over $\Sigma \times 0$ and $\Sigma \times 1$ is $\Lambda_+$ and $\Lambda_-$ respectively. Hence from Corollary \ref{cobordism} we deduce that the cobordism $\Lambda$ must have at least one component which is not simply connected.
 \end{example}
 
 

\section{Mesh Legendrians}\label{Section: mesh Legendrians}

\subsection{Generating families from systems of disks}\label{systems of disks}
 
We begin our study of mesh Legendrians by proving that they all admit generating families on circle bundles. More precisely, we prove that $\Lambda_G$ is generated by a function on the circle bundle of Euler number $w(G)$. 

We first make a small digression into the Morse theory of $S^1$. The space of Morse functions on $S^1$ is homotopy equivalent to a disjoint union of infinitely many circles. However, if in addition to quadratic (Morse) singularities we also allow cubic (Morse birth/death) singularities we get a connected space, and if we only allow positive cubic singularities we get a contractible space. Indeed, identify $S^1 = [0,1]/ 0 \sim 1$ so that of the derivative of a function $f:S^1 \to \bR$ is another function $f':S^1 \to \bR$.  
\begin{definition}
A smooth function $f:S^1 \to \bR$ is called a generalized Morse function (GMF) if for every $x \in S^1$ one of the first three derivatives of $f$ at the point $x$ is nonzero. We say that $f$ is a positively oriented generalized Morse function (PGMF) if $f'''(x)>0$ whenever $f'(x)=f''(x)=0$. \end{definition}

 We endow the space $\text{PGMF}(S^1)$ of PGMFs on $S^1$ with the $C^\infty$ topology. The following result is proved in \cite{IK93}. It is closely related to the more general result \cite{I87}, \cite{EM12} that the space of framed functions is contractible.
 
 \begin{theorem}\label{contractibility}
 The space $\text{PGMF}(S^1)$ is contractible.
 \end{theorem}
 
 Next, let $p: E \to \Sigma$ be an oriented circle bundle over a surface $\Sigma$. 
 
 \begin{definition}
A  fibrewise PGMF is a function $f:E \to \bR$ such that the restriction of $f$ to every (oriented) fibre is a PGMF.
 \end{definition}
 
Equivalently, fibrewise PGMFs are sections of the bundle $\text{PMGF}(S^1) \to \text{PGMF}(E) \to \Sigma$ whose fibre over a point $x \in \Sigma$ consists of all PGMFs on the oriented fibre $S^1$ of $E$ over $x$. We will prove below that every mesh Legendrian is generated by a fibrewise PGMF on a circle bundle. We focus on a special class of fibrewise PGMFs which can be understood explicitly in terms of systems of disks.
 
 \begin{definition}
 A system of disks on the total space of a circle bundle $p:E \to \Sigma$ consists of a finite union of disjoint embedded 2-disks $D_i \subset E$ such that each projection $p:D_i \to \Sigma$ is an immersion and such that the interiors of the projected disks $p(D_i) \subset \Sigma$ cover $\Sigma$. (See Figure \ref{Fig: system of disks}.)
 \end{definition}
 
The following result is proved in \cite{IK93}.
 
 \begin{proposition}\label{systems}
 To every system of disks $\{ D_i \}$ on $E$ there exists a fibrewise PGMF $f:E \to \bR$ whose fibrewise critical set consists of positive cubic singularities along each $\partial D_i$ and Morse singularities along two parallel copies of $D_i$, which are obtained by pushing $D_i$ in the positive and negative $S^1$ directions relative to $\partial D_i$. Moreover, the space of such $f$ is contractible. 
 \end{proposition}
 
  %
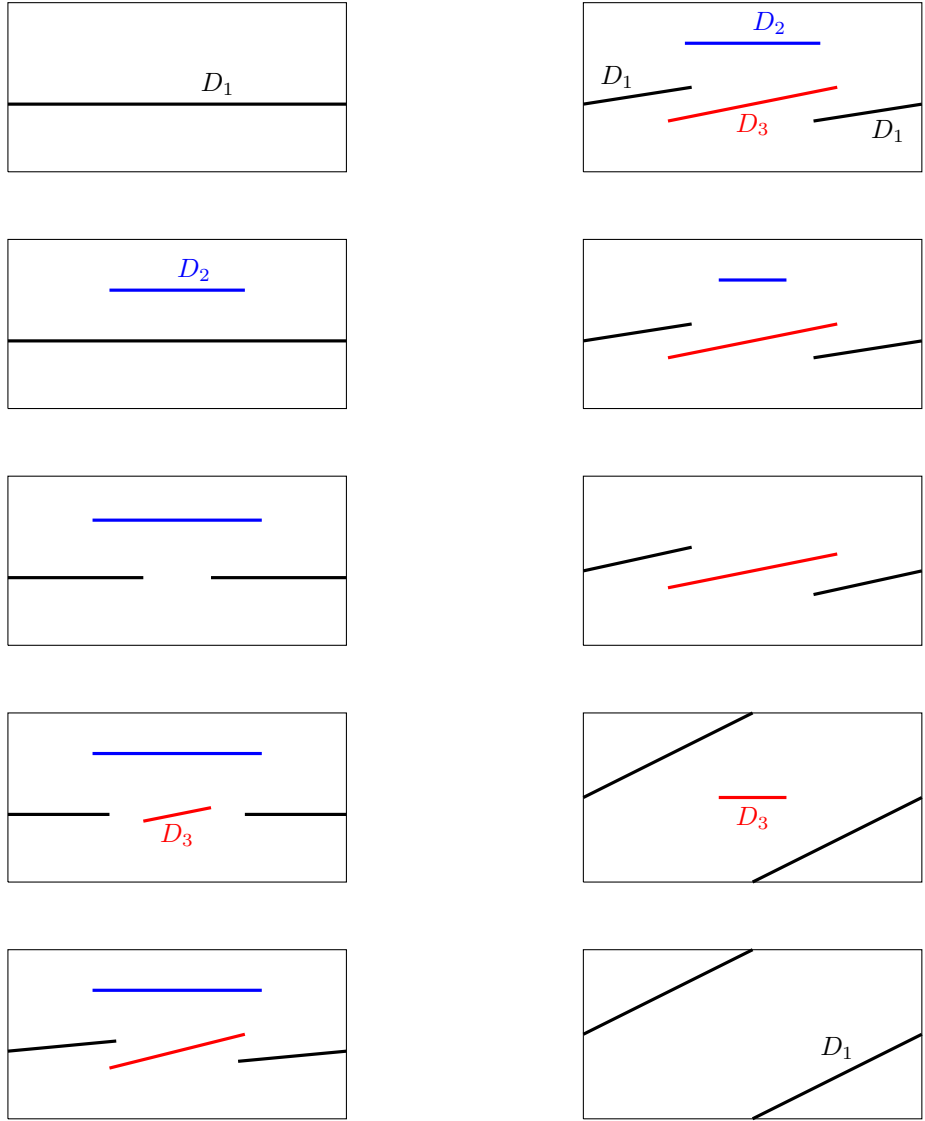
\begin{figure}[htbp]
\begin{center}
\begin{tikzpicture}[scale=.45]
%
\begin{scope}[xshift=-3cm,yshift=28cm]
\draw (-5,0)--(5,0)--(5,5)--(-5,5)--(-5,0);
\draw[very thick] (-5,2)--(5,2); 
\draw (1.2,2) node[above]{$D_1$};
\end{scope}
%
\begin{scope}[xshift=-3cm,yshift=21cm]
\draw (-5,0)--(5,0)--(5,5)--(-5,5)--(-5,0);
\draw[very thick] (-5,2)--(5,2); 
\draw[very thick,color=blue] (-2,3.5)--(2,3.5); 
\draw[color=blue] (.5,3.5) node[above]{$D_2$};
\end{scope}
%
\begin{scope}[xshift=-3cm,yshift=14cm]
\draw (-5,0)--(5,0)--(5,5)--(-5,5)--(-5,0);
\draw[very thick] (-5,2)--(-1,2) (1,2)--(5,2); 
\draw[very thick,color=blue] (-2.5,3.7)--(2.5,3.7); 
\end{scope}
%
\begin{scope}[xshift=-3cm,yshift=7cm]
\draw (-5,0)--(5,0)--(5,5)--(-5,5)--(-5,0);
\draw[very thick] (-5,2)--(-2,2) (2,2)--(5,2); 
\draw[very thick,color=blue] (-2.5,3.8)--(2.5,3.8); 
\draw[very thick,color=red] (-1,1.8)--(1,2.2); 
\draw[color=red] (0,2) node[below]{$D_3$};
\end{scope}
%
\begin{scope}[xshift=-3cm,yshift=0cm]
\draw (-5,0)--(5,0)--(5,5)--(-5,5)--(-5,0);
\draw[very thick] (-5,2)--(-1.8,2.3) (1.8,1.7)--(5,2); 
\draw[very thick,color=blue] (-2.5,3.8)--(2.5,3.8); 
\draw[very thick,color=red] (-2,1.5)--(2,2.5); 
\end{scope}
%
\begin{scope}[xshift=14cm,yshift=28cm]
\draw (-5,0)--(5,0)--(5,5)--(-5,5)--(-5,0);
\draw[very thick] (-5,2)--(-1.8,2.5) (1.8,1.5)--(5,2); 
\draw[very thick,color=blue] (-2,3.8)--(2,3.8); 
\draw[very thick,color=red] (-2.5,1.5)--(2.5,2.5); 
\draw (-4,2.2) node[above]{$D_1$};
\draw (4,1.8) node[below]{$D_1$};
\draw[color=blue] (.5,3.8) node[above]{$D_2$};
\draw[color=red] (0,2) node[below]{$D_3$};
\end{scope}
%
\begin{scope}[xshift=14cm,yshift=21cm]
\draw (-5,0)--(5,0)--(5,5)--(-5,5)--(-5,0);
\draw[very thick] (-5,2)--(-1.8,2.5) (1.8,1.5)--(5,2); 
\draw[very thick,color=blue] (-1,3.8)--(1,3.8); 
\draw[very thick,color=red] (-2.5,1.5)--(2.5,2.5); 
\end{scope}
%
\begin{scope}[xshift=14cm,yshift=14cm]
\draw (-5,0)--(5,0)--(5,5)--(-5,5)--(-5,0);
\draw[very thick] (-5,2.2)--(-1.8,2.9) (1.8,1.5)--(5,2.2); 
\draw[very thick,color=red] (-2.5,1.7)--(2.5,2.7); 
\end{scope}
%
\begin{scope}[xshift=14cm,yshift=7cm]
\draw (-5,0)--(5,0)--(5,5)--(-5,5)--(-5,0);
\draw[very thick] (-5,2.5)--(0,5) (0,0)--(5,2.5); 
\draw[very thick,color=red] (-1,2.5)--(1,2.5); 
\draw[color=red] (0,2.5) node[below]{$D_3$};
\end{scope}
%
\begin{scope}[xshift=14cm,yshift=0cm]
\draw (-5,0)--(5,0)--(5,5)--(-5,5)--(-5,0);
\draw[very thick] (-5,2.5)--(0,5) (0,0)--(5,2.5); 
\draw (2.5,1.5) node[above]{$D_1$};
\end{scope}
\end{tikzpicture}
\caption{A system of 3 disks $D_1,D_2,D_3$ in $E$ the Hopf circle bundle over $S^2$ is shown in this sequence of cross-sections of these disks. This gives the fiberwise PGMF illustrated in Figure \ref{Fig: start}.}
\label{Fig: system of disks}
\end{center}
\end{figure}
 
  Let $G$ be a bicolored trivalent ribbon graph and let $\Sigma_G$ be the closed oriented surface associated to it, see Section \ref{Subsection: main results}. Let $U_1 \subset \Sigma_G$ be a thickening of the vertices and let $U_2 \subset \Sigma_G$ be a thickening of the centerpoints of those edges having endpoints of the same color. Let $S^1 \to E_G \to \Sigma_G$ be the oriented circle bundle given by gluing the trivial bundle on $U_1 \cup U_2$ to the trivial bundle on $\Sigma_G \setminus (U_1 \cup U_2)$ using a clutching function around the boundary of each component of $U_1 \cup U_2$ which (for reasons explained below and illustrated in Figure \ref{Fig: 3 circles}) has winding number $-1$ around each positive vertex and each centerpoint of an edge connecting two negative vertices and $+1$ around each negative vertex and each centerpoint of an edge connecting two positive vertices.
  
Start by placing a disk $D_i$ above each face $F_i$ of $\Sigma_G \setminus G$. Over each edge of $G$ we extend the two disks $D_i, D_j$ corresponding to either side of the edge in the homotopically unique way to do so. To finish the construction we need to specify how we extend the three disks which surround a given vertex of $G$. Number the disks clockwise around the vertex in $G$ so that $D_i,D_j,D_k$ are in cyclic order/anti-cyclic order depending on whether the vertex is positive/negative, respectively. In a deleted neighborhood of a positive vertex, the bundle $E_G$ can be trivialized by taking the section which interpolates, say $D_i$ and $D_j$, along the edge separating these disks by moving in the positive direction around the fiber as we go from $D_i$ to $D_j$ and going the negative direction the other way. This gives a winding number of $-1$ at the positive vertex. It also gives a winding number of $+1$ at the center of the edge separating $D_i,D_j$ in the case when the other endpoint of that edge is also positive. This is because the section moves from $D_i$ to $D_j$ positively around the fiber close to the other endpoint, see Figure \ref{Fig: 3 circles}.

Note that there could be repetitions, for example one could have the same disk on the two sides of an edge, but this is ok since $p:D_i\to \Sigma_G$ is only required to be an immersion. This produces a system of disks $\{D_i\}$ on $E_G$. Any function $f:E_G \to \bR$ corresponding to $\{D_i\}$ under Proposition \ref{systems} generates the mesh Legendrian $\Lambda_G$, up to Legendrian isotopy, see \cite{IK93}. Therefore it only remains to compute the Euler number of $E_G$.  %
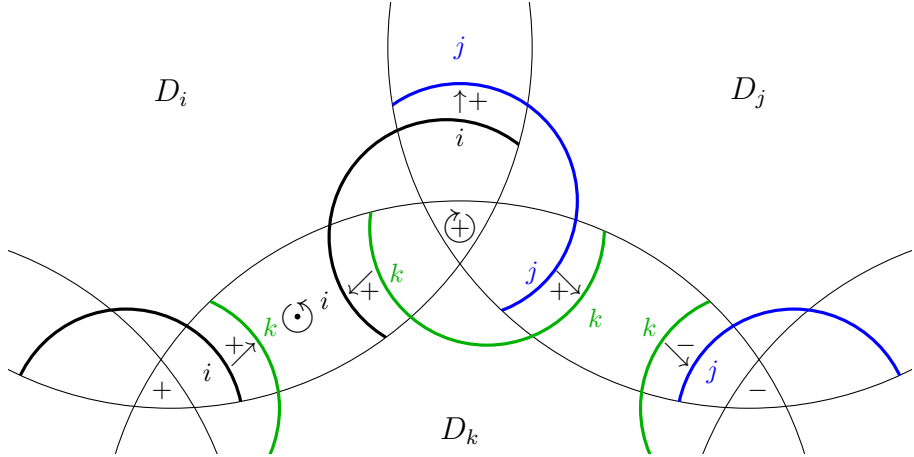
\begin{figure}[htbp]
\begin{center}
\begin{tikzpicture}[scale=1.2]
\coordinate (A) at (-3.2,1.5);
\coordinate (A1) at (-3.2,2);
\coordinate (B) at (3.2,1.5);
\coordinate (B1) at (3.2,2);
\coordinate (C) at (0,-2);
\coordinate (C1) at (0,-3.7);
\coordinate (D) at (-4.4,-2);
\coordinate (D1) at (-6.4,-4);
\coordinate (E1) at (6.4,-4);
\coordinate (E) at (4.4,-2);
\coordinate (V1) at (-3.3,-1.8);
\coordinate (V2) at (0,0);
\coordinate (V3) at (3.3,-1.8);
\clip (-5,-2.5) rectangle (5,2.5);
\draw (V1) node{$+$};
\draw (V2) node{$+$};
\draw (V3) node{$-$};
\draw (A1) circle[radius=4cm] (A) node{\Large$ D_i$};
\draw (B1) circle[radius=4cm] (B) node{\Large$D_j$};
\draw (C1) circle[radius=4cm] (C) node[below]{\Large$D_k$};
\draw (D1) circle[radius=4cm];
\draw (E1) circle[radius=4cm];
%
\draw (0,1.4) node{$\uparrow$};
\draw (1.2,-.6) node{$\searrow$};
\draw (.2,1.4) node{$+$};
\draw (1.1,-.7) node{$+$};
\draw (-1,-.7) node{$+$};
\draw (-2.5,-1.3) node{$+$};
\draw (2.5,-1.3) node{$-$};
\draw (-1.1,-.6) node{$\swarrow$};
\draw (-2.4,-1.4) node{$\nearrow$};
\draw (2.4,-1.4) node{$\searrow$};
\draw (0,0) node{\huge $\circlearrowright$};
\draw (-1.8,-1) node{\Huge$\cdot$};
\draw (-1.8,-.99) node{\huge $\circlearrowleft$};
\begin{scope}
\clip (B1) circle[radius=4cm];
\draw[very thick,color=blue] (0,0.3)circle[radius=1.3cm] (0,2)node{$j$}(0.8,-.5)node{$j$}
(3.7,-2.2)circle[radius=1.3cm]
(2.8,-1.6) node{$j$};
\end{scope}
\begin{scope}
\clip (A1) circle[radius=4cm];
\draw[very thick] (-.15,-.1)circle[radius=1.3cm]
(-1.5,-.8)node{$i$}(0,1)node{$i$}
(-3.7,-2.2)circle[radius=1.3cm]
(-2.8,-1.6) node{$i$};
\end{scope}
\begin{scope}
\clip (C1) circle[radius=4cm];
\draw[very thick,color=green!70!black] (.3,0) circle[radius=1.3cm]
(1.5,-1)node{$k$}(-.7,-.5)node{$k$}
(-3.3,-2)circle[radius=1.3cm]
(3.3,-2)circle[radius=1.3cm]
(-2.1,-1.1) node{$k$}
(2.1,-1.1) node{$k$};
\end{scope}
\end{tikzpicture}
\caption{As we go clockwise around a positive vertex (center) the canonical section moves positively around the fiber circle $i\to j\to k\to i$ giving winding number $-1$. As we go counterclockwise around the midpoint of the edge $E_{ik}$ connecting two positive vertices, we go from $i$ to $k$ back to $i$ in the positive direction giving winding number $+1$. Around the midpoint of an $E_{jk}$ connecting positive and negative vertices we go positively from $j$ to $k$ then negatively from $k$ to $j$ giving winding number zero.}
\label{Fig: 3 circles}
\end{center}
\end{figure}
%
 
 
 \begin{corollary}\label{generation}
 A mesh Legendrian $\Lambda_G \subset J^1(\Sigma_G)$ is generated by a fibrewise PGMF on the oriented circle bundle $S^1 \to E_G \to \Sigma_G$ of Euler number $w(G)$.
 \end{corollary}
 
 \begin{proof}
 Let $P$ be the number of positive vertices and let $N$ be the number of negative vertices. By Figure \ref{Fig: 3 circles} and the discussion above, the Euler number is equal the number of edges connecting two positive vertices minus $P$ plus $N$ minus the number of edges connecting two negative vertices. This is
 \[
 	e(E_G)=\frac{3P-Q}2-P+N-\frac{3N-Q}2=\frac12(P-N)=w(G).
 \]
 where $Q$ is the number of edges connecting vertices of opposite colors. \end{proof}
 
 
In fact, the converse to Corollary \ref{generation} also holds: 
 
  \begin{proposition}\label{euler for circle}
 Suppose that $\Lambda_G \subset J^1(\Sigma_G)$ is a mesh Legendrian generated by a fibrewise PGMF on the oriented circle bundle $S^1 \to E \to \Sigma_G$. Then $e(E)=w(G)$. 
 \end{proposition}
 
We give a proof of Proposition \ref{euler for circle} using the ``cyclic set cocycle'' from \cite{I04}, which we briefly recall. Consider the category $\cZ$ of all finite nonempty cyclically ordered sets and cyclic order preserving monomorphisms. It follows from \cite{K92} that the geometric realization $B\cZ$ of this category is homotopy equivalent to $\mathbb CP^\infty$. Kontsevich gave a 2-form $\sum \theta_i\wedge \theta_j$ for his version of this space which he called $BU(1)^{comb}$. In \cite{I04} the integral of the Kontsevich 2-form is computed. Any 2-simplex in $B\cZ$ is a triple $A\subset B\subset C$ of cyclically ordered set. On this 2-simplex the value of the cocycle $c_\cZ$ is given by
 \[
 	c_\cZ(A,B,C)= -\frac12 \big(\mathbb P(a,b,c\text{ in cyclic order})-\mathbb P(a,c,b\text{ in cyclic order})\big)
 \]
In other words, $-2c_\cZ(A,B,C)$ is the probability that randomly chosen elements $a,b,c$ in $A,B,C$ are distinct and in cyclic order in $C$ minus the probability that they are distinct and in reverse cyclic order, see Figure \ref{Fig: cyclic sets}.

Given any oriented circle bundle $E$ over a surface $\Sigma$, any system of disks for $E$ gives a mapping from $\Sigma$ into $B\cZ$. It is shown in \cite{I04} and \cite{J89} that the pull-back of $[c_\cZ]$ to $ H^2(\Sigma)$ is the Euler class for $E$. We are now ready to give the proof.
 
 \begin{proof}[Proof of Proposition \ref{euler for circle}]
We use the cyclic set cocycle $c_\cZ$ defined above to compute the Euler number of $E$.
 Around a positive vertex in a ribbon graph a PGMF is given by three disjoint sections of the corresponding oriented circle bundle giving disks $D_i,D_j,D_k$ going clockwise around the vertex on the surface and clockwise around the fiber. Let $A=\{k\}, B=\{j,k\}, C=\{i,j,k\}$. With probability $\frac16$, randomly chosen elements of $A,B,C$ will be distinct and $a=k,b=j,c=i$. Since these are in reverse cyclic order around the fiber circle, the sign is positive. Thus we get a contribution of $+\frac1{12}$. Each of the six triangles in the barycentric subdivision of the triangle give the same contribution. So, we get a total of $c_\cZ=\frac12$ for each positive triangle and, similarly, $-\frac12$ for each negative triangle, see Figure \ref{Fig: cyclic sets}. Therefore, $e(E_G)=\frac12(P-N)=w(G)$ as claimed.
 \end{proof}

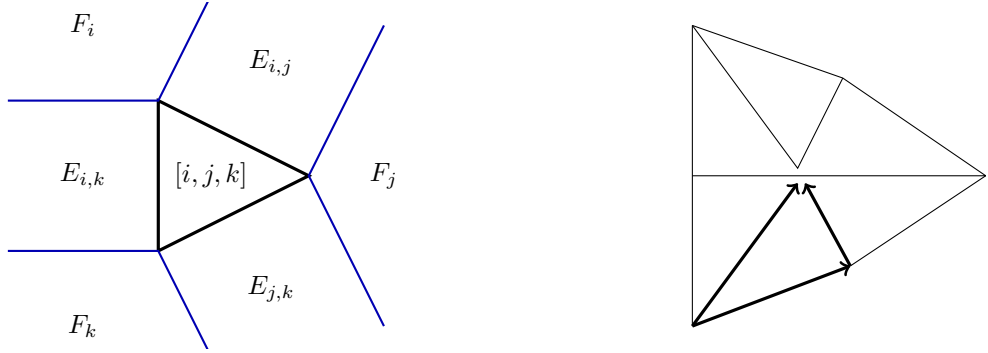
\begin{figure}[htbp]
\begin{center}
\begin{tikzpicture}
\clip (-2.5,-2.3) rectangle (11,2.3);
\coordinate (A) at (0,-1);
\coordinate (Am) at (-2,-1);
\coordinate (Ap) at (1,-3);
\coordinate (B) at (2,0);
\coordinate (Bm) at (3,-2);
\coordinate (Bp) at (3,2);
\coordinate (C) at (0,1);
\coordinate (Cm) at (-2,1);
\coordinate (Cp) at (1,3);
\coordinate (Xacb) at (.7,0);
\coordinate (Xab) at (1.5,1.5);
\coordinate (Xac) at (-1,0);
\coordinate (Xcb) at (1.5,-1.5);
%
\draw[very thick,color=black] (A)--(B)--(C)--(A);
\draw[thick, color=blue!70!black] (Am)--(A)--(Ap); 
\draw[thick, color=blue!70!black] (Cm)--(C)--(Cp); 
\draw[thick, color=blue!70!black] (Bm)--(B)--(Bp); 
\draw[color=black] (Xacb) node{$[i,j,k]$};
\draw[color=black] (Xab) node{$E_{i,j}$};
\draw[color=black] (Xac) node{$E_{i,k}$};
\draw[color=black] (Xcb) node{$E_{j,k}$};
\draw[color=black] (-1,2) node{$F_i$}
(-1,-2) node{$F_k$}
(3,0) node{$F_j$};
\begin{scope}[xshift=8cm]
\draw[very thick,color=black,->] (-.9,-2)--(.5,-.1);
\draw[very thick,color=black,->] (-.9,-2)--(1.2,-1.2);
\draw[very thick,color=black,->] (1.2,-1.2)--(.6,-.1);
\draw[color=black] (-.9,-2)--(-.9,2);
\draw[color=black] (-.9,2)--(1.1,1.3)--(3,0)--(-.9,0);
\draw[color=black] (1.1,1.3)--(.5,.1)--(-.9,2);
\draw[color=black] (3,0)--(1.2,-1.2);
\end{scope}
\end{tikzpicture}
\caption{The cyclic poset on a positive triangle coming from a positive vertex in a ribbon graph is shown. For a negative triangle we reverse all the cyclic orders. As morphisms in the category of cyclic sets and inclusions, this gives six $2$-simplices. Each gives a contribution of
\[
c_\cZ(\{k\},\{j,k\},\{i,j,k\})=\frac1{12}
\]
to the Euler number of the covering for a total of $\frac12$ for each positive vertex.}
\label{Fig: cyclic sets}
\end{center}
\end{figure}

 \begin{remark}
 Note that Theorem \ref{euler} is the generalization of Proposition \ref{euler for circle} from circle bundles to stabilized circle bundles, which is necessary in order to deduce consequences about the Legendrian isotopy class. For stabilized circle bundles we cannot argue as above. Instead, we will extract the Euler number from the handle slide bifurcation picture.  \end{remark}

\subsection{Formal triviality}\label{formal triviality}
 The goal of this section is to establish property (a) in Corollary \ref{corollarypairs}. Indeed, all our examples $\Lambda_{\pm}$ in that corollary are Legendrian links to which the following proposition applies.
 
\begin{proposition}\label{unlinked} Suppose that $\Lambda_G \subset J^1(\Sigma_G)$ is the mesh Legendrian associated to a trivalent bicolored ribbon graph $G$ whose vertices are all the same color. Then $\Lambda_G$ is trivial as a formal Legendrian.
\end{proposition}

Before we prove Proposition \ref{unlinked}, we recall the definition of a formal Legendrian.

\begin{definition}
A formal Legendrian in a contact manifold $(V^{2n+1},\xi)$ is a pair $(\Lambda, F_t)$ such that $\Lambda \subset V$ is an $n$-dimensional smooth embedded submanifold and $F_t: T \Lambda \to TV$ is a homotopy of injective bundle maps between $F_0=\text{id}_{T\Lambda}$ and a bundle map $F_1$ with image in $\xi$ such that $F_1(T\Lambda) \subset \xi$ is Lagrangian with respect to the conformal symplectic structure on $\xi$.
\end{definition}

\begin{remark}
For example, in the co-oriented case when $\xi= \ker(\alpha)$ for $\alpha$ a 1-form we require that $F_1(T\Lambda) \subset \xi$ is Lagrangian with respect to $d \alpha |_\xi$.
\end{remark}

When $\Lambda$ is connected we call $(\Lambda, F_t)$ a formal Legendrian knot and in general we call $(\Lambda, F_t)$ a formal Legendrian link. If $\Lambda$ is a genuine Legendrian submanifold of $V$, we can also think of it as a formal Legendrian $(\Lambda, F_t)$ by setting $F_t \equiv \text{id}_{T\Lambda}$, $t \in [0,1]$. A homotopy of formal Legendrians, also called a formal homotopy, is a family $(\Lambda^s,F_t^s)$.

\begin{definition}
A formal Legendrian link $(\Lambda, F_t)$ is trivial if it is formally homotopic to a union of standard Legendrian unknots, each contained in a disjoint Darboux ball (see Figure \ref{FigureDisjointUnion}).
\end{definition}

\begin{figure}[htbp]
\begin{center}
\begin{tikzpicture}[scale=.8]
\begin{scope}[yshift=-4cm,xshift=.5cm]
\draw[thick,color=blue] (-3.5,1) ellipse [x radius=2.5cm,y radius=.5cm];
 \draw[thick] (-8,0)--(8,0)--(8.4,3)--(-6.2,3)--(-8,0);
\end{scope}
\begin{scope}[xshift=5.5cm,yshift=-1cm, scale=.75]
\draw[thick,color=blue] (0,-2.5) ellipse [x radius=2.5cm,y radius=.5cm];
\begin{scope}
\draw[fill,color=white] (0,0) circle[radius=1cm];
\draw[thick,color=blue] (-2,0)..controls (-1,0) and (-1,-1)..(0,-1); 
\draw[thick,color=blue] (2,0)..controls (1,0) and (1,-1)..(0,-1);
\draw[fill,color=white] (0,0) ellipse [x radius=2.5cm,y radius=.5cm];
\end{scope}
\begin{scope}
\draw[very thick,color=blue] (0,0) ellipse [x radius=2.5cm,y radius=.5cm];
\draw[fill,color=white] (0,0.05) ellipse[x radius=1.8cm, y radius=.5cm];
\end{scope}
\draw[thick,color=blue] (-2,0)..controls (-1,0) and (-1,1)..(0,1); 
\draw[thick,color=blue] (2,0)..controls (1,0) and (1,1)..(0,1);
\draw (0.15,-1.5) node{$\downarrow p$};
\end{scope}
\begin{scope}[xshift=2.2cm,yshift=0cm, scale=.75]
\draw[thick,color=blue] (0,-2.5) ellipse [x radius=2.5cm,y radius=.5cm];
\begin{scope}
\draw[fill,color=white] (0,0) circle[radius=1cm];
\draw[thick,color=blue] (-2,0)..controls (-1,0) and (-1,-1)..(0,-1); 
\draw[thick,color=blue] (2,0)..controls (1,0) and (1,-1)..(0,-1);
\draw[fill,color=white] (0,0) ellipse [x radius=2.5cm,y radius=.5cm];
\end{scope}
\begin{scope}
\draw[very thick,color=blue] (0,0) ellipse [x radius=2.5cm,y radius=.5cm];
\draw[fill,color=white] (0,0.05) ellipse[x radius=1.8cm, y radius=.5cm];
\end{scope}
\draw[thick,color=blue] (-2,0)..controls (-1,0) and (-1,1)..(0,1); 
\draw[thick,color=blue] (2,0)..controls (1,0) and (1,1)..(0,1);
\draw (0.15,-1.5) node{$\downarrow p$};
\end{scope}
%
\begin{scope}[xshift=-3cm,yshift=-.2cm]
\begin{scope}
\draw[fill,color=white] (0,0) circle[radius=1cm];
\draw[thick,color=blue] (-2,0)..controls (-1,0) and (-1,-1)..(0,-1); 
\draw[thick,color=blue] (2,0)..controls (1,0) and (1,-1)..(0,-1);
\draw[fill,color=white] (0,0) ellipse [x radius=2.5cm,y radius=.5cm];
\end{scope}
\begin{scope}
\draw[very thick,color=blue] (0,0) ellipse [x radius=2.5cm,y radius=.5cm];
\draw[fill,color=white] (0,0.05) ellipse[x radius=1.8cm, y radius=.5cm];
\end{scope}
\draw[thick,color=blue] (-2,0)..controls (-1,0) and (-1,1)..(0,1); 
\draw[thick,color=blue] (2,0)..controls (1,0) and (1,1)..(0,1);
\draw (0.15,-1.5) node{$\downarrow p$};
\end{scope}
%
\end{tikzpicture}
\caption{A disjoint union of flying saucers lying over disjoint balls.}
\label{FigureDisjointUnion}
\end{center}
\end{figure}
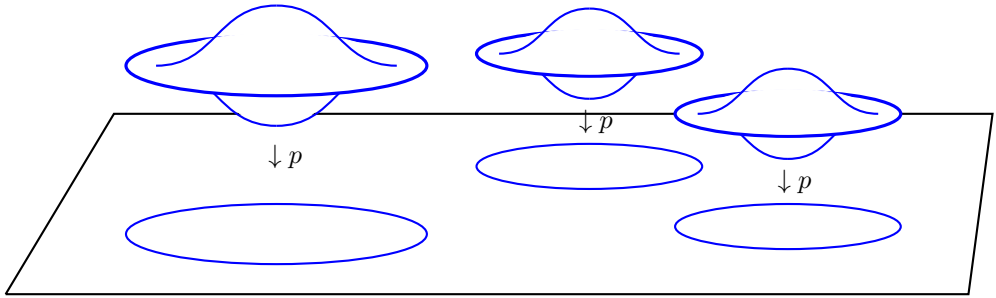

We review a useful viewpoint on formal homotopies. Recall the front projection $\pi:J^1(B) \to J^0(B)$. A generic Legendrian $\Lambda \subset J^1(B)$ is determined by its front $\pi(\Lambda) \subset J^0(B)$, because the missing coordinates $p_1, \ldots , p_n$ can be recovered by the formula $p_j = \partial z / \partial q_j$, where $z$ is the $\bR$ coordinate on $J^0(B)=B \times \bR$,  $q_1, \ldots, q_n$ are local coordinates on $B$ and $p_1 , \ldots , p_n$ are the dual coordinates. Equivalently, denote by $\text{Gr}_n(B \times \bR)$ the Grassmannian of $n$-planes in $T(B \times \bR)$ and let $P: \Lambda \to  \text{Gr}_n(B \times \bR)$ be the field of $n$-planes tangent to $\pi(\Lambda)$ (this is well defined even at a cuspidal point of the front). At a point $x \in \Lambda$ that is regular for the front projection we have $P(x)=d\pi_x(T \Lambda)$, which is non-vertical and hence equal to the graph of a linear form $\sum_{j=1}^n p_j dq_j$, from which the coordinates $p_j$ can be recovered.


Suppose that we deform the field $P=P_0$ through non-vertical plane fields $P_s:\Lambda \to \text{Gr}_n(B \times \bR)$. We obtain a deformation $(\Lambda^s,F_t^s)$ of $\Lambda$ through formal Legendrians. Indeed, the plane $P_s$ gives coordinates $p_1, \ldots , p_n$ as before which determine a smooth isotopy of submanifolds $\Lambda^s \subset J^1(B)$, and the homotopy $F^s_t$ is given by tracing back the derivative of the smooth isotopy (so $F^s_1$ has image $T \Lambda$ for all $s$). But we need to be careful. If $\pi(x)=\pi(y) \in J^0(B)$ for distinct points $x,y \in \Lambda$, then we need to ensure that the corresponding planes $P_s(x)$ and $P_s(y)$ never coincide. Otherwise $\Lambda_s$ would develop a self-intersection and hence would no longer be embedded. However, if this doesn't occur then $(\Lambda^s,F^s_t)$ is indeed a homotopy of formal Legendrians. 

\begin{proof}[Proof of Proposition \ref{unlinked}]
Suppose for concreteness that all the vertices of $G$ have positive labels (the case where all the labels are negative is entirely analogous). Set $\Lambda=\Lambda_G$ and $\Sigma=\Sigma_G$ to simplify notation. Pick a face $F_1$ of $\Sigma \setminus G$ and let $\Lambda_1$ be the corresponding component of $\Lambda$. Choose $\theta \in S^1$ a coordinate for $\partial F_1 \simeq S^1$ compatible with its boundary orientation which we extend to a tubular neighborhood $U=S^1 \times (-1,1)$ of $\partial F_1$ in $\Sigma$. Let $r \in (-1,1)$ be the collar direction, so that $\partial F_1 = \{ r=0\} \subset U$. 

The plane tangent to $\pi(\Lambda_1)$ at any point which lies over $U$ is spanned by a vector of the form $ \partial_\theta + a_1 \partial_z$ and a vector of the form $ \partial_r + b_1 \partial_z$. Take any other component $\Lambda_2$ of $\Lambda$ corresponding to another face $F_2$ of $\Sigma \setminus G$ which shares an edge with $F$. Then the plane tangent to $\pi(\Lambda_2) \subset J^0(\Sigma)=\Sigma \times \bR$ at any point which lies over $U$ is spanned by a vector of the form $  \partial_\theta + a_2 \partial_z$ and a vector of the form $ \partial_r + b_2 \partial_z$. Here $a_i,b_i \in \bR$ for $i=1,2$ and the key fact is that along the intersection locus $\pi(\Lambda_1) \cap \pi(\Lambda_2)$ we have either $b_1 \neq b_2$ or $a_1>a_2$, which can be verified by inspection. 

Consider the homotopy of plane fields along $\Lambda_1$ given by the formula $$P_s : \Lambda_1 \to T(\Sigma \times \bR), \qquad s \mapsto \text{span} \big\{ (1-s)\partial_\theta + a_1\partial_z, \, \,  \partial_r + b_1 \partial_z \big\},$$ cut off to the identity in the interior of the face $F_1$ (where there are no interactions with any other components of $\Lambda$) and stopping at $s$ near but strictly smaller than $1$. Observe that $P_s$ is always distinct to the plane field tangent to $\pi(\Lambda_2)$ along $\pi(\Lambda_1) \cap \pi(\Lambda_2)$. Indeed, if they happened to coincide, then the following matrix would have rank 2.

\[ \mat{1-s & 0 & 1 & 0 \\
0 & 1 & 0 & 1 \\
a_1& b_1 &a_2 &b_2 } \]

This implies $b_1=b_2$ and $a_1=a_2(1-s)<a_2$, a contradiction. Therefore we obtain a homotopy of formal Legendrians $(\Lambda^s ,F^s_t)$ with $\Lambda^0=\Lambda$, $F^0_t \equiv \text{id}_{T \Lambda}$. Note that every component except $\Lambda_1$ is fixed. 

After this deformation, we can further homotope $\Lambda_1$ through formal Legendrians into a small Darboux ball lying over the interior of $F_1$. There is no risk of creating self intersections due to the large $p$ coordinate of the deformed $\Lambda_1$ in the region where it overlaps with the other components of $\Lambda$. Once we have put $\Lambda_1$ in this ball we can undo the deformation of the plane field, i.e. we can homotope it through formal Legendrians to the standard Legendrian unknot. We then inductively repeat this procedure for every face of $\Sigma \setminus G$. This completes the proof.  \end{proof}

 
  
 
 
 
 
 




\subsection{0-parametric Morse theory}

We now turn our attention to the Morse theory of evenly stabilized circle bundles. First, consider the case where $B$ is a point, i.e. consider a single Morse function on the evenly stabilized circle $W=S^1\times \bR^{4k}$. We will only consider functions which are a bounded $C^1$ distance from the standard quadratic form $||x||^2-||y||^2$, where $(x,y) \in \bR^{2k} \times \bR^{2k}=\bR^{4k}$. Let $U(R)$ be the group of units of the commutative ring $R=\bZ[u,u^{-1},(1-u)^{-1}]$. This is shorthand for $\bZ[u,v,w]/(uv=1,w(1-u)=1)$. Let $\rho_0:\pi_1W \to U(R)$ be the representation which sends the oriented circle $[S^1]$ to $u^{-1}$. Recall that an Euler structure on a Morse function with two critical points consists of a path between those two points. In this case this is the same as a weak Euler structure.


\begin{lemma}\label{Lemma 0}
Let $f$ be a Morse function on $W=S^1\times \bR^{4k}$ which is a bounded $C^1$ distance from the standard quadratic form. Then for $\rho_0$ as above and any weak Euler structure $e$, we have
\[
	\tau(f,e,\rho_0)=\pm u^n(1-u).
\]
for some choice of sign $\pm1$ and some $n \in \bZ$.
\end{lemma}

\begin{proof}
Observe that there is one Morse function $f_0$ which is a bounded $C^1$ distance from the standard quadratic form and has torsion $1-u$, namely the stabilization of the height function on $S^1$. By Lemma \ref{lemma: we need C1}, $f$ is homotopic to $f_0$ through fibrations at infinity. During the homotopy the fibre Turaev torsion can change, but the Reidemeister torsion stays the same. The Reidemeister torsion is the image of the Turaev torsion under the map $U(R) / \pm 1 \to U(R)/\pm \rho(\pi_1F)$. Hence the Turaev torsion can only change by multiplication by elements of $\pm \rho(\pi_1F)=\{\pm u^n\}$. We conclude that the possible values of the Turaev torsion are $\pm u^n(1-u)$, as claimed.
\end{proof}

\begin{lemma}\label{Lemma 0.0}
Let $f$ be a Morse function on $W=S^1\times \bR^{4k}$ which is a bounded $C^0$ distance from the standard quadratic form and has exactly two critical points. Then their indices are $2k$ and $2k+1$.
\end{lemma}

\begin{proof}
The indices are determined for homological reasons, see Lemma \ref{Lemma: bars}.
\end{proof}

The coefficient ring $R=\bZ[u,u^{-1},(1-u)^{-1}]$ is also appropriate for any trivial stabilized $S^1$ bundle $W=B\times S^1\times \bR^{4k}$. Suppose $f:W\to \bR$ is of Euler type, see Definition \ref{Def: function euler}, and assume that the restriction of $f$ to each fibre is a bounded $C^1$ distance from the standard quadratic form. Let $\rho:\pi_1W\to U(R)$ be given by $\rho_0:\pi_1S^1\to U(R)$ composed with the map $\pi_1W\to \pi_1S^1$ induced by the projection $W\to S^1$. Note that any element $x\in R$ can be written as $x=1-v$ (let $v=1-x$).

\begin{proposition}\label{prop: un=vm lemma}
Suppose that the fibre Turaev torsion of $f$ with respect to $\rho$ is $\tau(f,\rho)=\pm(1-v)\in R$ and $v^n=u^m$ for integers $n,m$ with $n>0$. Then $|m|=n$ and $v=u^\varepsilon$ where $\varepsilon=m/n$ is the sign of $m$.
\end{proposition}

\begin{proof}
We prove this by taking representations $R\to \bC$. For any small $\theta>0$ let $\rho_\theta:R\to\bC$ be the unique ring homomorphism so that $\rho_\theta(u)=e^{i\theta}$. To second order in $\theta$ this is $1+i\theta-\frac12\theta^2$. Then $\rho_\theta(v^n)=\rho_\theta(u)^m=e^{im\theta}$. This implies that $\rho_\theta(v)=e^{i\psi}$ where $\psi=\frac mn\theta$ plus an integer multiple of $2\pi/n$. By Lemma \ref{Lemma 0}, $|\rho_\theta(1-v)|=|\rho_\theta(1-u)|\approx \theta$ which is very small. In other words, $\rho_\theta(v)$ is very close to 1. So, $\psi=\frac mn\theta$. By Lemma \ref{Lemma 0} we have $1-v= \pm u^k(1-u)$ for some integer $k$. But:
\[
	\rho_\theta(1-v)\approx -\frac mn\theta i+\frac{m^2}{2n^2}\theta^2
\]
\[
	\rho_\theta\big(\pm u^k (1-u)\big)\approx \pm e^{ik\theta}\left(
	-i\theta +\frac12 \theta^2
	\right)\approx \pm \left(
	-i\theta +k\theta^2+\frac12\theta^2
	\right)
\]
to second order in $\theta$. Comparing the linear terms we see that $|m|= n$ and the sign $\pm$ in the second equation is the sign $\ve$ of $m$. Comparison of the $\theta^2$ terms give two cases.
\begin{enumerate}
\item $\ve=+$. Then $k=0$ and $v=u$.
\item $\ve=-$. Then $k=-1$ and $1-v=-u^{-1}(1-u)=1-u^{-1}$. So $v=u^{-1}$.
\end{enumerate}
In both cases, $v=u^\ve$ where $\ve=m/n$ proving the proposition.
\end{proof}

\subsection{1-parametric Morse theory}
 
 A bundle $F\to W\to J$ over an interval $J\subset \bR$ is always trivial. So $W=J\times F$ and any smooth function $f:W\to\bR$ can be viewed as a 1-parameter family of functions $f_t:F\to\bR, t\in J$. Let $F=S^1\times \bR^{4k}$ be the evenly stabilized circle and let $\rho:\pi_1W \to U(R)$ be as above, namely the representation which sends the oriented circle $[S^1]$ to $u^{-1}$, where $R=\bZ[u,u^{-1},(1-u)^{-1}]$.
 
 A generic 1-parameter family of functions $f_t:F\to \bR$ will have birth-death points and handle slides at isolated times. Handle slides, also known as $i/i$ incidences, occur when, at some parameter value $t_0 \in J$, a trajectory of the gradient-like vector field $Z$ goes between two Morse critical points of $f_{t_0}$ of the same index $i$, say $x_a$, $x_b$, with $f_{t_0}(x_a)<f_{t_0}(x_b)$. Such a handle slide is denoted $x_{ab}^{\pm s}$. Here $s=\rho(\sigma)$ for $\sigma\in \pi_1 F$ the homotopy class of $p_b\# \gamma \# \overline {p}_a$, where $p_a$ is the chosen path from the basepoint to $x_a$ and $\overline{p}_a$ is the path $p_a$ but with reversed orientation. The sign $\ve=\pm 1$ is determined by comparing the orientation of the descending manifold of $x_b$ with the normal orientation of the ascending manifold of $x_a$ at the moment these cross.
 
 As $t$ crosses $t_0$ the matrix for the boundary map $\partial_i:C_i(f_t;R^\rho)\to C_{i-1}(f_t;R^\rho)$ will change by the column operation which adds $\varepsilon \rho(\sigma)$ times the $a$-th column to the $b$-th column. Recall from Remark \ref{second remark about IK conventions} that the entries in these matrices are in the opposite ring $R^{op}=R$. Indeed, if $\gamma'$ is a trajectory of $-Z$ from $x_a$ to a critical point $y_c$ of index $i-1$ contributing $\rho(\tau)$ to the $(c,a)$ entry of the matrix of $\partial_i$ (where $\tau=[p_a\# \gamma' \# \overline{p}_c$]), then a new trajectory from $x_b$ to $y_c$ will be created (namely $p_b\# \gamma\# \gamma' \# \overline{p}_c$), whose homotopy class is $\sigma\tau$. The column operation $E_{ab}^{\varepsilon \rho(\sigma)}$ corresponding to the handle slide $x_{ab}^{\varepsilon\rho(\sigma)}$ described above will then add $\rho(\tau)\rho(\sigma)$ to the $(c,b)$ entry of the matrix of $\partial_i$. But $\rho(\tau)\rho(\sigma)=\rho(\sigma\tau)$ since $R$ is commutative. This column operation will be denoted $x_{ab}^{\varepsilon\rho(\sigma)}$, the same as the geometric handle slide, when there is no possibility of confusion. Following \cite{IK93}, row operations, coming from $i-1/i-1$ handle slides given by a trajectory of $-Z$ from $y_b$ down to $y_a$, which add a multiple of the $b$-th row of the matrix of $\partial_i$ to the $a$-th row of that matrix, will be denoted $y_{ab}^s$.
 
 An important property of row and column operations on matrices is that multiplication of $x_{ab}^{\sigma}$ is additive in $\sigma$:
 \[
 	\prod_i x_{ab}^{\sigma_i}=x_{ab}^{\sum \sigma_i}
 \]
 We refer to this property as the additivity of row and column operations.

 Figure \ref{Fig: link} shows the graphic of a 1-parameter family of functions on the evenly stabilized circle $F=S^1\times \bR^{4k}$ parametrized by an interval $J=[t_-,t_+]$.

\begin{figure}[htbp]
\begin{center}
\begin{tikzpicture}
\coordinate (A1) at (-1,.35);
\coordinate (A2) at (-1,1.1);
\coordinate (A0) at (-1,.7);
\draw[<->] (A1)--(A2);
\draw (A0) node[left]{$-a$};
\coordinate (B1) at (1,.35);
\coordinate (B2) at (1,1.1);
\coordinate (B0) at (1,.7);
\draw[<->] (B1)--(B2);
\draw (B0) node[right]{$b$};
\coordinate (C1) at (-1,-.35);
\coordinate (C2) at (-1,-1.1);
\coordinate (C0) at (-1,-.7);
\draw[<->] (C1)--(C2);
\draw (C0) node[left]{$-b$};
\coordinate (D1) at (1,-.35);
\coordinate (D2) at (1,-1.1);
\coordinate (D0) at (1,-.7);
\draw[<->] (D1)--(D2);
\draw (D0) node[right]{$a$};
\coordinate (2K) at (-3.1,-1.2);
\draw(2K) node[left]{$y_i$};
\coordinate (2Kp) at (-3.1,1.2);
\draw(2Kp) node[left]{$x_i$};
\coordinate (Lb) at (-3,-2.2);
\coordinate (Cb) at (0,-2.2);
\coordinate (Rb) at (3,-2.2);
\coordinate (Lbu) at (-3,-2);
\coordinate (Cbu) at (0,-2);
\coordinate (Rbu) at (3,-2);
\draw(Lb) node{$t_-$};
\draw(Cb) node{$t_0$};
\draw(Rb) node{$t_+$};
\draw(Lbu) node[above]{$\uparrow$};
\draw(Cbu) node[above]{$\uparrow$};
\draw(Rbu) node[above]{$\uparrow$};
\coordinate (2Kr) at (3.1,-1.2);
\draw[color=blue] (2Kr) node[right]{$y_j$};
\coordinate (2Kpr) at (3.1,1.2);
\draw[color=blue] (2Kpr) node[right]{$x_j$};
\clip (-3,-2) rectangle (3,1.5);
\begin{scope}[xshift=-3cm] 
\draw[very thick] (-1,1.25)--(1,1.25)..controls (3,1.25) and (4,0)..(5,0);
\draw[very thick] (-1,-1.25)--(1,-1.25)..controls (3,-1.25) and (4,0)..(5,0);
\end{scope}
\begin{scope}[xshift=1cm] 
\draw[very thick, color=blue] (-3,0)..controls (-2,0) and (-1,1.25)..(1,1.25)--(3,1.25);
\draw[very thick, color=blue] (-3,0)..controls (-2,0) and (-1,-1.25)..(1,-1.25)--(3,-1.25);
\end{scope}
\end{tikzpicture}
\caption{Portion of the front projection of $\Lambda_G$ lying over a curve transverse to an edge of $G$: $t_-,t_0,t_+$ are parameter values along this curve. The critical points $y_i,y_j$ have index $2k$ and $x_i,x_j$ have index $2k+1$. The four two-headed vertical arrows indicate the possible handle slides which we claim must be: $y_{ij}^{-b},x_{ji}^{-a}, x_{ij}^{b},y_{ji}^{a}$.
}
\label{Fig: link}
\end{center}
\end{figure}
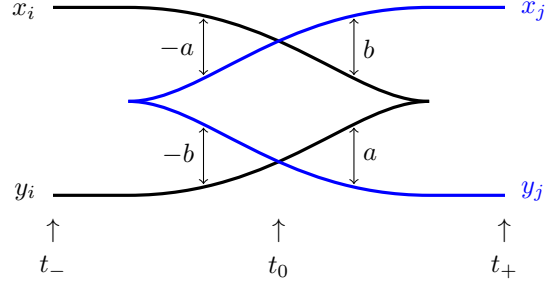

 \begin{lemma}\label{lem: ab=v}
Consider a 1-parameter family of functions on the evenly stabilized circle $S^1\times \bR^{4k}$ which is a bounded $C^1$ distance from the standard quadratic form and whose Cerf diagram is as shown in Figure \ref{Fig: link}. Choose orientations at the critical points so that the incidence at each birth-death point is $+1$. Then the four handle slides which occur are $y_{ij}^{-b},x_{ji}^{-a}, x_{ij}^{b},y_{ji}^{a}$ for some $a,b\in R$ such that
 $ab=v$, where $\tau(f,\rho)=\pm(1-v)$ is the fibre Turaev torsion of the function. 
 \end{lemma}
 
  \begin{remark}
In the following proof and in what follows we use the convention that if a pair of critical points have died or haven't been born yet then we still include them in the incidence matrix with a 1 in the diagonal and 0s elsewhere. This way all the incidence matrices have the same size. There is no room for confusion since in the examples at hand the critical points are always paired up such that each pair does not interact (i.e. has no birth/death) with the other pairs.
 \end{remark}
 
 \begin{proof}
 Take the parameter space to be the interval $J=[t_-,t_+]$ with center point $t_0$ where there are only two critical values. Since handle slides can only occur between critical points of the same index, the only possible handle slides are as indicated in Figure \ref{Fig: link}. There may be multiple handle slides. However, by additivity of handle slides, there are, algebraically, only four handle slides: $y_{ij}^{-b}, x_{ji}^{-a}$ on the left and $x_{ij}^{b'},y_{ji}^{a'}$ on the right for $a,b,a',b'\in R$, where we claim that $a'=a$ and $b'=b$.

 By assumption, the incidence matrices of $f_{t_-}$ and $f_{t_+}$ have the form $\mat{x&0\\0&1}$ and $\mat{1 & 0\\0& y}$ respectively for some $x,y\in R$ which we can write $x=1-v$, $y=1-v'$. 
 
 At $t_0$ it becomes
 \[
 	\mat{1-v+ba & -b\\ -a & 1}=\mat{1 & -b'\\ -a' & 1-v'+a'b'}.
 \]
 Comparing entries we see that $a'=a, b'=b$ and $ab=v=v'$. Thus $\tau(f,\rho)=\pm(1-v )=\pm(1-v' )$.
 \end{proof}
 
\subsection{2-parametric Morse theory}

In a 2-parameter family of functions there are additional bifurcations which will generically occur. These consist of exchange points and double handle slides, which respectively give exchange and Steinberg relations among row and column operations. Although the Steinberg relations were crucial in the study of higher Reidemeister torsion in \cite{IK93}, in the present study they will not play a role, so we ignore them and deal only with the exchange points. Another bifurcation which generically occurs in 2-parameter families is the swallowtail (quartic) singularity, where two arcs of cusps are born/die together. However, mesh Legendrians don't have any swallowtails, so we won't need to discuss them either. 

In a family of Morse functions $f_t$ on a manifold $M$, an \emph{exchange} occurs at $t_0$ when a trajectory, say $\gamma$, of $-Z$, the negative gradient-like vector field for $f_{t_0}$, goes from a critical point $y_b$ of index $i$ down to a critical point $x_a$ of index $i+1$. This occurs only at isolated parameter values in a 2-parameter family of functions. The exchange is labelled $Z_{ab}^{\pm s}$ where $s=\rho(\sigma)$ and $\sigma\in \pi_1M$ is the homotopy class of $p_b\# \gamma\# \overline{p}_a$, where the notation and sign is as in the previous subsection. 
 In our case we will see that both $s$ and its sign are uniquely determined by 
 the Cerf diagram of $f_t$.

At the exchange point $Z_{ab}^{\pm s}$, the critical value of $y_a$ is above the critical value of $x_b$. So, the $(a,b)$-entry of the matrix of the boundary map $\partial_{i+1}:C_{i+1}\to C_i$ must be zero.
\[
\partial_{i+1}=\qquad \begin{matrix}\ \\ p\\a
\end{matrix}
\begin{matrix}
\begin{matrix} \  & b\quad &  & q\end{matrix}\\
\mat{ 
\cdot & c_{pb} & \cdot & c_{pq}\\
\cdot & 0 & \cdot & c_{aq} 
}
\end{matrix}
\qquad
\begin{matrix}
 \xrightarrow{x_{bq}^{sc_{aq}}}\\
 \xleftarrow{y_{pa}^{-c_{pb}s}}
 \end{matrix}
\qquad \begin{matrix}\ \\ p\\a
\end{matrix}
\begin{matrix}
\begin{matrix} \  & b\quad &  & q\end{matrix}\\
\mat{ 
\cdot & c_{pb} & \cdot & c_{pq}+c_{pb}sc_{aq}\\
\cdot & 0 & \cdot & c_{aq} 
}
\end{matrix}
\]

At the exchange point $Z_{ab}^{s}$ a family of handle slides is created: for every critical point $y_p\neq y_a$ of index $i$ and every critical point $x_q\neq x_b$ of index $i+1$ a pair of handle slides $x_{bq}^{s c_{aq}}$ and $y_{pa}^{-c_{pb}s}$ are created. The corresponding pair of row and column operations will add and subtract the quantity $c_{pb}s c_{aq}$ to $c_{pq}$, the $(p,q)$ entry of the matrix for $\partial_{i+1}$, resulting in no change in the matrix. Another viewpoint is that the column operation $x_{bq}^{s c_{aq}}$ is ``exchanged'' for the row operation $y_{pa}^{c_{pb}s}$ since these do the same thing to the matrix of $\partial_{i+1}$. See \cite{HW73} for details in general and \cite{IK93} for the details in the particular case of 2-parameter Morse theory on circle bundles.

As in the case of handle slides, exchange points are additive in the sense that a collection of exchange points $Z_{ab}^{s_i}$ of $y_b$ over $x_a$ has the same algebraic effect as the algebraic exchange $Z_{ab}^{\sum s_i}$.

%
%

\begin{lemma}\label{possible exchange points}
Suppose that a mesh Legendrian $\Lambda_G$ is generated by a function $f$ on a stabilized circle bundle $W \to \Sigma_G$ which is a bounded $C^0$ distance from the standard quadratic form. Then the only possible exchange points of the function occur at the corners of the triangles given by the vertices of $G$. Furthermore, at a corner where two birth-death lines lying in components $C_i,C_j$ of the singular set of $f$ cross and where $f$ is larger on $C_j$ than on $C_i$, the only possible exchanges near that corner have the form $Z_{ij}^c$ for some $c$ in the coefficient ring.
\end{lemma}

\begin{proof}
Over the edges of $G$, the Cerf diagram of the generating family, which is the front projection of $\Lambda_G$, is precisely as in Figure \ref{Fig: link}. In particular there are no possible exchange points since the critical values are in order of index. 

Around each positive vertex of $G$ we have a triangle in $\Sigma_G$ bounded by three birth-death lines as indicated in Figure \ref{Fig: birth-death triangle} on the right. The graphic for the family of functions along the $z_3$ birth-death line is indicated on the left side of Figure \ref{Fig: birth-death triangle}. 
Denote the singular components of $f:W \to \bR$ over this triangle by $C_1,C_2,C_3 \subset W$. 
The critical points of $f$, which come in pairs $x_i,y_i$ in each component $C_i$ of $\Lambda_G$, have index $2k, 2k+1$ respectively, see Lemma \ref{Lemma 0.0}. Denote by $z_i$ the birth/death of $x_i$ and $y_i$.

For example, at the bottom corner of the triangle, the 
critical values satisfy $f(y_3)<f(z_1)<f(z_2)<f(x_3)$ as indicated in the middle part of Figure \ref{Fig: birth-death triangle}. Thus, in the bottom shaded region on the right side of Figure \ref{Fig: birth-death triangle}, the function $f$ is Morse and the critical values of $f$ are in the order:
\[
	f(y_3)< f(y_1)< \boxed{f(x_1)< f(y_2)}<f(x_2)<f(x_3).
\]
An exchange $Z_{12}$ is possible anywhere we have $f(x_1)<f(y_2)$. This is the region shaded near the bottom corner on the right side of Figure \ref{Fig: birth-death triangle}. Note that, at the corners, and in fact along the entire birth-death line, there can be no exchange points since, for stabilized functions, we can arrange that birth-death points are independent of all other critical points. Thus, there are no exchange point along the boundary of the triangle and the shaded regions in Figure \ref{Fig: birth-death triangle} are disjoint from this boundary.

Similarly, the critical values at the other two corners, indicated in the left part of Figure \ref{Fig: birth-death triangle}, must come in the order $f(y_1)<f(z_2)<f(z_3)<f(x_1)$ and $f(y_2)<f(z_3)<f(z_1)<f(x_2)$. The only constraint on the critical values of the other functions in the two parameter family are given by a system of convex inequalities.
In particular, the only possible exchange points are $Z_{12},Z_{23}$ and $Z_{31}$. 
Furthermore, each exchange precludes the other two since, e.g., $f(x_2)<f(y_3)$ gives 
\begin{equation}\label{eq: no concurrent exchanges}
f(y_1),f(y_2)<f(x_2)<f(y_3)<f(x_3),f(x_1).
\end{equation}
This completes the proof. \end{proof}

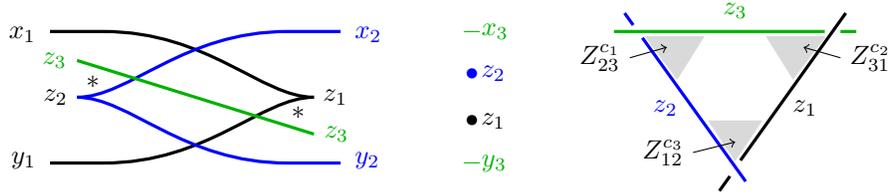
\begin{figure}[htbp]
\begin{center}
\begin{tikzpicture}[scale=.7]
%
\begin{scope}
\coordinate (2K) at (-3.1,-1.2);
\draw(2K) node[left]{$y_1$};
\coordinate (2Kp) at (-3.1,1.2);
\draw(2Kp) node[left]{$x_1$};
\coordinate (2Kr) at (2.6,-1.2);
\draw[color=blue] (2Kr) node[right]{$y_2$};
\coordinate (2Kpr) at (2.6,1.2);
\draw[color=blue] (2Kpr) node[right]{$x_2$};
\clip (-3,-2) rectangle (2.5,1.5);
\begin{scope}[xshift=-3cm] 
\draw[very thick] (-1,1.25)--(1,1.25)..controls (3,1.25) and (4,0)..(5,0);
\draw[very thick] (-1,-1.25)--(1,-1.25)..controls (3,-1.25) and (4,0)..(5,0);
\end{scope}
\begin{scope}[xshift=.5cm] 
\draw[very thick, color=blue] (-3,0)..controls (-2,0) and (-1,1.25)..(1,1.25)--(3,1.25);
\draw[very thick, color=blue] (-3,0)..controls (-2,0) and (-1,-1.25)..(1,-1.25)--(3,-1.25);
\end{scope}
\end{scope}

\draw(2,0) node[right]{$z_1$};
\draw(-2.5,0) node[left]{$z_2$};
\begin{scope}
\draw[very thick, color=green!70!black] (-2.5,.7) -- (2,-.7)
(-2.5,.7) node[left]{$z_3$}
(2,-.7) node[right]{$z_3$};
\draw (-2.2,.3)node{$\ast$};
\draw (1.7,-.3)node{$\ast$};
\end{scope}
\begin{scope}[xshift=5cm]
\draw[ thick, color=green!70!black] (0, 1.25) node{$-$} node[right]{$x_3$}
(0, -1.25) node{$-$} node[right]{$y_3$};
\draw (0, -.45) node{$\bullet$}node[right]{$z_1$};
\draw[color=blue] (0, .45) node{$\bullet$}node[right]{$z_2$};
\end{scope}
\begin{scope}[xshift=10cm]
\draw[very thick,color=blue] (-2.1,1.6)--(0.2,-1.6) (-.9,-.2)node[left]{$z_2$};
\draw[fill,color=white] (-1.85,1.25) circle[radius=4pt];
\draw[very thick,color=green!70!black] (-2.3,1.25)--(2.3,1.25)
(0,1.3)node[above]{$z_3$};
\draw[fill,color=white] (1.85,1.25) circle[radius=4pt];
\end{scope}
\begin{scope}[xshift=10cm]
\draw[very thick] (2.1,1.6)--(.1,-1.2)
(-.1,-1.5)--(-.3,-1.78) (.9,-.2)node[right]{$z_1$};
\end{scope}
\begin{scope}[xshift=10cm, yshift=-.05cm]
\draw[fill,color=gray!30!white] (-1.7,1.2)--(-.6,1.2)--(-1.1,.4)--(-1.7,1.2);
\draw[fill,color=gray!30!white] (1.7,1.2)--(.6,1.2)--(1.1,.4)--(1.7,1.2);
\draw[fill,color=gray!30!white] (0,-1.15)--(-.5,-.4)--(.5,-.4)--(0,-1.15);
\draw[<-] (-1.3,1)--(-2,.8)node[left]{$Z_{23}^{c_1}$}; 
\draw[<-] (1.3,1)--(2,.8)node[right]{$Z_{31}^{c_2}$}; 
\draw[<-] (0,-.8)--(-.8,-1)node[left]{$Z_{12}^{c_3}$}; 
\end{scope}
\end{tikzpicture}
\caption{Critical points of the lower index $2k$ are denoted $y_i$. The upper index critical points are $x_i$. Corresponding birth-death points are indicated by $z_i$. The triangle on the right is the triangle of birth death lines projected to the base with overcrossing indicating that $z_3$ is above $z_2$, etc. at the corners. The shaded region indicates the position of possible exchange points. The critical values of $x_3,y_3,z_1,z_2$ at the bottom corner are indicated by the figure in the middle. On the left side is a possible graphic for the function along the $z_3$ birth-death line with possible exchanges indicated with $\ast$. See also Figure \ref{FigureT} and the left side of \ref{FigureR} showing the same thing.
}
\label{Fig: birth-death triangle}
\end{center}
\end{figure}

Since the second cohomology of any proper subset of $\Sigma_G$ is zero, the restriction of $W$ to any such subset will be trivial and we can take the canonical representation in the ring $R=\bZ[u,u^{-1},(1-u)^{-1}]$ which sends the oriented circle $[S^1]$ to $u^{-1}$. Later we specialize $u=\zeta$, a root of unity. In particular, over a neighborhood of the triangle surrounding one vertex, the mesh Legendrian $\Lambda_G$ will have 3 contractible components which map to the base with a degree zero map. With respect to the weak Euler structure coming from $\Lambda_G$, the fibre Turaev torsion is well defined as an element of $U(R)/\pm1$. 
We call this the universal fibre Turaev torsion at the vertex.

 \begin{lemma}\label{a3=v2 and b3=v}
Given any generating family $f$ for $\Lambda_G$ with universal fibre Turaev torsion $\tau(f,\rho)=\pm(1-v)$ at one chosen vertex, let $(a_i,b_i)$ for $i=1,2,3$ be the handle slide labels on the edges adjacent to this vertex. If the vertex is positive we have $a_1a_2a_3=v^2$ and $b_1b_2b_3=v$. If the vertex is negative then $a_1a_2a_3=v$ and $b_1b_2b_3=v^2$.
 \end{lemma}
 
 \begin{proof}
 At the vertices of a positive triangle 
 the exchange points can be ``pushed'' to the corner. Indeed, there is no obstruction to moving exchange points in convex parameter regions in which there are no other critical points with critical value between the two incident to give the exchange, as is the case here by \eqref{eq: no concurrent exchanges}. More precisely, by a homotopy of the pseudo-gradient for $f$ we can push all the exchanges in each of the shaded regions into small enough neighborhoods of the vertices, where no other handle slides occur. 
After doing this, there exists a simple closed curve $\gamma$ contained in the interior of the triangle and satisfying the following properties, see Figure \ref{Fig: gamma}.
\begin{itemize}
\item[(a)] There are no exchanges inside the disk $D$ bounded by $\gamma$.
\item[(b)] $\gamma$ intersects the handle slides $a_i$, $b_i$ along each side of the triangle and the handle slides produced by each exchange, and nothing more.
\end{itemize}

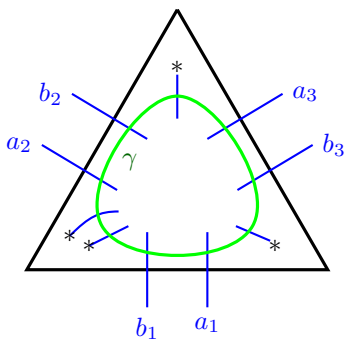
\begin{figure}[htbp] 
\begin{center}
\begin{tikzpicture}[scale=1]

\begin{scope}[scale=2,yshift=-2.9mm]
\coordinate (A) at (0,1.73);
\coordinate (B) at (-1,0);
\coordinate (C) at (1,0);
\draw[very thick] (A)--(B)--(C)--(A);
\begin{scope}[xshift=-1mm,yshift=1.7mm]
\coordinate (Y1) at (.8,1);
\coordinate (Y2) at (.3,.7);
\draw[thick,color=blue] (Y2)--(Y1) node[right]{$a_3$};
\end{scope}
\begin{scope}[xshift=1mm,yshift=1.7mm]
\coordinate (X1) at (-.8,1);
\coordinate (X2) at (-.3,.7);
\draw[thick,color=blue] (X2)--(X1) node[left]{$b_2$};
\end{scope}
\begin{scope}[xshift=-1mm,yshift=-1.7mm]
\coordinate (X1) at (-.8,1);
\coordinate (X2) at (-.3,.7);
\draw[thick,color=blue] (X2)--(X1)node[left]{$a_2$};
\end{scope}
\begin{scope}[xshift=1mm,yshift=-1.7mm]
\coordinate (Y1) at (.8,1);
\coordinate (Y2) at (.3,.7);
\draw[thick,color=blue] (Y2)--(Y1) node[right]{$b_3$};
\end{scope}
\begin{scope}
\draw[thick,color=blue] (-.2,.25)--(-.2,-.25) node[below]{$b_1$};
\draw[thick,color=blue] (.2,.25)--(.2,-.25)node[below]{$a_1$};
\end{scope}
\end{scope}

\begin{scope}[scale=1.3,yshift=-.1cm]
\coordinate (A) at (0,1.73);
\coordinate (Ap) at (0,1.65);
\coordinate (App) at (0,1.2);
\coordinate (B) at (-1.1,0);
\coordinate (Bp) at (-.9,-.1);
\coordinate (C) at (1,-.1);
\draw[color=blue, thick] (Ap)--(App) (.95,-.05)--(.6,.1)
(B) .. controls ( -.9,.2) and (-.8,.25)..(-.6,.25)
(Bp)--(-.5,.1);
\draw (A) node{$\ast$};
\draw (B) node{$\ast$};
\draw (Bp) node{$\ast$};
\draw (C) node{$\ast$};
\end{scope}
\begin{scope}
\coordinate (A) at (0,1.73);
\coordinate (Ap) at (-.52,1.73);
\coordinate (Am) at (.52,1.73);
\coordinate (B) at (-1,0);
\coordinate (Bp) at (-1.3,.52);
\coordinate (Bm) at (-.7,-.52);
\coordinate (C) at (1,0);
\coordinate (Cp) at (1.3,.52);
\coordinate (Cm) at (.7,-.52);
\draw[color=green, very thick] (A)..controls (Am) and (Cp) ..(C)..controls (Cm) and (Bm)..(B)..controls (Bp) and (Ap)..(A) ;
\draw[color=green!40!black] (-.63,.85) node{$\gamma$};
\end{scope}

\end{tikzpicture}
\caption{The curve $\gamma$.}
\label{Fig: gamma}
\end{center}
\end{figure}
%



 By additivity of exchange, at each of the vertices of the triangle the exchange points can be algebraically collected together to give three algebraic exchange points $Z_{12}^{c_3}$, $Z_{23}^{c_1}$ and $Z_{31}^{c_2}$ for some $c_i\in R$. 
 Let us examine what the exchange $Z_{23}^{c_1}$ does. It occurs near the two birth-death points $z_2,z_3$. So, the incidence matrix is the diagonal matrix with diagonal entries $(1-v,1,1)$ for $v\in R$. The exchange creates two handle slides $x_{23}^{c_1}$ and $y_{23}^{-c_1}$. The other two exchange points have a similar effect. 
We claim that the sequence of column operations and row operations going around $\gamma$ (with the exchange points on the outside) must separately cancel out. Indeed, the disk $D$ gives a null homotopy of these sequences of row and column operations. More precisely, by contracting $\gamma$ inside $D$ we can simplify the sequence of column and row operations using Steinberg relations until we get the empty sequence. By property (a) there are no exchanges between the row and column operations inside the disk. Let us therefore focus on just the column operations. From property (b), this means that the product of the following column operations must be trivial, i.e., must give the identity matrix:
 \[
 	(x_{23}^{c_1}x_{21}^{-a_3}x_{12}^{b_3})
	(x_{31}^{c_2}x_{32}^{-a_1}x_{23}^{b_1})
	(x_{12}^{c_3}x_{13}^{-a_2}x_{31}^{b_2})=I_3.
 \]
 We will solve this equation for the variables. Let $A$ be the product of the first five factors and let $B$ be the product of the last four factors. Then we are given that $AB=I_3$ or $A=B^{-1}$. We compute using the identity $a_ib_i=v$ given by Lemma \ref{lem: ab=v}.
 \[
 	A=x_{23}^{c_1}x_{21}^{-a_3}x_{12}^{b_3}x_{31}^{c_2}x_{32}^{-a_1}=\mat{1 & b_3 & 0\\
	-a_3 & 1-a_3b_3 & c_1\\
	0 & 0 & 1
	}x_{31}^{c_2}x_{32}^{-a_1}
 \]
 \[
 	= \mat{1 & b_3 & 0\\
	c_1c_2-a_3 & 1-v & c_1\\
	c_2 & 0 & 1}x_{32}^{-a_1}
	= \mat{1 & b_3 & 0\\
	c_1c_2-a_3 & 1-v-c_1a_1 & c_1\\
	c_2 & -a_1 & 1}.
 \]
 This must be equal to:
 \[
 	B^{-1}=x_{31}^{-b_2}x_{13}^{a_2}x_{12}^{-c_3}x_{23}^{-b_1}=\mat{1 & 0 & a_2\\
	0 & 1 & 0\\
	-b_2 & 0 & 1-v}x_{12}^{-c_3}x_{23}^{-b_1}
 \]
 \[
 	= \mat{
	1 & -c_3 & a_2+c_3b_1\\
	0 & 1 & -b_1\\
	-b_2 & b_2c_3 & 1-v-b_2c_3b_1
	}
 \]
 Equating the corresponding off-diagonal entries we get $c_i=-b_i$ for all $i$ and $a_i=b_jb_k$ for $i,j,k=1,2,3$ in cyclic order. The $(2,2)$ entries reiterate the identity $b_1a_1=v$ and the $(3,3)$ entry gives the identity $b_2b_3b_1=v$. Finally, $a_1a_3a_2=b_2b_3b_1b_2b_3b_1=v^2$ proving the two identities that we need.
 
A negative triangle is given by taking the mirror image. The roles of $a_i,b_i$ are then reversed and we get the required analogous equations.
 \end{proof}
 \begin{example}
 In Figure \ref{FigureR}, we have $(a_1,b_1,c_1)=(1,u,-u)$ since these are the exponents of $x_{32},x_{23}, Z_{23}$. Similarly, $(a_2,b_2,c_2)=(u,1,-1)$ and $(a_3,b_3,c_3)=(u,1,-1)$. Putting $u=v$, this agrees with the Lemma and proof above since $a_1a_2a_3=u^2$, $b_1b_2b_3=u$ and $c_i=-b_i$ for each $i$. On the right hand triangle in Figure \ref{FigureR} we have (from \cite{IK93}) that $(a_1',b_1')=(u,1)$ which equals $(b_1,a_1)$ since the clockwise direction around one triangle is the counterclockwise direction around the other triangle. Similarly, $(a_3',b_3')=(1,u)=(b_3,a_3)$. Finally, $(a_2',b_2')=(u,1)$ which agrees since $a_1'a_2'a_3'=u^2$. However, $(a_2',b_2')\neq (b_2,a_2)=(1,u)$ which means the labels on the two ends of the edge $E_{31}$ in Figure \ref{FigureR} do not match (unless $u=1$). The green dotted line indicates the discontinuity. This is resolved by Theorem \ref{thm: punching hole} and another example is given in Figure \ref{FigureCA2}.
 \end{example}

\subsection{Conclusion of the computation}  Let $G$ be a bicolored trivalent ribbon graph and $\Lambda_G\subset J^1(\Sigma_G)$ the corresponding mesh Legendrian. Let $W=E \times \bR^{4k}$ be an even stabilization of an oriented circle bundle $S^1 \to E \to \Sigma_G$ with $e(E)=n>0$. Let $\zeta$ be a primitive $n$th root of unity.

\begin{lemma}
There is a rank 1 unitary local system $\rho:\pi_1W\to  U(1)$ whose image is generated by $\zeta$ so that $\rho$ sends the preferred generator of $\pi_1S^1$ to $\zeta^{-1}$.
\end{lemma}

\begin{proof}
There is a fiberwise $n$-fold cyclic covering of $E$, hence also of $W$. The holonomy of this covering gives the representation.
\end{proof}

We now give the proof of Theorems \ref{euler} and \ref{calculation}, which we recall for convenience.

\begin{theorem}\label{thm: punching hole}
Suppose that $\Lambda_G\subset J^1(\Sigma_G)$ is generated by a function $f:W\to \bR$ which is a bounded $C^1$ distance from the standard quadratic form. Then $e(E)= |w(G)|$ and $\tau(f,\rho)=\pm(1-\zeta^\ve)$, where $\ve$ is the sign of $w(G)$.
\end{theorem}

\begin{proof} We compute the product of the $a_i$ labels clockwise around the vertices of $G$ in two different ways. In order to perform this calculation using coefficients in the ring $R=\bZ[u,u^{-1},(1-u)^{-1}]$ we must puncture the surface to trivialize the bundle. Then we pass to $\bC^\times$ using the unique ring homomorphism $R\to \bC$ sending $u$ to $\zeta$. 

To obtain the punctured surface $\Sigma^0$, we cut $\Sigma_G$ across one of the edges of $G$. This creates two copies of Figure \ref{Fig: link}. We denote with primes $(')$ the labels on the back copy. The restriction of $W$ to the punctured surface is trivial. That the Euler number of the bundle over the unpunctured surface $\Sigma_G$ is  $n>0$ means any arc going clockwise around the hole (counterclockwise around $\partial\Sigma^0$, the boundary of the punctured surface), which used to be null homotopic, will now wrap around the fiber $n$ times. 

Such a null homotopic loop is given by first moving left to right along the back of the hole going along the top critical line $x_i'$ in Figure \ref{Fig: link}, going down to $x_j'$ along one of the trajectories indicated by $-a$ in Figure \ref{Fig: link}, moving to the right along $x_j'$, then doing the reverse on the front of the hole. However, moving left to right on the back of the hole is going counterclockwise around the punctured surface. So, $a$ in Figure \ref{Fig: link} is equal to $b'$, the counterclockwise label on that back face. The fact that this cycle wraps around the fiber $n$ times implies that the label ``$-a$'' on the back side of the cut ($-b_5$ in Figure \ref{FigureCA2}, $-b'$ in general) is equal to $-u^{-n}a$ ($-u^{-n}a_6$ in Figure \ref{FigureCA2}, $-u^{-n}a$ in general) where the exponent of $u$ is negative by our convention (from \cite{IK93}) that $u^{-1}$ is the positive direction along the fiber. Since $a'b'=v$, the product of the counterclockwise labels on the front and back is
\[
	aa'=av/b'=av/u^{-n}a=vu^n.
\]
In Figure \ref{FigureCA2}, $n=3$ and $(a,b,a',b')=(a_6,b_6,a_5,b_5)$. So $a_6a_5=vu^3$.

The Turaev torsion on $(W|_{\Sigma^0},f)$ is $\tau(f,e,\rho)=\pm(1-v)\in U(R)/\pm 1$ for some $v\in R$. By Lemma \ref{a3=v2 and b3=v}, the product of the $a_i$ labels is $\prod a_i=v^{2p+q}$ if $G$ has $p$ positive and $q$ negative vertices. This product can also be computed over the $k=\frac32(p+q)$ edges of $G$. Over each edge except the cut edge we have two arrows with labels $a_i=b_j$ and $a_j=b_i$ with product $a_ia_j=a_ib_i=v$. This means that $\prod a_i$ is $v^{k-1}$ times the product $aa'$ on the cut edge which we computed in the last paragraph to be $a a'=u^nv$. This gives the equation
\[
	v^{2p+q}=(v^{k-1})u^nv=u^nv^k
\]
or $u^n=v^{2p+q-k}$. But $2p+q-k=(p-q)/2=w(G)$. So, $u^n=v^{w(G)}$ in the ring $R$.

By Proposition \ref{prop: un=vm lemma}, we obtain $|w(G)|=n=e(E)$ and $v=u^{\ve}$ where $\ve$ is the sign of $w(G)$. Passing to $\bC^\times$, with $u$ mapping to $\zeta$, we get $\tau(f,\rho)=\pm(1-\zeta^\ve)$ as claimed.
\end{proof}

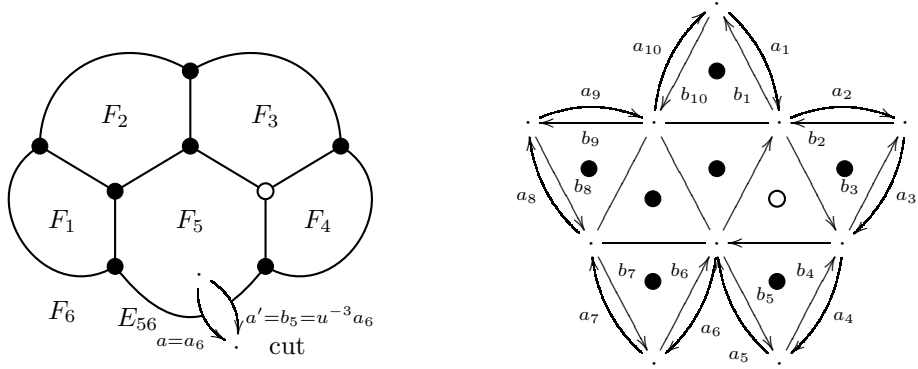
\begin{figure}[htbp]
\begin{center}
\begin{tikzpicture}
\coordinate (A1) at (-7,1.6);
\coordinate (A2) at (-5,1.6);
\coordinate (K1) at (-8,1);
\coordinate (K2) at (-6,1);
\coordinate (K3) at (-4,1);
\coordinate (L1) at (-8,0); 
\coordinate (L2) at (-6,0);
\coordinate (L3) at (-4,0);
\coordinate (M0) at (-9,-.6);
\coordinate (M1) at (-7,-.6);
\coordinate (M2) at (-5,-.6); 
\coordinate (M3) at (-3,-.6);
\coordinate (N1) at (-7,-1.6); 
\coordinate (N2) at (-5,-1.6);
\coordinate (P1) at (-8,-2.2);
\coordinate (P2) at (-6.3,-2.5);
\coordinate (P3) at (-5.7,-2.5);
\coordinate (P4) at (-4,-2.2);
\coordinate (C) at (-6.7,-2.3);
\draw (C) node{$E_{56}$};
\draw[thick] (K2).. controls (A2) and (K3)..(L3);
\draw[thick] (K2).. controls (A1) and (K1)..(L1);
\draw[thick] (N1).. controls (P1) and (M0)..(L1);
\draw[thick] (N2).. controls (P4) and (M3)..(L3);
\draw[thick] (N1).. controls (P2) and (P3)..(N2);
\foreach\x in {K2,L1,L2,L3,M1,N1,N2}\draw[fill](\x) circle[radius=3pt];
\begin{scope}
\draw[thick] (L1)--(M1)--(L2)--(K2) (L2)--(M2)--(L3);
\draw[thick]  (M1)--(N1) (M2)--(N2);
\end{scope}
\draw[fill,color=white] (M2) circle[radius=3pt];
\draw[thick] (M2) circle[radius=3pt]; 
\coordinate (X1) at (-7.7,-1);
\coordinate (X2) at (-7,.4);
\coordinate (X3) at (-5,.4);
\coordinate (X4) at (-4.3,-1);
\coordinate (X5) at (-6,-1);
\coordinate (X6) at (-7.7,-2.2);
\draw (X1) node{$F_1$};
\draw (X2) node{$F_2$};
\draw (X3) node{$F_3$};
\draw (X4) node{$F_4$};
\draw (X5) node{$F_5$};
\draw (X6) node{$F_6$};
\begin{scope}
\draw[thick] (1.8,-.7) circle[radius=3pt];
\draw[fill] (1,1) circle[radius=3pt];
\draw[fill] (.15,-.7) circle[radius=3pt];
\draw[fill] (1,-.3) circle[radius=3pt];
\draw[fill] (2.7,-.3) circle[radius=3pt];
\draw[fill] (-.7,-.3) circle[radius=3pt];
\draw[fill] (.15,-1.8) circle[radius=3pt];
\draw[fill] (1.8,-1.8) circle[radius=3pt];
\end{scope}
\coordinate (cut) at (-5,-2.2);
\coordinate(hole) at (-5.65,-2.17);
\draw[color=white,fill] (hole) circle[radius=.2cm];
\draw (cut) node{
$
\xymatrixrowsep{5pt}\xymatrixcolsep{5pt}
\xymatrix{
\cdot \ar@/_.5pc/[ddr]_(.7){a=a_6}\ar@/^.5pc/[ddr]^(.7){a'=b_5=u^{-3}a_6}\\
\ \\
& \cdot
&\text{cut}	}$};
%
\coordinate (R) at (1,-.5);
\draw(R) node{
$\xymatrixrowsep{35pt}\xymatrixcolsep{15pt}
\xymatrix{
 & & 
	&
	\cdot\ar[ld]^(.7){b_{10}}\ar@/^.5pc/[rd]^{a_1}\\ 
\cdot\ar[dr]^(.6){b_8} \ar@/^.5pc/[rr]^{a_9}&& \cdot\ar@/^.5pc/[ru]^{a_{10}}\ar[ll]^{b_9}\ar@{-}[rr] & & \cdot\ar@/^.5pc/[rr]^{a_2}\ar[lu]^(.3){b_1}\ar[rd] &&
	\cdot \ar[ll]^(.7){b_2}\ar@/^.5pc/[dl]^{a_3}\\ 
& \cdot \ar@/^.5pc/[lu]^{a_8}\ar[dr]^(.3){b_7} \ar@{-}[rr]\ar@{-}[ru]& &
	\cdot\ar@/^.5pc/[ld]^(.6){a_6}\ar[dr]^{b_5}\ar[ru]\ar@{-}[lu]&&\cdot\ar[ll]\ar[ru]^(.4){b_3}\ar@/^.5pc/[dl]^{a_4}\\
	&&\cdot\ar@/^.5pc/[lu]^{a_7}\ar[ru]^(.7){b_6} &&\cdot\ar[ru]^(.7){b_4}\ar@/^.5pc/[lu]^(.2){a_5}
	}$};
\end{tikzpicture}
\caption{The ribbon graph (left) shows that $\Sigma_G=S^2$, $\Lambda_G$ has 6 components and $w(G)=\frac12(7-1)=3$. So, $n=3$, $\varepsilon=+$. The handle slide pattern (right) has $p=7,q=1$ giving $\prod a_i=v^{15}$. The edges $E_{16},E_{26},E_{36},E_{46}$ force $a_7=b_8$, $a_9=b_{10}$, $a_1=b_2$, $a_3=b_4$. So, $a_7a_8a_9a_{10}a_1a_2a_3a_4=v^4$. Since there are 7 internal edges the internal $a_i$ multiply to $v^7$. Therefore $a_5a_6=v^{15-11}=v^4$. But $a_6=u^3b_5$. So, $v^4=a_5a_6=u^3a_5b_5=u^3v$. So, $v^3=u^3$.}
\label{FigureCA2}
\end{center}
\end{figure}

\begin{remark}
There is one subtle point about the proof of Theorem \ref{thm: punching hole}. We should cut an edge which is adjacent to two disjoint regions in the complement of $G$ in $\Sigma_G$. Then the components of the critical set of $f|_{\Sigma^0}$ will be simply connected since two spherical components have been punctured. This can be done as long as there is more than one component in the complement of $G$ in $\Sigma_G$. If there is only one component, the surface cannot be a sphere. So, it has a 2-fold connected covering $\widetilde \Sigma_G$. The pull back of $W$ to $\widetilde \Sigma_G$ will have double the Euler number and double the winding number and will have two components in the complement of the graph. So, the theorem holds on $\widetilde\Sigma_G$. So, $|2w(G)|=2e(E)$ and the Turaev torsion is $\pm(1-\zeta^\ve)$ where $\ve=2w(G)/2e(E)=w(G)/e(E)$.
\end{remark}

\begin{remark}\label{proof of JKS}
The claims made in the caption of Figure \ref{FigureCA} come from an earlier version of the proof of Theorem \ref{thm: punching hole}. Take a spanning tree $T$ in the trivalent ribbon graph $G$. Then $T$ has $p+q-1$ internal edges and $m=p+q+2$ leaves. The dual polygon of $T$ has $m$ sides which have arrows $a_i,b_i$ as shown in Figure \ref{FigureCA2}. The product of all $a_i$ is $\prod a_i=v^{2p+q}$. On the internal edges it is $v^{p+q-1}$. So, the product of the $a_i$ on the perimeter is $v^{p+1}$ which equals $v^k$ where $k=p+1$. Let $J$ be any set of $k$ edges of the dual polygon. Let $I$ be the complementary set of $m-k$ edges. Let $B=\prod_{j\in J} b_j$. Then $A=\prod_{i\in I\cup J}a_i=v^k$. So, 
\[
AB=v^kB=\prod_{i\in I} a_i\prod_{j\in J} a_jb_j=v^{k}\prod_{i\in I}a_i.
\]
So, $B=\prod_{j\in J} b_j=\prod_{i\in I}a_i$. In other words, the product of any $k$ of the $b_j$ is equal to the product of the $a_i$ for the other $m-k$ indices $i$. In short: $b^k=a^{m-k}$.
\end{remark}


\appendix
\section{Overview of torsion}\label{ap:torsion}

In this appendix we give some background on Whitehead, Reidemeister and Turaev torsion, as well as their higher analogues. We also briefly review the different ways in which it makes an appearance in symplectic and contact topology, both related to the present article and otherwise.

\subsection{Reidemeister and Turaev torsion}
 
Reidemeister torsion was first used by Reidemeister, Franz and de Rham, achieving the combinatorial classification of lens spaces \cite{R35}, \cite{F35}, \cite{dR39}. Reidemeister torsion can distinguish lens spaces which are homotopy equivalent but not homeomorphic, hence can access rather subtle topological information. It has since enjoyed many other applications, such as Milnor's disproof of the Hauptvermutung \cite{M61}, where two homeomorphic finite simplicial complexes without a common subdivision were distinguished using Reidemeister torsion.

Generally speaking, the Reidemeister torsion of a manifold contains information about the simple homotopy type of the manifold. It is closely related to the Whitehead torsion $\text{Wh}_1$, but is both more and less general. On the one hand it is sometimes defined even when Whitehead torsion is not, on the other it requires additional input. The Whitehead torsion lives in the Whitehead group $\text{Wh}_1(\pi)=K_1(\bZ[\pi])/\pm \pi$, while Reidemeister torsion is in essence a determinant $K_1(R) \to U(R)$ induced by an acyclic representation $\bZ[\pi] \to R$ for $R$ a commutative ring. We refer the interested reader to Milnor's survey \cite{M66} for further discussion of the Whitehead torsion.

We briefly recall the Morse-theoretic definition of Reidemeister torsion. Let $M$ be a closed orientable manifold and let $\rho:\pi_1M  \to U(1)$ be a rank 1 unitary local system such that the twisted cohomology $H^*(M;\bC^\rho)$ is trivial. One can consider more general representations of $\pi_1M$ but we will restrict to the unitary rank 1 case for concreteness. Consider the Thom-Smale complex $(C_*(f ;  \bZ[\pi_1M]), \partial)$. This chain complex is generated by the critical points of a Morse function $f:M \to \bR$ and its differential $\partial$ counts gradient trajectories between critical points of index difference $1$, after capping these trajectories with a fixed choice of paths to a basepoint. By applying $\rho:\pi_1M \to U(1)$ on the coefficients we obtain $(C_*(f;\bC^\rho), \partial^\rho)$, a chain complex  over $\bC$ which computes $H^*(M; \bC^\rho)$, hence is acyclic. Therefore there exists a chain contraction $\delta^\rho$ for $\partial^\rho$, i.e. a chain homotopy between the identity and zero. It follows that
 \[
 \partial^\rho+ \delta^\rho: C_{\text{odd}}(f;\bC^\rho) \to C_{\text{even}}(f;\bC^\rho)
 \]
 is an isomorphism. The determinant of $\partial^\rho+\delta^\rho$ is a nonzero complex number $r \in \bC^\times$ which is well defined up to multiplication by an element of $\pm \rho( \pi_1M) \subset U(1)$. The Reidemeister torsion is the image of $r$ in $\bC^\times / \pm \rho( \pi_1M)$. Perhaps a more familiar expression is the real number $\log | r |  \in \bR$, which is the additive form of Reidemeister torsion. In either case, Reidemeister torsion only depends on $M$ and $\rho$. See Section \ref{Section: turaev torsion} for a more thorough discussion.
 
Turaev showed that the $\rho(\pi_1M)$ ambiguity in the definition of Reidemeister torsion can be removed by choosing what he called an Euler structure, and moreover the sign ambiguity can also be removed in the presence of a homology orientation. This yields a finer invariant, which we call Turaev torsion, and which exists in various flavors depending on how much of the ambiguity one is able to remove. Turaev torsion has some remarkable applications to knot theory and low-dimensional topology \cite{T86}, perhaps the most striking of which is to give a combinatorial interpretation of the Seiberg-Witten invariants of 3-manifolds \cite{MT96}, \cite{T98}. 
 
 
\subsection{Higher Reidemeister torsion} 
 Consider a fibre bundle of closed, orientable and connected manifolds $F \to W \to B$ and a rank 1 unitary local system $\rho : \pi_1W \to U(1)$. Assume that the cohomology of the fibre $H^*(F;\bC^\rho)$ twisted by the pullback of $\rho$ to $\pi_1F$ is trivial. More generally, it suffices to assume that $\pi_1B$ acts trivially on $H^*(F;\bC^\rho)$. The higher Reidemeister torsion of $W$ with respect to $\rho$ is a collection of characteristic classes $r_{2k} \in H^{2k}(B;\bR)$. The existence of higher Reidemeister torsion was first suggested by Wagoner \cite{W76} and established in Klein's PhD thesis \cite{K89}. In joint work of Klein with the second author other descriptions of higher Reidemeister torsion were developed \cite{IK93a}, \cite{IK93}, which made it more amenable for computation. 
 
 The higher Reidemeister torsion classes contain information about the bundle $F \to W\to B$ which is closely related to pseudo-isotopy theory. For example, consider $\cP(\ast)$, the stable pseudo-isotopy space of a point. Higher Reidemeister torsion detects $\pi_{4k-1}\cP( \ast) \otimes \bQ \simeq \bQ$, which is generated by Hatcher's exotic disk bundles \cite{I02}. 
 The relation between higher Reidemeister torsion and higher Whitehead torsion is analogous to that between Reidemeister torsion and Whitehead torsion. For even-dimensional fibres there is also a close relation with the Miller-Morita-Mumford classes \cite{I04}.

In terms of Morse theory, pick a function $f:W \to \bR$ and view it as a family of functions $f_b$ on the fibres, parametrized by $b \in B$. The family of Thom-Smale complexes  $C_*(f_b;\bZ[\pi_1W])$ yields a map from $B$ to $\text{Wh}(\bZ[\pi_1W],\pi_1W)$, the geometric realization of the Whitehead category of based free chain complexes over $\bZ[\pi_1W]$. After applying $\rho:\pi_1 W \to U(1)$ we get a map $B \to \text{Wh}(\bC,U(1))$, which by the condition on $\rho$ lifts to $\text{Wh}^h(\bC,U(1))$, the geometric realization of the full subcategory of $\text{Wh}(\bC,U(1))$ consisting of based free acyclic chain complexes over $\bC$. In $\text{Wh}^h(\bC,U(1))$ there exist classes $r_{2k} \in H^{2k}(\text{Wh}^h(\bC,U(1)); \bR)$, coming from the continuous cohomology of $GL(\bC)$, which we can pull back to $B$ to obtain classes in $H^{2k}(B; \bR)$. 

For all this to make sense we must require that the family $f$ is generalized Morse, which means that each function $f_b$ in the family only has Morse or Morse birth/death singularities. Given such a family, one can define cohomology classes in $H^{2k}(B;\bR)$ as above, but in general these depend on the choice of $f$. To remove the ambiguity and obtain the actual higher Reidemeister torsion classes one performs this construction for $f$ a fibrewise framed function, the space of which is nonempty and contractible \cite{I87}, \cite{EM12}. Fibrewise framed functions are generalized Morse families such that the negative eigenspaces at each critical point of $f_b$ are equipped with framings, which vary continuously with $b \in B$ and are suitably compatible at birth/death points. 

One can understand the significance of the framing condition as follows. Given a generalized Morse family $f:W \to \bR$, the collection of negative eigenspaces at the fibrewise critical points can be assembled into a $\widetilde{KO}$ class which may be nontrivial and contain information about $W$. This information is lost when we pass to the family of Thom-Smale complexes $C_*(f_b;\bZ[\pi_1W])$. However, for a fibrewise framed function this  $\widetilde{KO}$ class is trivial and all the information is concentrated in the family of Thom-Smale complexes $C_*(f_b;\bZ[\pi_1W])$, so the resulting map $B \to \text{Wh}^h(\bC,U(1))$ is the ``correct'' one.

There exist several variations of the definition of higher Reidemeister torsion. For example, one can consider higher rank unitary local systems, weaken the condition on the action of $\pi_1B$ on $H^*(F;\bC^\rho)$ or consider a relative version of the construction. The interested reader is referred to the book \cite{I02}. Moreover, in the presence of an almost complex structure on the fibres one can define a complex torsion which also takes the almost complex structure into account \cite{I05}. Higher Reidemeister torsion is very closely related to the higher analytic torsion of Bismut and Lott \cite{BL95}, just like Reidemeister torsion is very closely related to the analytic torsion of Ray and Singer \cite{RS71}. There is also a very close connection with the smooth torsion of Dwyer, Weiss and Williams \cite{DWW03}. 
 
 
\subsection{Pictures for $K_3$}\label{pictures for K3}
 
We now restrict our discussion of higher Reidemeister torsion to the case of a circle bundle over the sphere $S^1 \to E \to S^2$. This case is the most relevant to the present article and is simple enough to be understood explicitly. The calculation was carried out in joint work of the second author with Klein \cite{IK93a}, \cite{IK93}. Fix an orientation of the fibre. A function $f:E \to \bR$ is a PGMF (positive generalized Morse function) if whenever it is restricted to a fibre $F \simeq S^1$ it only has Morse (quadratic) or Morse birth/death (cubic) singularities, and moreover if at a cubic singularity $x^3$ the direction in which the function is increasing agrees with the specified orientation of $F$. In particular, PGMFs are fibrewise framed. It was proved in \cite{IK93} that the the space of PGMFs is contractible. 

A PGMF $f$  on $E$ together with a nontrivial rank 1 unitary local system $\rho : \pi_1E \to U(1)$ produces an element of $K_3(\bC)$. The Borel regulator $b:K_3(\bC) \to \bR$ applied to this element yields a number, which is an invariant of $E$ and $\rho$. This number is secretly a cohomology class in $H^2(S^2; \bR) \simeq \bR$, namely the higher Reidemeister torsion class $r_1$. Explicitly, it was computed in  \cite{IK93} that
\[ r_1 =n \, \text{Im}\Big( \,  \sum_{k=1}^\infty \frac{\zeta^k}{k^2} \, \Big), \]
 where $\zeta^{-1}$ is the image by $\rho$ of the class of the oriented fibre $S^1$ in $\pi_1E$. In particular, if we know $\zeta$ then we can recover the Euler number $n=e(E)$ from $r_1$. However, more relevant to us than the result of the calculation is the method which was used to obtain it, which we now briefly discuss.

The $K_3(\bC)$ element is obtained from the PGMF $f$ by considering its locus of handle slides. These handle slides occur along connecting trajectories between critical points of index difference $0$. The algebraic effect of a handle slide on the Thom-Smale complex is essentially that of an elementary row/column matrix operation. The locus of handle slides has codimension 1 and forms an immersed graph on $S^2$, each edge of which is decorated by an element of $\bC$ determined by $\rho$. The vertices of this graph come in two kinds. One corresponds to Steinberg relations between the handle slides. The other corresponds to exchange points, which are isolated bifurcations of handle-slides corresponding to connecting trajectories of index difference $-1.$ Algebraically, the effect of an exchange point on the Thom-Smale complex is to exchange row operations for column operations. From this decorated graph one can formally extract a picture for $K_3(\bC)$ in the sense of \cite{K}. The computation of the higher Reidemeister torsion is completed using the formulas given in \cite{IK93} for the evaluation of the Borel regulator on a picture for $K_3(\bC)$. The result is independent of the choice of the PGMF.

To explicitly compute the higher Reidemeister torsion of $E$ with respect to $\rho$ it suffices to perform the above calculation for a single well chosen PGMF. For each circle bundle $S^1 \to E \to S^2$, the chosen function $f:E \to \bR$ in \cite{IK93} was such that it generates a mesh Legendrian $\Lambda \subset J^1(S^2)$. In fact $\Lambda=\Lambda_G$ for a certain bicolored trivalent ribbon graph $G$, all of whose vertices have positive labels. This is precisely the kind of $\Lambda_G$ we use to produce our examples of $\Lambda_+$ in Corollary \ref{corollarypairs}.  The picture of handle slides for these $f$ was described in \cite{IK93} and is shown in Figure \ref{FigureR}. Our figure differs slightly from that in \cite{IK93} in that the exchange points $Z_{ij}$ are moved to the corners of the triangles. 

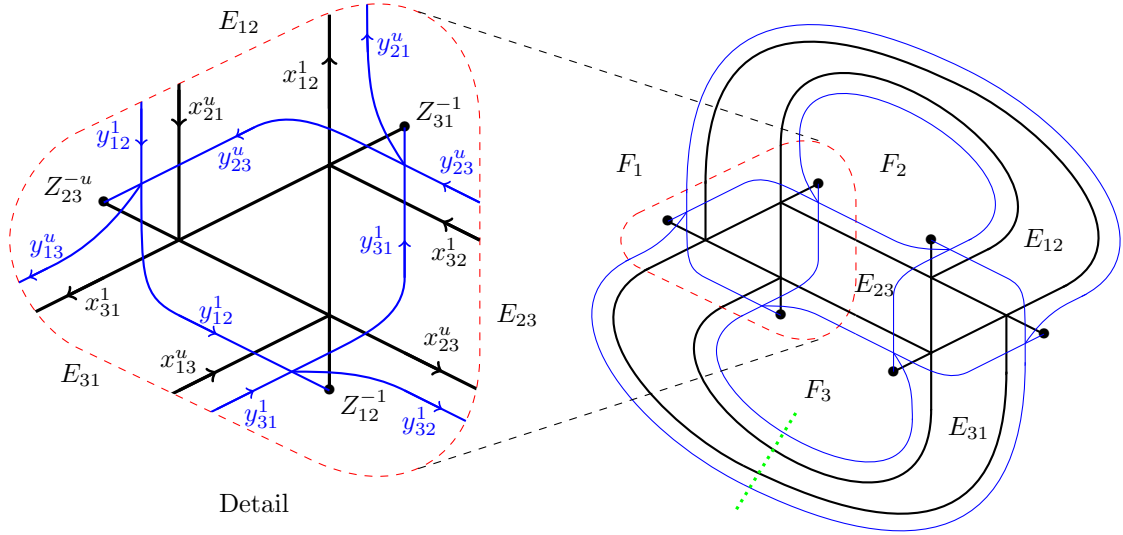
\begin{figure}[htbp]
\begin{center}
\begin{tikzpicture}
\clip (-9.5,-4) rectangle (6,3.5);
\draw (-6,-3.5) node{Detail};
\begin{scope}[xshift=-7cm]
\draw (.8,2.9) node{$E_{12}$};
\draw (4.5,-1) node{$E_{23}$};
\draw (-1.3,-1.8) node{$E_{31}$};
\draw[thick,color=blue,<-] (.75,1.375)--(1,1.5); 
\draw[thick,color=blue] (.75,1.375)node[below]{$y_{23}^u$}; 
\draw[color=blue, thick,->] (4,0.5)--(3.5,.75);
\draw[color=blue](3.7,.75)node[above]{$y_{23}^u$};
\draw[very thick,->] (4,0)--(3.5,.25);
\draw (3.6,.15) node[below]{$x_{32}^1$} ;
\draw[very thick,->] (3,-1.5)--(3.5,-1.75);
\draw (3.5,-1.6) node[above]{$x_{23}^u$};
\draw[very thick,->] (0,-2)--(.5,-1.75) ;
\draw (0,-1.9) node[above]{$x_{13}^u$};
\draw[thick,color=blue,->] (.5,-2.25)--(1,-2);
\draw[color=blue] (1.1,-2) node[below]{$y_{31}^1$};
\draw[thick,color=blue,->] (3,-.5)--(3,0) node[left]{$y_{31}^1$};
\draw[thick,color=blue,->] (-.5,1.75)--(-.5,1.25);\draw[thick,color=blue,->] (3,-2)--(3.5,-2.25) ;
\draw[color=blue] (3.2,-2.1)node[below]{$y_{32}^1$};
\draw[thick,color=blue,->] (0,-1)--(.5,-1.25)node[above]{$y_{12}^1$};
\draw[thick,color=blue,<-] (-2,-.5)--(-1.75,-.375) node[above]{$y_{13}^u$};
\draw[thick,color=blue] (-.5,1.4)node[left]{$y_{12}^1$};
\draw[very thick, ->] (2,1.75)--(2,2.45) ;
\draw (2,2.2) node[left]{$x_{12}^1$};
\draw[very thick, ->] (0,2)--(0,1.5);
\draw (0,1.75) node[right]{$x_{21}^u$};
\draw[thick,color=blue,->] (2.51,2.5)--(2.5,2.75);
\draw[thick,color=blue] (2.5,2.6) node[right]{$y_{21}^u$};
\draw[very thick,<-] (-1.5,-.75)--(-1,-.5)node[below]{$x_{31}^1$};
\draw[dashed, color=red, rounded corners=1cm] 
(-2.25,1){[rounded corners=2.5cm]--(4,4)--(4,-4)}--(-2.25,-1)--cycle;
\clip[rounded corners=1cm] 
(-2.25,1){[rounded corners=2.5cm]--(4,4)--(4,-4)}--(-2.25,-1)--cycle;
\draw[very thick] 
(4,-2)--(-1,.5)node{$\bullet$} (-1,.7)node[left]{$Z_{23}^{-u}$}
 (-2,-1)--(3,1.5)node{$\bullet$} (3,1.7)node[right]{$Z_{31}^{-1}$}
 (2,3)--(2,-2)node{$\bullet$} (2,-2.2)node[right]{$Z_{12}^{-1}$};
\draw[very thick] (0,0)--(0,3) (2,-1)--(-2,-3) (2,1)--(4,0);
\draw[thick,color=blue] (-1,.5)--(1,1.5)..controls (1.5,1.75) and (2,1.5)..(2.5,1.25)--(4,.5); 
\draw[thick,color=blue] (-.5,.75)..controls (-1,0) and (-1.5,-.25)..(-1.7,-.35)--(-2.7,-.85); 
\draw[thick,color=blue] (3,1.5)--(3,-.5)..controls (3,-1) and (2.5,-1.25)..(2,-1.5)--(-1,-3);
\draw[thick,color=blue] (3,1)..controls (2.5,1.75) and (2.5,2)..(2.5,3.5); 
\draw[thick,color=blue] (2,-2)--(0,-1)..controls (-.5,-.75) and (-.5,-.5)..(-.5,.75)--(-.5,2.5) ;
\draw[thick,color=blue] (1.5,-1.75)..controls (2,-1.75) and (2.5,-1.75)..(3,-2)--(4,-2.5); 
\end{scope}
\begin{scope}[scale=.5]
\draw (9,0) node{$E_{12}$};
\draw (7,-5) node{$E_{31}$};
\draw (4.5,-1.2) node{$E_{23}$};
\draw (-2,2) node{$F_1$};
\draw (5,2) node{$F_2$};
\draw (3,-4) node{$F_3$};
\draw[color=blue] (-2,-.5)..controls (-4.5,-1.75) and (-2,-5)..(1,-6.5) ;
\draw[color=blue] (1,-6.5)..controls (5,-8.5) and (8.5,-8.5)..(8.5,-3.25);
\draw[thick] (-1.5,-.75)..controls (-4,-2) and (-1,-5)..(1,-6);
\draw[thick] (8,-3.5)..controls (8,-8) and (5,-8)..(1,-6);
\draw[color=blue] (1,-2).. controls (-2,-3.5) and (5.5,-8)..(5.5,-4.75);
\draw[thick] (.5,-1.75)..controls (-3,-3.5)and (6,-9.5)..(6,-4.5);
\draw[dashed,thin] (3,2.65)--(-7.7,6.1) (3,-2.65)--(-7.7,-6.1);
\draw[dashed, color=red, rounded corners=.25cm] 
(-2.25,.5){[rounded corners=1.25cm]--(4,3.5)--(4,-3.5)}--(-2.25,-.5)--cycle;
\clip (-2.25,.5)--(4,3.5)--(4,-3.5)--(-2.25,-.5)--cycle;
%
\draw[ thick] 
(4,-2)--(-1,.5)node{$\bullet$} 
 (-2,-1)--(3,1.5)node{$\bullet$} 
 (2,3)--(2,-2)node{$\bullet$} (2,-2.2);
\draw[ thick] (0,0)--(0,3) (2,-1)--(-2,-3) (2,1)--(4,0);
\draw[color=blue] (-1,.5)--(1,1.5)..controls (1.5,1.75) and (2,1.5)..(2.5,1.25)--(4,.5); 
\draw[color=blue] (-.5,.75)..controls (-1,0) and (-1.5,-.25)..(-1.7,-.35)--(-2.7,-.85); 
\draw[color=blue] (3,1.5)--(3,-.5)..controls (3,-1) and (2.5,-1.25)..(2,-1.5)--(-1,-3);
\draw[color=blue] (3,1)..controls (2.5,1.75) and (2.5,2)..(2.5,3); 
\draw[color=blue] (2,-2)--(0,-1)..controls (-.5,-.75) and (-.5,-.5)..(-.5,.75)--(-.5,2.5) ;
\draw[color=blue] (1.5,-1.75)..controls (2,-1.75) and (2.5,-1.75)..(3,-2)--(4,-2.5); 
\end{scope}
\begin{scope}[scale=.5,rotate=180,xshift=-8cm,yshift=2cm]
\draw[color=blue] (-2,-.5)..controls (-4.5,-1.75) and (-2,-5)..(1,-6.5) ;
\draw[color=blue] (1,-6.5)..controls (5,-8.5) and (8.5,-8.5)..(8.5,-3.25);
\draw[thick] (-1.5,-.75)..controls (-4,-2) and (-1,-5)..(1,-6);
\draw[thick] (8,-3.5)..controls (8,-8) and (5,-8)..(1,-6);
\draw[color=blue] (1,-2).. controls (-2,-3.5) and (5.5,-8)..(5.5,-4.75);
\draw[thick] (.5,-1.75)..controls (-3,-3.5)and (6,-9.5)..(6,-4.5);
\clip (-2.25,.5)--(4,3.5)--(4,-3.5)--(-2.25,-.5)--cycle;
\draw[ thick] 
(4,-2)--(-1,.5)node{$\bullet$} 
 (-2,-1)--(3,1.5)node{$\bullet$} 
 (2,3)--(2,-2)node{$\bullet$} (2,-2.2);
\draw[ thick] (0,0)--(0,3) (2,-1)--(-2,-3) (2,1)--(4,0);
\draw[color=blue] (-1,.5)--(1,1.5)..controls (1.5,1.75) and (2,1.5)..(2.5,1.25)--(4,.5); 
\draw[color=blue] (-.5,.75)..controls (-1,0) and (-1.5,-.25)..(-1.7,-.35)--(-2.7,-.85); 
\draw[color=blue] (3,1.5)--(3,-.5)..controls (3,-1) and (2.5,-1.25)..(2,-1.5)--(-1,-3);
\draw[color=blue] (3,1)..controls (2.5,1.75) and (2.5,2)..(2.5,3); 
\draw[color=blue] (2,-2)--(0,-1)..controls (-.5,-.75) and (-.5,-.5)..(-.5,.75)--(-.5,2.5) ;
\draw[color=blue] (1.5,-1.75)..controls (2,-1.75) and (2.5,-1.75)..(3,-2)--(4,-2.5); 
\end{scope}
\draw[very thick,color=green,dotted] (1.2,-2.3)--(.4,-3.6);
\end{tikzpicture}
\caption{The black lines indicate $k+1/k+1$ handle slides $x_{ij}^a$ giving column operations in the incidence matrices. The blue lines indicate $k/k$ handle slides $y_{ij}^b$ giving row operations. Exchange point, denoted $Z_{ij}^c$, is where the $i$th upper index critical point slides under the $j$th lower index critical point. The green dotted line indicates a discontinuity in the labeling. This is the ``cut'' that appears in the proof of Theorem \ref{thm: punching hole}. The triangles are two black vertices connected by three edges of the ribbon graph labeled $E_{ij}$.}
\label{FigureR}
\end{center}
\end{figure}

For the purposes of the present article the object of interest is the Legendrian $\Lambda$ and the generating family $f$ is only one of many possible, so the viewpoint is flipped. For fixed $E$ and $\rho$, the Legendrian Turaev torsion of $\Lambda$ is a priori a set which could contain multiple elements, corresponding to the Turaev torsions obtained from different generating families $f$ for $\Lambda$ on an even stabilization of $E$. However, for mesh Legendrians this is not the case, the Legendrian Turaev torsion is a one-element set.

For a mesh Legendrian, the Turaev torsion associated to a generating family $f$ can be extracted from a certain monodromy in the handle slides of $f$. This monodromy can be read from the $K_3(\bC)$ picture and turns out to be governed by the exchange points. Our Morse theoretic analysis shows that the $K_3(\bC)$ picture computed for the specific $f$ of \cite{IK93} is essentially the only possible, hence the Turaev torsion is completely determined by $\Lambda$. In fact, in this article we don't need the $K_3$ formalism and so we work directly with Turaev torsion instead. 


\begin{remark} Both of the Legendrians $\Lambda_{\pm}$ in each of the pairs we construct for Corollary \ref{corollarypairs} are generated by a generalized Morse family on the same circle bundle $E$. However, $\Lambda_+$ is generated by PGMF on $E$ when the circle bundle is oriented so that its Euler number is positive and $\Lambda_-$ is generated by a PGMF on $E$ when the circle bundle is oriented so that its Euler number is negative. The Legendrian Turaev torsion sees this sign. Hence Corollary \ref{turaev} can be thought of as saying that the mesh Legendrian remembers not only the circle bundle which generates it, but also its orientation. This can be summed up in the equation $e(E)=w(G)$.
\end{remark}

We conclude with a discussion of the special cases excluded in Corollary \ref{turaev}. Let $n=|e(E)|$ for $S^1 \to E \to S^2$ a circle bundle. If $n=1$, then $E=S^3$, so there are no nontrivial rank 1 unitary local systems $\rho: \pi_1E \to U(1)$. If $n=2$, then $E=\bR P^3$ and the unique nontrivial rank 1 unitary local system $\rho : \pi_1E \to U(1)$ is equal to its complex conjugate. Therefore it cannot distinguish between the two possible orientations of the bundle. Hence the limitation $|w(G)| \neq 1,2$ in Corollary \ref{turaev}. However, for $n=2$ the picture of handle slides for the function $f:\bR P^3 \to \bR$ taken with $\bZ$ coefficients produces one of the exotic elements of $K_3(\bZ)=\bZ/48$ (in fact a generator, as shown in \cite{IK93a}, \cite{IK93}). So one may hope to extend some version of our methods to this case. For $n=1$ it is even less clear what to do, although we speculate that there may be a relation with a relative $K_3$ group. 

\subsection{Pseudo-isotopy theory}

We now briefly explain the relation with pseudo-isotopy theory. We begin by recalling Smale's h-cobordism theorem \cite{S61}. Let $X$ be a closed manifold. We consider cobordisms $M$ with boundary divided into two parts $\partial M = \partial_0M \sqcup \partial_1M$, equipped with an identification $X \simeq \partial_0 M$. We say that $M$ is an h-cobordism if each inclusion $\partial_i M \subset M$ is a homotopy equivalence. If $X$ is simply connected and $\dim X \geq 5$, then Smale's theorem says that $M$ is diffeomorphic to the trivial cylinder $X \times [0,1]$ relative to $X=X \times 0 $. When $X$ is not simply connected, the set of diffeomorphism classes of h-cobordisms $M$ relative to $X$ is in bijection with $\text{Wh}_1( \pi_1 X)$, where $\text{Wh}_1(\pi)$ is a certain quotient of $K_1(\bZ[\pi])$ for $\pi$ a group. This is the s-cobordism theorem of Barden, Mazur and Stallings \cite{B64}, \cite{M63}, which also requires the dimensional assumption $\dim X \geq 5$. The group $\text{Wh}_1(\pi)$ was introduced by Whitehead \cite{W50}, and lies at the origin of simple homotopy theory. Interpreted appropriately, we can rephrase the s-cobordism theorem as a bijection
\[ \pi_0 \cH(X) \to \text{Wh}_1( \pi_1 X), \]
where $\cH(X)$ is the space of h-cobordisms on $X$. Pseudo-isotopy theory is concerned with the higher homotopy groups of $\cH(X)$. For $k \geq 1$, instead of $\pi_k\cH(X)$ we can equivalently study $\pi_{k-1}\cC(X)$. Here $\cC(X)$ is the pseudo-isotopy space of $X$, which is homotopy equivalent to the loop space $\Omega \cH(X)$.
 
The pseudo-isotopy space of $X$ is defined to be $\cC(X)=\text{Diff}(X \times [0,1], X \times 0)$. It forms a group under composition. Following Cerf, we observe that the space $\cC(X)$ is homotopy equivalent to the space $\cF(X)$ of functions $f: X \times [0,1] \to [0,1]$ which have no critical points and which agree with the projection to the second factor $\pi:X \times [0,1] \to [0,1]$ near $X \times 0$ and $X \times 1$. Informally, the homotopy equivalence $\cC(X) \to \cF(X)$ is given by the map $\varphi \mapsto \pi \circ \varphi$, though strictly speaking we should first replace $\cC(X)$ with the subspace of pseudo-isotopies which are level-preserving near top and bottom. To see that $\cC(X) \to \cF(X)$ is a homotopy equivalence it suffices to observe that it is a fibration whose fibre is the space $\text{Isot}(X)$ of isotopies of $X$. Being the space of paths in $\text{Diff}(X)$ starting at $\text{id}_X$, the space $\text{Isot}(X)$ is contractible, from which the desired conclusion follows. 
 
We now consider $\pi_1\cH(X) = \pi_0 \cC(X)$. Observe that $\cF(X)$ is a subspace of the space $\cM(X)$ of all functions $f:X \times [0,1] \to [0,1]$ which agree with $\pi$ near $X \times 0$ and $X \times 1$. Since $\cM(X)$ is convex, it is contractible. It follows that $\pi_0 \cC(X) = \pi_0  \cF(X)  = \pi_1 \big( \cM(X), \cF(X) \big)$. This observation was one of the key insights of Cerf. It allowed him to study pseudo-isotopies by means of 1-parametric Morse theory, resulting in his theorem \cite{C70} that $\pi_0 \cC(X) =0$ whenever $X$ is simply connected and $\dim X \geq 5$. 

In \cite{HW73}, Hatcher and Wagoner used this same viewpoint to study the non simply connected case. By considering the word of handle slides produced by such a 1-parametric family of functions, they defined a homomorphism 
\[ \pi_0 \cC(X) \to \text{Wh}_2(\pi_1X ), \]
where $\text{Wh}_2(\pi)$ is a certain quotient of $K_2( \bZ[\pi])$ for $\pi$ a group. They proved its surjectivity when $\dim X \geq 5$. Moreover, the kernel was identified to consist of those elements in $\pi_0 \cC(X) = \pi_1 \big( \cM(X), \cF(X) \big)$ which can be represented by a path 
\[ f: ([0,1],\partial [0,1]) \to \big(\cM(X), \cF(X)\big), \quad  t \mapsto f_t,\]
 whose Cerf diagram $\Sigma_f = \{ (t,z) : \, \text{$z$ is a critical value of $f_t$ } \} \subset \bR^2$ consists of a single `eye', as illustrated in Figure \ref{Fig: eye}. This is the simplest possible Cerf diagram, which has exactly two cusps and no self-intersections. It is also the front of the standard Legendrian unknot in $\bR^3=J^1(\bR)$.

In \cite{EG98}, Eliashberg and Gromov explored a number of connections between Lagrangians (resp. Legendrians) in cotangent bundles (resp. 1-jet spaces) and pseudo-isotopy theory. The closing remark in \cite{EG98} combines the stability of the $\text{Wh}_2$ invariant and the homotopy lifting property for generating families to deduce the following corollary of the Hatcher-Wagoner theorem. 

\begin{theorem}\label{Theorem: wh1} Suppose that $\Lambda$ is a Legendrian link in the standard contact $\bR^3=J^1(\bR)$ generated by a family $f:X \times [0,1] \to \bR$ which represents a pseudo-isotopy of $X$ with nontrivial $\text{Wh}_2$ invariant. Then $\Lambda$ is nontrivial as a Legendrian link, i.e. there is no Legendrian isotopy such that the front projection of $\Lambda$ becomes a disjoint union of `eyes'. \end{theorem}

\begin{figure}[htbp]
\begin{center}
\begin{tikzpicture}[scale=.75]
\draw[very thick] (-3,0)..controls (-2,0) and (-1.5,1.25)..(0,1.25)..controls (1.5,1.25) and (2,0)..(3,0);
\draw[very thick] (-3,0)..controls (-2,0) and (-1.5,-1.25)..(0,-1.25)..controls (1.5,-1.25) and (2,0)..(3,0);
\draw[thick,dotted] (0,-1.8)--(0,1.8);
\end{tikzpicture}
\caption{An ``eye'' (homeomorphic to $S^1$) in $\bR^2=J^0(\bR)$. Spin this around the dotted line to get a ``lens'' (homeomorphic to $S^2$) in $\bR^3=J^0(\bR^2)$. }
\label{Fig: eye}
\end{center}
\end{figure}
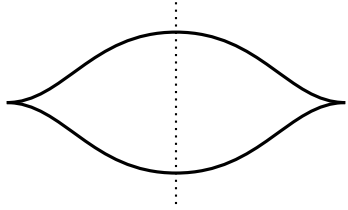

Let us add one more parameter, so we consider $\pi_2 \cH(X) = \pi_1 \cC(X)$. We now have a homomorphism 
\[ \pi_1\cC(X)  \to \text{Wh}_3( \pi_1X ), \]
where $\text{Wh}_3(\pi)$ is a certain quotient of $K_3(\bZ[\pi])$ for $\pi$ a group. This was defined in the second author's PhD thesis \cite{I79}. Moreover, in \cite{I82} it was shown that for $\dim X \geq 7$ there is an exact sequence
\[ \pi_1 \cC(X) \to \text{Wh}_3(\pi_1 X) \to  \text{Wh}^+_1(\pi_1X ; \bZ/2 \oplus \pi_2 X) \to \pi_0 \cC(X)  \to \text{Wh}_2(\pi_1 X) \to 1, \]
the middle term of which is related to nontrivial generating families for the standard Legendrian unknot in $\bR^3=J^1(\bR^1)$. We have $ \pi_1\cC(X)  = \pi_2\big( \cM(X) , \cF(X) \big)$ for the same reason as before. And as in Theorem \ref{Theorem: wh1}, we can deduce a result about the nontriviality of 2-dimensional Legendrian links in the standard contact $\bR^5=J^1(\bR^2)$ generated by a 2-parameter family of functions on $X \times [0,1]$ whose $\text{Wh}_3$ invariant is nontrivial. The conclusion is that there is no Legendrian isotopy such that the front projection becomes a disjoint union of `lenses', see Figure \ref{Fig: eye}.

 The mesh Legendrians considered in this article are also 2-dimensional and are closely related to this circle of ideas. In particular they also produce a nontrivial $K_3$ invariant. However, instead of $X \times [0,1]$ as the fibre we take $S^1$, instead of $[0,1]^2$ as our parameter space we take a closed surface $\Sigma$ and  instead of a trivial bundle we take an arbitrary circle bundle $S^1 \to E \to \Sigma$.

One can keep adding parameters. As one adds more and more parameters one also needs to increase the dimension of $X$, for example in order to make enough room for the parametrized Whitney trick to work. Hence it is better to work stably from the onset, particularly since one can use the stability theorem \cite{I88} to recover unstabilized results in a range. The key result is the stable parametrized h-cobordism theorem of Waldhausen \cite{W82}, which gives a weak homotopy equivalence
\[ A(X) \simeq Q(X_+) \times \text{Wh}^{\text{Diff}}(X). \]

Here $A(X)$ is the algebraic $K$-theory of the space $X$, $Q(X_+)=\Omega^\infty \Sigma^\infty (X_+)$ is the zero space of the suspension spectrum of $X_+$ and $\text{Wh}^{\text{Diff}}(X)$ is the (smooth) Whitehead space of $X$. This space has the property that $\Omega \text{Wh}^{\text{Diff}}(X)$ is the stable space of h-cobordisms of $X$ and $\Omega^2\text{Wh}^{\text{Diff}}(X)$ is the stable pseudo-isotopy space of $X$. It follows from Waldhausen's theorem that one should also be able to construct K-theoretically nontrivial Legendrians in $J^1(\bR^n$) for $n>2$, even though exhibiting explicit Cerf diagrams might be difficult. When $n$ is sufficiently high we expect to obtain interesting examples even when $X$ is a point. For example, $\pi_5 \text{Wh}^{\text{Diff} }( \ast )   \otimes \bQ \neq 0$.

In a different but related direction Kragh has proved that every long exact Lagrangian knot $L \subset T^*\bR^n$ admits a generating family on a trivial vector bundle which is quadratic at infinity \cite{K18}. Using this result he assigns to each such $L$ an element of $\pi_{n-1}(\cM_\infty)$, where $\cM_\infty$ is a functional space related to pseudo-isotopy theory by a fibration sequence $\cM_\infty \to G/O \to \cH_\infty$. Here $G=\lim_n G_n$ for $G_n$ the space of self homotopy equivalences of $S^n$, $O=\lim_nO_n$ is the stable orthogonal group and $\cH_\infty$ is the stable space of h-cobodisms of a point. A single nontrivial example of this construction would disprove the nearby Lagrangian conjecture. In the present paper we exploit the fact that it is much easier to generate embedded Legendrians than it is to generate embedded Lagrangians.

\subsection{Floer-theoretic torsions}
In cotangent bundles, the generating family construction provides a connection between the study of Lagrangian submanifolds and the study of finite dimensional parametrized Morse theory, which in turn is closely related to algebraic K-theory. At least part of the story extends outside of the Weinstein neighborhood of a Lagrangian submanifold of a more general symplectic manifold. Indeed, if Floer theory is the Morse theory of the action functional, then we might hope to extract not just homological information from Floer theory, but also K-theoretic information. Of course the situation is in general rather subtle and one needs to impose serious restrictions in order for the Floer theory to be well behaved, for example exactness or monotonicity. Nevertheless, there has been some progress in this direction, which we now briefly survey.  

In the exact setting, the existence of a Floer-theoretic $\text{Wh}_1$ torsion for the Lagrangian intersection problem had been hinted at by Fukaya in \cite{F95} and was established by Sullivan in \cite{S02}. Moreover, Sullivan defined a Floer-theoretic $\text{Wh}_2$ torsion for the problem of displacing a 1-parametric family of Lagrangians away from a fixed Lagrangian and gave nontrivial examples for both his $\text{Wh}_1$ and $\text{Wh}_2$ torsions. These Floer-theoretic $\text{Wh}_1$ and $\text{Wh}_2$ torsions generalize the corresponding torsions in cotangent bundles considered by Eliashberg and Gromov using generating families \cite{EG98}.  

The Floer-theoretic $\text{Wh}_1$ invariant was used by Abouzaid and Kragh in \cite{AK18} to prove that the projection to the base of any nearby Lagrangian in a cotangent bundle is a simple homotopy equivalence. Another application was found by Su\'arez \cite{S17} in her study of exact Lagrangian cobordisms. Extending Sullivan's Floer-theoretic $\text{Wh}_1$ and $\text{Wh}_2$ to define higher Whitehead torsions of higher parametric families of Lagrangians is a nontrivial open problem, even in the exact case.

Hutchings and Lee gave a definition of Reidemeister torsion in the setting of Morse-Novikov theory \cite{HL99} (which is also related to Turaev torsion, but in a different way). To get an invariant it is necessary to correct the Morse-theoretic definition by a certain zeta function which counts closed orbits. This invariant was then adapted to the Floer theory of symplectomorphisms by Lee \cite{L05a}, \cite{L05b}, where the zeta function now counts perturbed pseudo-holomorphic tori and is related to a genus 1 Gromov-Witten invariant. In a  different direction, Charette has defined a quantum Reidemeister torsion for the Biran-Cornea pearl complex of a monotone Lagrangian \cite{C17}, which bears a relation with a genus 0 Gromov-Witten invariant.

\end{document}